\documentclass[12pt]{amsart}

\usepackage{amsmath}
\usepackage[alphabetic]{amsrefs}
\usepackage{mathrsfs}
\usepackage{amssymb}
\usepackage{amsthm}
\usepackage{bbm}
\usepackage{bm}
\usepackage{dsfont}
\usepackage[margin=1in]{geometry}
\usepackage{paralist}
\usepackage{multirow}
\usepackage{array}   % for \newcolumntype macro
\usepackage{subfiles}
\usepackage{mathtools}
\usepackage[normalem]{ulem}

\usepackage{hyperref}

\newcolumntype{L}{>{$}l<{$}}

\newtheorem{theorem}{Theorem}[subsection]

\newtheorem{lemma}{Lemma}[subsection]
\newtheorem{corollary}{Corollary}[subsection]
\newtheorem{proposition}{Proposition}[subsection]
\newtheorem{remark}{Remark}[subsection]
\newtheorem{conjecture}{Conjecture}[subsection]

\theoremstyle{definition}
\newtheorem{definition}{Definition}[subsection]
\newtheorem{assumption}{Assumption}[subsection]

\def\*{\times}
\def\1{\mathbbm{1}}
\def\one{\mathbf{1}}
\def\a{\mathfrak{a}}
\def\A{\mathbb{A}}

\def\B{\mathcal{B}}

\def\back{\backslash}
\def\C{\mathbb{C}}

\def\ds{\delta_\sigma}

\def\F{\mathbb{F}}

\def\Gen{\mathcal G}
\def\GmodZ{\overline{G}}
\def\GL{\text{GL}}
\def\GSp{\text{GSp}}
\def\H{\mathscr H}
\def\Hecke{\mathcal H}

\def\I{\mathcal I}
\def\Ind{\text{Ind}}

\def\Kloos{\text{Kl}}

\def\m{\mathbf m}

\def\Mat{\text{Mat}}
\def\Ak{A_{\text{K}}}
\def\As{A_{\text{S}}}
\def\Mk{M_{\text{K}}}
\def\Nk{N_{\text{K}}}
\def\Ms{M_{\text{S}}}
\def\Ns{N_{\text{S}}}
\def\Pk{P_{\text{K}}}
\def\Ps{P_{\text{S}}}

\def\Mat{\text{Mat}}
\def\n{\texttt{n}}

\def\O{\text{O}}

\def\PW{\mathscr H}
\def\Q{\mathbb{Q}}
\def\R{\mathbb{R}}
\def\res{\text{res }}
\def\Res{\text{Res }}

\def\SO{\text{SO}_2}
\def\Sp{\text{Sp}}

\def\PGL{\text{PGL}}
\def\TmodZ{\overline{T}}
\def\W{\mathcal{W}}

\def\Z{\mathbb{Z}}

\newcommand{\trans}[1]{{}^\top#1}
\newcommand{\Lie}[1]{\mathfrak{#1}}
\def\R{\mathbb{R}}
\newcommand{\mat}[4]{{\setlength{\arraycolsep}{0.5mm}\left[
		\begin{smallmatrix}#1&#2\\#3&#4\end{smallmatrix}\right]}}

\newcommand{\rvline}{\hspace*{-\arraycolsep}\vline\hspace*{-\arraycolsep}}
\newcommand{\bigmat}[4]{\left[
	\begin{smallmatrix}
		#1
		& \rvline &#2 \\
		\hline
		#3 & \rvline &
		#4
	\end{smallmatrix}
	\right]
}
\newcommand{\block}[4]{\begin{smallmatrix}#1&#2\\#3&#4\end{smallmatrix}}
\newcommand{\diag}[4]{\mat{\block{#1}{}{}{#2}}{}{}{\block{#3}{}{}{#4}}}
\newcommand{\scal}[2]{\langle #1,#2\rangle}

\title[Kuznetsov formula]{A relative trace formula approach to the Kuznetsov formula on $\GSp_4$}
\author{F\'{e}licien Comtat}

\makeatletter
\def\l@subsection{\@tocline{2}{0pt}{2.5pc}{5pc}{}}
\def\l@subsubsection{\@tocline{2}{0pt}{3.5pc}{5pc}{}}
\makeatother

\begin{document}

\begin{abstract}
	Relative trace formulas play a central role in studying automorphic forms. 
	In this paper, we use a relative trace formula approach to derive a Kuznetsov type formula for the group $\GSp_4$.
	We focus on giving a final formula that is as explicit as possible, and we plan on returning to applications elsewhere.
\end{abstract}

\maketitle
\tableofcontents

\section{Introduction}

In this work we develop a Kuznetsov formula for the group $\GSp_4$.  
To motivate our results, we first recall the Kuznetsov formula for $\GL_2$, an identity relating spectral information about the quotient space 
$\Gamma \back \mathbb H$ (where $\Gamma$ is a congruence subgroup) to some arithmetic input.

For arbitrarily chosen nonzero integers $n$ and $m$ and any reasonable test function $h$, the spectral side involves 
the quantity 
\begin{equation}\label{BasicQuantity}
	h(t_u)a_m(u)\overline{a_n(u)},
\end{equation} 
where $u$ ranges over eigenfunctions of the Laplace operator involved in the
spectral decomposition of $L^2(\Gamma \back \mathbb H)$, $a_m(u)$ is the $m$-th Fourier coefficient of $u$, 
and $t_u$ is the corresponding spectral parameter.
More precisely, the spectrum of $L^2(\Gamma \back \mathbb H)$ can be described as the direct sum of the 
\textit{discrete spectrum} and the \textit{continuous spectrum}. 
The discrete spectrum is the direct sum of $1$-dimensional subspaces spanned by cuspidal Maa{\ss} forms 
(the \textit{cuspidal spectrum}) plus the constant function (the \textit{residual spectrum}).
The continuous spectrum is a direct integral of $1$-dimensional subspaces spanned by the Eisenstein series.
The spectral side of the Kuznetsov formula correspondingly splits as a discrete sum over Maa{\ss} forms plus a 
continuous integral over Eisenstein series.

The arithmetic-geometric side is a sum of two contributions, that may be seen as the contributions from the 
two elements of the Weyl group of $\GL_2$.
The identity contribution is given by the delta symbol $\delta_{n,m}$ times the integral of the spectral test function
$h$ against the spectral measure $\frac{t}{\pi^2} \tanh(\pi t) dt$.
For this reason, the Kuznetsov formula may be viewed as a result of quasi-orthogonality for the Fourier coefficients $a_m(\cdot)$
and $a_n(\cdot)$, provided the remaining contribution can be controlled. 
The latter consists of a sum of Kloosterman sums weighted by some integral transform of the test function $h$, involving Bessel
functions.

Applications of the Kuznetsov formula involve using known results about any of the two sides to derive information
about the other side. 
On one hand, the flexibility allowed by the choice of the test function $h$ enables one to use known bounds about the Kloosterman 
sums to study the distribution of the discrete spectrum and the size of the Fourier coefficients of Maa{\ss} forms.
On the other hand, understanding the Fourier coefficients of Maa{\ss} forms as well as the integral transform appearing on
the geometric side yields strong bounds for sums of Kloosterman sums. 

Recently, Kuznetsov formulae for $\GL_3$ have been developed by Blomer and Buttcane, with similar applications as described above.
It would thus be interesting to establish the corresponding formulae for other groups.

In the classical proof of the $\GL_2$ Kuznetsov formula, one computes the inner product of two Poincar\'e series in two
different ways, one involving the spectral decomposition of $L^2(\Gamma \back \mathbb H)$, and the other one by computing 
the Fourier coefficients of the Poincar\'e series and unfolding. This gives a ``pre-Kuznetsov formula", that one then 
proceeds to integrate against the test function $h$, obtaining on the geometric side the integral transforms of $h$ 
mentioned above.

Another approach, which one may call the relative trace formula approach to the Kuznetsov formula, builds upon the relative trace formula introduced by
Jacquet~\cite{JacquetLai}. 
In the case of $\GL_2$, the relative trace formula approach to the Kuznetsov formula is apparently based on unpublished work of Zagier, 
detailed in~\cite{Joyner}. 
This approach is developed in the adelic framework in~\cite{KL}for the congruence subgroup $\Gamma=\Gamma_1(N)$.
It proceeds by integrating an automorphic kernel 
$$K_f(x,y)=\sum_{\gamma \in \PGL_2(\Q)} f(x^{-1}\gamma y),$$ where $f$ is a suitable test function.
The spectral expansion of the kernel will then involve the quantity $\tilde{f}(t_u)u(x)\overline{u(y)}$, where $u$ ranges 
over the eigenfunctions involved in the spectral decomposition of $L^2(\Gamma(N) \back \mathbb H)$, 
$t_u$ is the spectral parameter of $u$, and $\tilde{f}$ is the \textit{spherical transform} of $f$. 
Thus integrating $K_f(x,y)$ against a suitable character on $U \times U$, where $U=\mat{1}{*}{}{1}$, one gets the
quantity~(\ref{BasicQuantity}) with $h=\tilde{f}$. 
On the other hand, using the Bruhat decomposition for $\PGL_2(\Q)$, one can decompose the integral over $U \times U$ as 
a sum over elements of the Weyl group and over diagonal matrices in $\PGL_2(\Q)$ of some orbital integrals.
In the case of the identity element, at most one diagonal matrix will have a non-zero contribution, which will turn out
to be a delta symbol times some integral transform of the function $f$. In the case of the longest element in the Weyl
group, each positive integer in $N\Z$ will have a nonzero contribution, given by a Kloosterman sum times a second kind of
integral transform of $f$. A more refined version is then obtained by taking the Mellin transform of the primitive formula 
obtained.
Note that in this approach, one gets on the geometric side some integral transforms of the function $f$, hence one has to do
some extra work to relate these to the test function $h=\tilde{f}$ appearing in the spectral side. 

A couple of remarks are in order about the choice of $f$. Firstly, the spectral expansion of the kernel involves the spectral
decomposition of $L^2(\R_{>0}\GL_2(\Q) \back \GL_2(\A))$ rather than $L^2(\Gamma \back \mathbb H)$. 
By restricting $f$ to be left and right $K_\infty$-invariant 
(where $K_\infty=\SO$), only right-$K_\infty$-invariant automorphic forms $\phi$ 
(thus corresponding to adelization of functions on the homogeneous space $\mathbb H = SL_2(\R) / K_\infty$)
will show up in the spectral expansion of the kernel, but other choices are possible.
Also one may choose the test function $f$ at unramified places so as to get a final formula that include the Hecke eigenvalues
of a fixed Hecke operator of index coprime to the level $N$.

Our plan is to implement the relative trace formula approach in the case of $\GSp_4$. 
In contrast to the case of $\GL_2$, there is more than one non-conjugate unipotent subgroups~$U$.
Choosing $U$ to be the unipotent radical of the Borel subgroup (that is the minimal parabolic subgroup)
will yield Whittaker coefficients of the automorphic forms involved (instead of the Fourier coefficients).
The Whittaker coefficients have a ``multiplicity one" property, which ensures that the global Whittaker coefficients factor
 into a product of local coefficients.
These local Whittaker coefficients can be written down in terms of local Satake parameters, which is important for applications. 
Also in contrast to the case of $\GL_2$, not every automorphic form has non-identically zero Whittaker coefficients. 
For instance, Siegel modular form give rise to automorphic forms whose Whittaker coefficients vanish identically.
Thus, only \textit{generic} automorphic forms (\textit{i.e,} with non-identically zero Whittaker coefficients) 
will survive the integration on $U \times U$ and contribute to the final formula.

Let us briefly sketch some similarities and differences with the Kuznetsov formula for $\GL_3$, which also has
rank $2$.
On the spectral side, the continuous contribution is in both cases given on the one hand by \textit{minimal Eisenstein series},
(that is, attached to the minimal parabolic subgroup), and on the other hand by Eisenstein series induced from non-minimal parabolics by Maa{\ss} forms
on $\GL_2$. However, in the case of $\GL_3$, the two non-minimal proper standard parabolic subgroup are \textit{associated},
hence by Langlands theory their Eisenstein series are essentially the same. On the other hand, for $\GSp_4$, we have 
two distinct non-associated such parabolic subgroups, giving rise to two distinct kinds of Eisenstein series.
As for as the geometric side, the Weyl group of $\GL_3$ has six elements, while the Weyl goup of $\GSp_4$ has eight.
However, it seems interesting to notice that in both case, only the identity element and the longest three elements
in the Weyl group have a non-zero contribution, thus eventually giving in total four distinct terms.

Finally, let us mention that Siu Hang Man has independently derived a Kuznetsov formula for $\Sp_4$ using the more classical technique
of computing the inner product of Poincar\'e series, and has derived some applications towards the Ramanujan Conjecture~\cite{SHMKuznetsov}.
However, because the techniques employed and the final formulae differ, the author believes that our works are complementary rather than redundant.
Indeed, the flexibility offered by the adelic framework enables us to treat the test function differently at each place.
As a result, by choosing an appropriate test function at finite places, our formula might incorporate the eigenvalues of an arbitrary Hecke operator.
Furthermore, at the Archimedean place, we make use of two deep theorems of functional analysis on real reductive groups (namely Harish-Chandra inversion theorem
and Wallach's Whittaker inversion theorem) in order to produce an arbitrary Paley-Wiener test function on the spectral side, and relate it explicitly to its transform
appearing on the arithmetic side.
As a last point, working with $\GSp_4$ instead of $\Sp_4$ enables us to work with a central character.

The primary goal of this work is to write down the main Kuznetsov formula (in Theorem~\ref{MainTheorem}) as explicitly as possible. 
Applications of Kuznetsov formulae classically include counts of the cuspidal spectrum,
bounds towards the Ramanujan conjecture and bounds for the number of automorphic forms violating the Ramanujan conjecture,
large sieve inequalities involving the Hecke eigenvalues,
estimates for moments of $L$-functions,
and distribution of the Whittaker coefficients. 
Some of these have been developed in~\cite{SHMKuznetsov}, and we hope to work out further applications in future work. 

\addtocontents{toc}{\protect\setcounter{tocdepth}{1}}

\subsection*{Acknowledgments} I wish to thanks Abhishek Saha for suggesting me this problem and for helpful comments, 
Ralf Schmidt for communicating a proof of Proposition~\ref{MustBePrincipalSeries} to me,
and Jack Buttcane for being in touch about the interchange of integrals conjecture.  
\addtocontents{toc}{\protect\setcounter{tocdepth}{2}}

\section{Generalities}
\begin{definition}\label{classical}
	The general symplectic group of degree 2 over a field $\F$ is the group
	$$\GSp_4(\F)=\{M \in \Mat_4(\F) : \exists \mu \in \F^\* , \trans{M}JM=\mu J\},$$
	where 
	$J=\mat{}{I_2}{-I_2}{}$ and $\trans{M}$ denotes the transpose matrix of $M$.
\end{definition}
Note that some authors use different realizations of $\GSp_4$, for instance the realization used in~\cite{RS2} 
(to which we refer, along with~\cite{RS1}, for expository details) is conjugated
in $\GL_4$ to ours by the matrix $\mat{\block{}{1}{1}{}}{}{}{\block{1}{}{}{1}}$.
From now on we denote $G=\GSp$.
The {\bf Cartan involution} of $G$ is given by $\theta(M)=J^{-1}MJ=\trans{M^{-1}}$.
The scalar $\mu=\mu(M)$ in the definition is called the {\bf multiplier system}.
The centre of $G$ consists in all the invertible scalar matrices.
We fix a maximal torus in $G(\F)$
$$T(\F)=\left\{\mat{\block{x}{}{}{y}}{}{}{\block{tx^{-1}}{}{}{ty^{-1}}} : x,y,t \in \F^\* \right\}.$$
\begin{definition}
	The symplectic group of degree 2 over a field $\F$ is the group
	$$\Sp_4(\F)=\{M \in G(\F) : \mu(M)=1\}.$$
\end{definition}
The centre of $\Sp_4$ is $\{\pm 1\}$, and a maximal torus in $\Sp_4(\F)$ is given by
$$A(\F)=T(\F) \cap \Sp_4(\F) =\left\{\mat{\block{x}{}{}{y}}{}{}{\block{x^{-1}}{}{}{y^{-1}}} : x,y \in \F^\* \right\}.$$
\subsection{Weyl group}
Let $N(T)$ be the normalizer of $T$.
The Weyl group $\Omega=N(T)/T$ is generated by (the images of)
$s_1=\mat{\block{}{1}{1}{}}{}{}{\block{}{1}{1}{}}$ and 
$s_2=\left[ \begin{smallmatrix}&&1&\\ &1&&\\-1&&&\\&&&1\end{smallmatrix}\right]$, and consists of the (images of the) eight elements 
$$1, s_1, s_2, 
s_1s_2=\left[ \begin{smallmatrix}&1&&\\ &&1&\\&&&1\\-1&&&\end{smallmatrix}\right],
s_2s_1=\left[ \begin{smallmatrix}&&&1\\1&&&\\&-1&&\\&&1&\end{smallmatrix}\right],$$
$$s_1s_2s_1=\left[ \begin{smallmatrix}1&&&\\ &&&1\\&&1&\\&-1&&\end{smallmatrix}\right],
s_2s_1s_2=\mat{}{\block{}{1}{1}{}}{\block{}{-1}{-1}{}}{},(s_1s_2)^2=J.$$
\subsection{Compact subgroups}
A choice of maximal compact subgroup of $G(\R)$ is given by the set $K_0$ of fixed points of the Cartan involution $\theta$.
An easy computation shows
$$K_0=K_\infty \sqcup \diag{1}{1}{-1}{-1}K_\infty,$$
where
$$K_\infty = \left\{ \mat{A}{B}{-B}{A}:
A\trans{A}+B\trans{B}=1,
A\trans{B}=B\trans{A}. \right\}.$$
The condition 
$$\begin{cases}	
	A\trans{A}+B\trans{B}=1\\
	A\trans{B}=B\trans{A}\end{cases}$$ is equivalent to $A+iB \in U(2)$, hence $K_\infty$ is isomorphic to $U(2)$.

For each prime $p$ we also consider a (compact open) congruence subgroup $\Gamma_p \subset G(\Z_p)$,
with the properties that $\Gamma_p=G(\Z_p)$ for all but finitely many $p$ and the multiplier system $\mu$ is
surjective from $\Gamma_p$ to $\Z_p^\*$ for all $p$. 
This implies we have the 
{\bf strong approximation}: setting $\Gamma=K_\infty \prod_p \Gamma_p$, we have
$$G(\A)=G(\R)^\circ G(\Q) \Gamma,$$
where $G(\R)^\circ$ is the connected component of the identity and $\A$ is the ring of ad\`eles of $\Q$.
Moreover we have the {\bf Iwasawa decomposition}
$G(\A)=P(\A)K$ for all standard parabolic subgroup $P$, where $K=K_\infty \times \prod_p G(\Z_p)$.
\subsection{Parabolic subgroups}
Parabolic subgroups are subgroups such that $G/P$ is a projective variety.
Given a minimal parabolic subgroup $P_0$,
{\bf standard parabolic subgroups} (with respect to $P_0$) are those parabolic subgroups that contain $P_0$.
If $P$ is a standard parabolic subgroup defined over $\Q$,
the {\bf Levi decomposition} of $P$ is a semidirect product $P=N_PM_P$ where $M_P$ is a reductive subgroup and
$N_P$ is a normal unipotent subgroup. 
We give here the three non-trivial standard (with respect to our choice of $P_0=B$) parabolic subroups and their Levi decomposition.
\subsubsection{Borel subgroup} \label{Borelsbgp}
The {\bf Borel subgroup} is the minimal standard parabolic subgroup.
It is given by 
$$B=\left[ \begin{smallmatrix}*&&*&*\\ *&*&*&*\\&&*&*\\&&&*\end{smallmatrix}\right] \cap \GSp_4$$
and has Levi decomposition $B=U T=TU$, where
$$U=\left\{\bigmat{\block{1}{}{x}{1}}{\block{c}{a-cx}{a}{b}}{}{\block{1}{-x}{}{1}},a,b,c,x \in \F\right\}.$$
We have the {\bf Bruhat decomposition}
\begin{align*}
	G=\coprod_{\sigma \in \Omega} B \sigma B=\coprod_{\sigma \in \Omega} UT \sigma U.
\end{align*}
We write the Iwasawa decomposition for $\Sp_4(\R)$ as follows.
\begin{definition}
	For every $g \in \Sp_4(\R)$ there is a unique element $A(g) \in \Lie{a}$, such that 
	$$g \in U \exp(A(g))K_\infty.$$
\end{definition} 
Later on we shall need the following technical lemma.
\begin{lemma}\label{A(Ju)}
	With the notation $$u(x,a,b,c)=\left[
	\begin{smallmatrix}
		1	& 		&	c	& 	a-cx	\\
		x 	& 	1	& a	& b			\\
		& 		&	 1	& 	-x		\\
		&		&		&	1			\\
	\end{smallmatrix}
	\right]$$
	we have for all $a,b,c,x \in \R$
	$$A(Ju(x,a,b,c))=A(Ju(-x,a,-b,-c)).$$
\end{lemma}
\begin{proof}
	To alleviate notations, set $u=u(x,a,b,c)$.
	By definition we have $u=u_1A(Ju)k$ for some $u_1 \in U(\R)$ and $\trans{k}=k^{-1}$.
	Then $$	u \trans{u} =J\trans{u}^{-1}u^{-1}J^{-1}\\
		=\trans{u_1}^{-1}\exp(-2A(Ju))u_1^{-1}.$$
The matrices $u_1$ and $A(Ju)$ can then be recovered from the entries of $u \trans{u}$ by an elementary calculation (see Lemma~\ref{uoppu}) below).
In particular, we find $$A(Ju)=-\frac12\diag{a_1-a_2}{a_2}{a_2-a_1}{-a_2},$$ where
$a_1 =\log(1+x^2+a^2+b^2)$ and $a_2=\log((a(a-cx)+1-bc)^2+(x(a-cx)-b-c)^2)$.
\end{proof}
\subsubsection{Klingen subgroup}
The {\bf Klingen subgroup} is
$$\Pk=\left[ \begin{smallmatrix}*&&*&*\\ *&*&*&*\\ *&&*&*\\&&&*\end{smallmatrix}\right] \cap \GSp_4.$$
It has Levi decomposition $\Pk=\Nk\Mk$, where
\begin{align*}
	\Nk=\left\{\mat{\block{1}{}{x}{1}}{\block{}{y}{y}{z}}{}{\block{1}{-x}{}{1}},x,y,z \in \F\right\}
	\text{ and } 
	\Mk=\left\{\left[ \begin{smallmatrix}a&&b&\\ &t&&\\ c&&d&\\&&&t^{-1}\delta\end{smallmatrix}\right], t \in \F^\*, 
	\delta = \det\left(\mat{a}{b}{c}{d}\right) \neq 0\right\}.
\end{align*}
The centre of $\Mk$ is
$\Ak=\left\{\left[ \begin{smallmatrix}u&&&\\&t&&\\&&u&\\&&&t^{-1}u^2\end{smallmatrix}\right], t,u
\in \F^\* \right\}.$
\subsubsection{Siegel subgroup}
The {\bf Siegel subgroup} is
$$\Ps=\left[ \begin{smallmatrix}*&*&*&*\\ *&*&*&*\\&&*&*\\&&*&*\end{smallmatrix}\right] \cap \GSp_4.$$
It has Levi decomposition $\Ps=\Ns\Ms$, where
$$\Ns=\left\{\mat{\block{1}{}{}{1}}{\block{x}{y}{y}{z}}{}{\block{1}{}{}{1}},x,y,z \in \F\right\}
\text{ and } 
\Ms=\left\{ \mat{A}{}{}{t\trans{A^{-1}}}, A \in \GL_2(\F), t \in \F^\* \right\}.$$
The centre of $\Ms$ is
$\As=\left\{\left[ \begin{smallmatrix}u&&&\\&u&&\\&&tu^{-1}&\\&&&tu^{-1}\end{smallmatrix}\right], t,u \in \F^\* \right\}.$
\subsection{Lie algebras and characters}
Following Arthur~\cite{ArthurIntro}, we parametrize the characters of the Levi component of the parabolic subgroups by the dual of the Lie algebras of their centre.
We fix  $|\cdot|_\A=\prod_v|\cdot|_v$ the {\bf standard adelic absolute value}.
Let $P=M_PN_P$ be a standard parabolic subgroup, and $A_P$ be the centre of $M_P$.
Then there is a surjective homomorphism 
$$H_P: M_P(\A) \to \text{Hom}_\Z(X(M_P),\R)$$ defined by 
\begin{equation}\label{DefinitionOfHp}
	\left(H_P(m)\right)(\chi)=\log(|\chi(m)|_\A),
\end{equation}
where we write $X(H)$ for the group of homomorphisms (of \textit{algebraic groups}) $H \to \GL_1$ that are defined over $\Q$. 
On the other hand, we may identify the vector space $\text{Hom}_\Z(X(M_P),\R)$ with the Lie algebra 
$\Lie{a}_P \oplus \Lie{z}$ of $A_P(\R)$ (where $\Lie{a}_P$ is the Lie algebra of $A_P(\R) \cap \Sp_4(\R)$ and $\Lie{z}$ is the Lie algebra of the centre).
If $\nu \in \Lie{a}_P^*(\C) \oplus \Lie{z}^*(\C)$, then the map
\begin{equation}\label{DefOfChar}
	m \mapsto \exp(\nu(H_P(m))
\end{equation}
defines a character of $M_P(\A)$. Moreover characters of $Z(\A)$
correspond to $\Lie{z}^*(\C)$ while characters that are trivial on $Z(\A)$ correspond to $\Lie{a}_P^*(\C)$.

\section{Representations}
	\subsection{Generic representations}

		\subsubsection{Generic characters}  \label{GenericCharacters}
A character $\psi$ of $U(\Q) \back U(\A)$ is said to be {\bf generic} if its differential is non-trivial 
on each of the eigenspaces $\Lie{n}_\alpha$ corresponding to simple roots $\alpha$.
Explicitly, If $\theta$ is the standard additive character of $\A/\Q$ and $\m=(m_1, m_2) \in (\Q^\*)^2$, 
generic characters of $U(\A)$ are given by 
\begin{equation}\label{genchar}
	\psi_\m\left(\left[\begin{smallmatrix}
		1	& 		&	c	& 	a-cx	\\
		x 	& 	1	& a	& b			\\
		& 		&	 1	& 	-x		\\
		&		&		&	1			\\
	\end{smallmatrix}\right]\right)=\theta(m_1x+m_2c).
\end{equation}
Note that all generic characters may be obtained from each other by conjugation by an element of $T/Z$,
as we have	for all $u \in U(\A)$
\begin{equation}\label{gencharconj}
	\psi_\m\left(u \right)=\psi_{\mathbf{1}}\left( t_\m^{-1} u t_\m\right),
\end{equation}
where 
\begin{equation}\label{tm}
	t_\m=\diag{m_1}{1}{m_1m_2}{m_1^2m_2}.
\end{equation}

\subsubsection{Whittaker coefficients and generic representations}
If $\phi$ is any automorphic form on $G(\A)$ and $\psi$ a generic character, the {\bf $\psi$-Whittaker coefficient}
of $\phi$ is by definition
\begin{equation} \label{Whittaker}
	\W_{\psi}(\phi) (g) = \int_{U(\Q) \backslash U(\A)} \phi(ug) \psi(u)^{-1} du.
\end{equation}
$\phi$ is called $\psi$-generic if $\W_\psi$ is not identically zero as a function of $g$.
Changing variable and using the left-$G(\Q)$-invariance of $\phi$, note that we have
$$\W_{\psi_{\m}}(\phi) (g)=\frac1{|m_1^4m_2^3|}\W_{\psi_{\mathbf 1}}(\phi) (t_{\m}^{-1}g)$$
In particular, $\phi$ is $\psi$-generic for some generic character $\psi$ if and only if it is $\psi$-generic for any
generic character $\psi$, henceforth we shall just say $\phi$ is {\bf generic}.
An irreducible automorphic  representation $(\pi, V_\pi)$ is called generic if $V_\pi$ contains a generic automorphic form
$\phi$. 
Equivalently, every automorphic form in the space of a generic irreducible automorphic representation $\pi$ is generic,
since otherwise the kernel of the map $\phi \mapsto W_\psi(\phi)$ would be lead to a non-trivial invariant subspace of $\pi$,
contradicting the irreducibility of $\pi$.
Since $U$ may as well be viewed as the unipotent part of the minimal parabolic subgroup of $\Sp_4$,
we can define the Whittaker coefficients of automorphic forms $\phi$ on $\Sp_4$ in the exact same way as~(\ref{Whittaker}), except the argument is restricted 
to $\Sp_4(\A)$. This gives a similar notion of generic automorphic forms and generic representations for $\Sp_4$.
Later on, we shall restrict automorphic forms on $\GSp_4$ to $\Sp_4$. 
Let us briefly explain the corresponding operations on automorphic representations.
\begin{definition}\label{DefOfRes}
	Let $(\pi,V_\pi)$ be an automorphic representation of $\GSp_4(\A)$ realized by right translation on a subspace of $L^2(G(\Q)Z(\R) \backslash G(\A))$.
	We define a representation $\res \pi$ of $\Sp_4(\A)$ as the action of $\Sp_4(\A)$ on $\left\{ \phi_{|\Sp_4(\A)}: \phi \in V_\pi\right\}$. 
	It is a quotient of the restriction $\Res \pi=\pi_{|\Sp_4(\A)}$.  
\end{definition}
The representation $\res \pi$ does not have finite length in general. However, the following shall be useful later on.
\begin{lemma}\label{genericrestriction}
	Let $\pi$ be an irreducible automorphic representation $\pi$ of $G(\A)$ that occurs discretely in $L^2(G(\Q)Z(\R) \backslash G(\A))$. 
	Then $\pi$ is generic if and only if $\res \pi$ has a generic constituent. 
\end{lemma}
\begin{proof}
	Fix a generic character $\psi$. 
	Note that for any automorphic form $\psi$ on $G(\A)$ we have 
	$\W_{\psi}(\phi_{|\Sp_4})=\left(\W_{\psi}(\phi)\right)_{|\Sp_4}$.
	From this, it is clear that if $\res \pi$ has a generic constituent then $\pi$ is generic. 
	Let us show the converse.
	Assume no constituent of $\res \pi$ is generic, so for all $\phi \in V_\pi$,
	$$\W_\psi(\phi)_{|\Sp_4(\A)}=0.$$
	Let $\phi \in \pi$ and $g \in G(\A)$. Then $\pi(g)\phi\in V_\pi$ hence 
	$$\W_{\psi}(\phi)(g)=\W_{\psi}(\pi(g)\phi)(1)=0.$$
	Thus $\pi$ is not generic.
\end{proof}
We now prove a similar lemma for the restriction of non-cuspidal representations.
\begin{lemma}\label{cuspidalrestriction}
	Let $\pi$ be an irreducible automorphic representation $\pi$ of $G(\A)$ that occurs discretely in $L^2(G(\Q)Z(\R) \backslash G(\A))$. 
	Then $\pi$ is non-cuspidal if and only if $\res \pi$ has no cuspidal constituent. 
\end{lemma}
\begin{proof}
	Recall $\pi$ is cuspidal if the constant term
	$$C_\phi(g)=\int_{U(\Q) \back U(\A)} \phi(ug)du$$
	of some (equivalently, any, since $\pi$ is irreducible) function $\phi$ in the space of $\pi$ vanishes identically.
	The exact same proof as Lemma~\ref{genericrestriction}, replacing the generic character $\psi$ by $1$, shows that $\pi$ is 
	non-cuspidal if and only if $\res \pi$ has a non-cuspidal component.
	However, we want to show that if $\pi$ is non-cuspidal, then $\res \pi$ has no cuspidal component.
	So suppose that $\res \pi$ has a cuspidal component.
	This means there is $\phi \in V_\pi$ such that $(C_\phi)_{|\Sp_4(\A)}=0$. We want to show that $C_\phi$ is identically zero on $\GSp_4(\A)$.
	Now changing variables and using the left-invariance of $\phi$ under $\GSp_4(\Q)$, if $t \in T(\Q)$ then we have $C_\phi(tg)=C_\phi(g)$.
	In addition, if $z\in Z(\A)$ then $C_\phi(zg)=\omega_\pi(z)C_\phi(g)$. 	Moreover, since $\pi$ is an admissible representation, $\phi$~is right-invariant by $\GSp_4(\Z_p)$ for almost all prime $p$.
	It follows that there exists a finite set of places $S$ such that for any $g \in \GSp_4(\A)$, 
	if $\mu(g) \in \Q^\times (\A^\times)^2 \prod_{p \not \in S}\Z_p^\*$ then $C_\phi(g)=0$.
	The following lemma concludes the proof.
\end{proof}

\begin{lemma}
 Let $S$ be any finite set of places containing $\infty$.
 We have $\Q^\times (\A^\times)^2 \prod_{p \not \in S} \Z_p^\times=\A^\times$.
\end{lemma}
\begin{proof}
	Let $x \in \A^\times$.
	By strong approximation, we have $x=qu$, with  $q \in \Q^\times$ and $u \in \R_{>0}\prod_{p < \infty} \Z_p^\times$.
	Now by the Chinese Remainders Theorem, there exists an integer $n>0$ such that for all finite $p \in S$, we have $n u_p \in (\Z_p^\times)^2$.
	For all $p \not \in S$, let $\epsilon_p \in  \Z_p^\times$ such that $\epsilon_p n u_p \in (\Z_p^\times)^2$.
	Define $\epsilon_p=1$ for $p \in S$.
	Then $n \epsilon u \in (\A^\times)^2$, and $x=(qn^{-1})(n \epsilon u)\prod_{p \not \in S}\epsilon_p^{-1}$.
\end{proof}

	\subsection{The basic kernel}\label{basickernel}
The group $Z(\Q)Z(\R) \backslash Z(\A)$ is compact and acts on the Hilbert space $L^2(G(\Q)Z(\R) \backslash G(\A))$ by 
right translation. 
Since $Z(\Q)Z(\R) \backslash Z(\A)$ is abelian, its irreducible representations are characters, thus
by Peter-Weyl theorem we have 
$$L^2(G(\Q)Z(\R) \backslash G(\A))=\bigoplus_{\omega}L^2(G(\Q)Z(\R) \backslash G(\A),\omega),$$
where the orthogonal direct sum ranges characters of $Z(\A)$ that are trivial on $Z(\Q)Z(\R)$, and
$L^2(G(\Q)Z(\R) \backslash G(\A),\omega)$ is the subspace of $L^2(G(\Q)Z(\R) \backslash G(\A))$
of functions $\phi$ satisfying $$\phi(gz)=\omega(z)\phi(g)$$ for all $z \in Z(\A)$.
Fix such a character $\omega$. If $f:G(\A) \to \C$ is a measurable function that satisfies
\begin{itemize}
	\item  $f(gz)=\overline{\omega}(z)f(g)$ for all $z \in Z(\A)$,
	\item $f$ is compactly supported modulo $Z(\A)$,
\end{itemize}
then we define an operator $R(f)$ on $L^2(G(\Q)Z(\R) \backslash G(\A),\omega)$ by 
$$R(f) \phi(x)=\int_{\GmodZ(\A)}f(y)\phi(xy)dy,$$
where $\overline{G}$ denotes $G / Z$. 
By $G(\Q)$-invariance of $\phi$, we have
\begin{align*}
	R(f) \phi(x)=\int_{\GmodZ(\A)}f(x^{-1}y)\phi(y)dy
	=\sum_{\gamma \in \GmodZ(\Q)}\int_{\GmodZ(\Q) \backslash \GmodZ(\A)}f(x^{-1}\gamma y)\phi(y)dy
\end{align*}
Hence, setting
\begin{equation}\label{TheKernel}
	K_f(x,y)=\sum_{\gamma \in \GmodZ(\Q)}f(x^{-1}\gamma y),
\end{equation}
we have 
\begin{equation}\label{basicequation}
	R(f) \phi(x)=\int_{\GmodZ(\Q) \backslash \GmodZ(\A)}K_f(x,y)\phi(y)dy.
\end{equation}

Now let us argue informally to motivate the more technical actual reasoning.
Let us pretend that $K_f(x,.)$ is an element of $L^2(G(\Q)Z(\R) \backslash G(\A),\omega)$, and that
$L^2(G(\Q)Z(\R) \backslash G(\A),\omega)$ has a Hilbert orthonormal base $\B$.
Then we would have 
$$K(x,.)=\sum_{\phi \in \B} \langle K(x,.) | \overline{\phi} \rangle \overline{\phi}.$$
But equation~(\ref{basicequation}) says that $\langle K(x,.) | \overline{\phi} \rangle=R(f) \phi(x)$.
Thus we might expect a spectral expansion of the kernel of the form
\begin{equation}\label{informal}
	K(x,y)=\sum_{\phi \in \B}R(f) \phi(x) \overline{\phi(y)}.
\end{equation}
If moreover each element $\phi$ of our base $\B$ is an eigenfunction of the operator $R(f)$, say
\begin{equation}\label{evalue}
	R(f)\phi=\lambda_f(\phi)
\end{equation}
then the above expansion becomes 
$$K(x,y)=\sum_{\phi \in \B}\lambda_f(\phi) \phi(x) \overline{\phi(y)}.$$
Finally, integrating $K(x,y)$ on $U \times U$ against a character $\overline{\psi_1(x)}\psi_2(y)$ would then yield 
a spectral equality involving the Whittaker coefficients and the eigenvalues $\lambda_f(\phi)$, of the form
\begin{equation}\label{informalspectral}
	\int_{(U(\Q) \back U(\A))^{2}}K(x,y)\overline{\psi_1(x)}\psi_2(y)dxdy
	=\sum_{\phi \in \B}\lambda_f(\phi) \W_{\psi_1}(\phi) \overline{\W_{\psi_2}(\phi)}.
\end{equation}
Note that in the last step we need~(\ref{informal}) to hold not only in the $L^2$ sense, but pointwise, as 
$(U(\Q) \back U(\A))^{2}$ has measure zero.

Of course, $L^2(G(\Q)Z(\R) \backslash G(\A),\omega)$ \textit{does not} have a Hilbert orthonormal base, due to the presence
of continuous spectrum. However, after adding the proper continuous contribution,
a spectral expansion of the form~(\ref{informal}) has been proved by Arthur~\cite{ArthurSpectralExpansion}*{pages 928-934},
building on the spectral decomposition of $L^2(G(\Q)Z(\R) \backslash G(\A),\omega)$ by Langlands.
We may then reduce from global to local as follows.
By general theory, we may choose automorphic forms $\phi$ appearing in the spectral expansion of the kernel
to be factorizable vectors $\phi_\infty \otimes \bigotimes_p \phi_p$.
If moreover we take $f$ factorizable, say $f =f_\infty \prod_p f_p$, then the computation of $R(f) \phi$ reduces
to the computation of the action of each local component $f_v$ on $\phi_v$.
By choosing the local components $f_v$ appropriately, we can ensure that each $\phi_v$ is an eigenvector of the 
operator corresponding to $f_v$, so that~(\ref{evalue}) holds.  
The determination of $\lambda_f(\phi)$ then amounts, at the infinite place, to the study 
of the spherical transform of $f_\infty$, and at finite places $p$, of the action of the local Hecke algebra.
Specifically, from now on we assume $f$ is as follows.
\begin{assumption}\label{testfunction}
	From now on we assume $f = f_\infty \prod_p f_p$ where
	\begin{itemize}
		\item $f_\infty$ is any smooth, left and right $K_\infty$-invariant and $Z(\R)$-invariant function on $G(\R)$,
		whose support is compact modulo the centre and contained in $G^+(\R)=\{g \in G(\R) : \mu(g)>0 \}.$
		\item for all prime $p$, $f_p$ is a left and right $\Gamma_p$-invariant function on $G(\Q_p)$, satisfying
		$f_p(gz)=\overline{\omega_p}(z)f(g)$ for all $z \in Z(\Q_p)$, and compactly supported modulo the centre,
		\item whenever $\Gamma_p \neq G(\Z_p)$, we have 
		$$f_p(g)=\begin{cases}
			\frac{\overline{\omega_p}(z)}{Vol(\overline{\Gamma_p})} \text{ if there exists } z \in Z(\Q_p) \text{ such that } g \in z \Gamma_p\\
			0 \text{ otherwise.}
		\end{cases}$$
	\end{itemize}
\end{assumption}
Note that this assumption can be fulfilled if and only if we have the following compatibility condition
\begin{assumption}
	For all prime $p$, the resriction of $\omega_p$ to $\Gamma_p \cap Z(\Q_p)$ is trivial.
\end{assumption}

Let us remind the following result~\cite{KL}*{Lemma~3.10}
\begin{proposition}\label{Kfixed}
	Let $G$ be a locally compact group, let $K \subset G$ be a closed subgroup, and let $\pi$ be a unitary representation of $G$
	on a Hilbert space $V$ with central character $\omega$. Let $f$ be any left and right $K$-invariant function satisfying
	\begin{itemize}
		\item $f(gz)=\overline{\omega}(z)f(g)$ for all $z$ in the centre $Z$ of $G$,
		\item $f$ is integrable on $G/Z$.
	\end{itemize} 
Then the operator $\overline{\pi}(f)$ on $V$ defined by
$$\overline{\pi}(f)v=\int_{G/Z}f(g)\pi(g)vdg$$
has its image in the $K$-fixed subspace $V^K$ and annihilates the orthogonal complement of this subspace.
\end{proposition}
Because of Assumption~\ref{testfunction}, this result implies only $\Gamma$-fixed automorphic forms having central character $\omega$ will appear
in the spectral decomposition of $K_f$. 
These automorphic forms come from admissible irreducible representations with central character $\omega$ and having a 
$\Gamma$-fixed vector.
In turn, these representations factor as restricted tensor products of local representations having similar local properties.
Furthermore, only those automorphic forms $\phi$ that are generic will survive the integration against a generic character on $U$, hence we may restrict
attention to local representations that are generic.

	\subsection{Non-Archimedean Hecke algebras}
Let $p$ be a prime number, and $f_p$ be the local component of the function $f$ in Assumption~\ref{testfunction}.
Let $(\pi,V)$ be a unitary representation of $G(\Q_p)$ with central character $\omega_p$.
Throughout this section the Haar measure on $G(\Q_p)$ is normalised so that $K_p=G(\Z_p)$ has volume one.
By Proposition~\ref{Kfixed} we have an operator
\begin{equation}\label{localR}
	\overline{\pi}(f_p)v=\int_{\GmodZ(\Q_p)} f(g)\pi(g)vdg.
\end{equation}
acting on the $\Gamma_p$-fixed subspace $V^{\Gamma_p}$ and annihilating the orthogonal complement of this subspace.

First, let us consider the case $\Gamma_p \neq G(\Z_p)$.
Then any $\Gamma_p$-fixed vector $v \in V$ is also fixed by 
$\overline{\pi}(f_p)$, since in this case $$\overline{\pi}(f_p)v=\frac1{Vol(\overline{\Gamma_p})}\int_{\overline{\Gamma}_p} \pi(g)vdg=v.$$

We now turn to the situation $\Gamma_p=K_p=G(\Z_p)$ (in particular, the character $\omega_p$ must be unramified).
We have have the following~\cite{RS2}*{Theorem~7.5.4}. 
\begin{proposition}
	Let $(\pi,V)$ be generic, irreducible, admissible, representation of $G(\Q_p)$. 
	Assume $\pi$ has a non-zero $K_p$-fixed vector.
	Then $V^{K_p}$ has dimension $1$.
\end{proposition}
\begin{remark}
	In~\cite{RS2}*{Theorem~7.5.4} it is assumed $\pi$ has trivial central character. 
	However, in our situation, the fact that $\pi$ has a non-zero $K_p$-fixed vector forces the central character 
	to be unramified. We can thus twist our representation by an unramified character to reduce to the hypothesis of~\cite{RS2}. 
\end{remark}
By definition, any non-zero vector $\phi$ in $V^{K_p}$ is then called the {\bf spherical vector}.
Since $\overline{\pi}(f_p)$ acts on $V^{K_p}$ which is one-dimensional, the spherical vector is an eigenvector of 
$\overline{\pi}(f_p)$. Finally, let us relate the operator $\overline{\pi}(f_p)$ to the action of the unramified Hecke algebra. 
The {\bf local Hecke algebra} $\Hecke(K_p)$ is the vector space of left and right-$K_p$ invariant 
compactly supported functions $f:G(\Q_p) \to \C$, endowed with the convolution product
$$(f * h) (g)=\int_{G(\Q_p)}f(gx^{-1})h(x)dx.$$
If $(\pi,V)$ is a smooth representation of $G(\Q_p)$, then the Hecke algebra $\Hecke(K_p)$ acts on the $K_p$-invariant subspace $V^{K_p}$ by 
$$\pi(f)v=\int_{G(\Q_p)} f(g)\pi(g)vdg.$$
\begin{lemma}\label{ModingOutCentre}
	Let $f$ be a bi-$K_p$ invariant function on $G(\Q_p)$, with a (unramified) central character, and compactly supported modulo the centre.
	There exists a compactly supported bi-$K_p$-invariant function $\tilde{f}$ on $G(\Q_p)$ 
	and a complete set of representatives $\GmodZ$ of $G(\Q_p)/\Q_p^\*$ satisfying
	$\tilde{f}(gz)=f(g)\1_{\Z_p^{\*}}(z)$ for all $g \in \GmodZ$ and $z \in \Q_p^\*$.
\end{lemma} 
\begin{proof}
	By the Cartan decomposition we have $G(\Q_p)=\coprod_{\substack{i \le j \in \Z\\ t \in \Z}}K_p \diag{p^i}{p^j}{p^{t-i}}{p^{t-j}}K_p$.
	Thus we have
	$$G(\Q_p) / \Q_p^\*=\left.
	\left(\coprod_{\substack{j \ge 0 \\ t \in \Z}} K_p \diag{1}{p^j}{p^t}{p^{t-j}}K_p\right) \middle/ \Z_p^\*. \right.$$
	Fix  a complete set of representatives $\overline{K_p}$ of $K_p/\Z_p^\*$.
	Then $\GmodZ= \coprod_{\substack{j \ge 0 \\ t \in \Z}} \overline{K_p} \diag{1}{p^j}{p^t}{p^{t-j}} \overline{K_p}$ 
	is a complete set of representatives of $G(\Q_p)/\Q_p^\*$.
	Moreover, defining $$S= \coprod_{\substack{j \ge 0 \\ t \in \Z}} K_p \diag{1}{p^j}{p^t}{p^{j+t}}K_p \cap Supp(f)=(\Z_p^\* \GmodZ )\cap Supp(f),$$
	the function $\tilde{f}=\1_S \times f$ has the desired properties.
\end{proof}
Now the function $\tilde{f_p}$ attached to $f_p$ by Lemma~(\ref{ModingOutCentre}) is an element of the Hecke algebra, 
and we have $\pi(\tilde{f_p})=\overline{\pi}(f_p)$, as
$$\pi(\tilde{f_p})v=\int_{G(\Q_p)} \tilde{f}(g)\pi(g)vdg
=\int_{G(\Q_p)/\Q_p^\*}\int_{\Q_p^\*} f_p(g)\1_{\Z_p^{\*}}(z)\pi(g)vdg
=\overline{\pi}(f_p)v.$$

 We summarize the above discussion in the following proposition.
\begin{proposition}\label{localevalue}
Let $p$ be a prime number, and $f_p$ be the local component of the function $f$ in Assumption~\ref{testfunction}.
Let $(\pi,V)$ be a unitary representation of $G(\Q_p)$ with central character $\omega_p$.
Then the operator $\overline{\pi}(f_p)$ from Proposition~\ref{Kfixed} acts by a scalar $\lambda_\pi(f_p)$ on the $\Gamma_p$ fixed
subspace $V^{\Gamma_p}$ and annihilates the orthogonal complement of this subspace.
Moreover, if $\Gamma_p \neq G(\Z_p)$ then $\lambda_\pi(f_p)=1$, and if $\Gamma_p = G(\Z_p)$ then $\overline{\pi}(f_p)$ 
equals the Hecke operator $\pi(\tilde{f_p})$, where $\tilde{f_p}$ is given by Lemma~\ref{ModingOutCentre}.
\end{proposition}

\subsection{The Archimedean representation}

We first show that in our situation the  representation at the Archimedean place must be an irreducible principal series representation, that is full induced from the Borel subgroup.
A representation of $G(\R)$ which has a non-zero $K_\infty$-fixed vector is called {\bf spherical}.
\begin{proposition}\label{MustBePrincipalSeries}
	Any generic irreducible spherical representation $(\pi, V)$ of $G(\R)$ vector is a principal series representation.
\end{proposition}
The author wishes to thank Ralf Schmidt for communicating the following argument.
\begin{proof}
	As explained at the end of~\cite{Vogan}, the generic representations are exactly the ``large" ones, 
	\textit{i.e.}, those with maximal Gelfand-Kirillov dimension.
	The Gelfand-Kirillov dimension of all irreducible representations of $\GSp_4(\R)$ have been calculated in~\cite{thesis}*{Appendix~A}.
	In particular the maximal Gelfand-Kirillov dimension is 4, and the irreducible large representations are either discrete series 
	or limit of discrete series, induced from the Siegel parabolic subgroup, Langlands quotient of representation induced from the
	Klingen subgroup, or principal series representations.
	Now the multiplicity of each possible $K_{\infty}$-type are described in~\cite{thesis}*{Chapter~4}, 
	and among large representations of $\GSp_4(\R)$ only principal series representations contain the trivial $K_{\infty}$-type. 	
\end{proof}
It is then known that the trivial $K_\infty$-type occurs in $\pi$ with multiplicity one~\cite{MiyOda}, that is to say there is 
a unique $K_\infty$-fixed vector in the space $V$. 
Moreover, $\pi$ has a unique {\bf Whittaker model}, and the image of a non-zero $K_\infty$-fixed vector is by definition given by the 
{\bf Whittaker function}. The Whittaker function is an eigenfunction of the centre of the universal enveloping algebra,
which acts as an algebra of differential operators. One may then obtain a system of partial differential equations characterizing 
the Whittaker function, and compute it explicitly. The Whittaker function may also be computed by the mean of the Jacquet integral.
This has been done by Niwa~\cite{Niwa} and Ishii~\cite{Ishii}.
\subsubsection{The spherical transform} 
In this section we normalize the Haar measure on $\Sp_4(\R)$ so that $K_\infty$ has measure $1$.
If $h$ is any bi-$K_\infty$-invariant compactly supported function on $\Sp_4(\R)$, its {\bf spherical transform} is the function $\tilde{h}$
defined on $\Lie{a}^*(\C)$ by
\begin{equation}\label{SphericalTransform}
	\tilde{h}(\nu)=\int_{\Sp_4(\R)}h(g)\phi_{-\nu}(g)dg,
\end{equation}
where 
\begin{equation}\label{sphericalfunction}
	\phi_{-\nu}(g)=\int_{K_\infty}e^{(\rho-\nu)(A(kg))}dk
\end{equation}
is the {\bf spherical function} with parameter $-\nu$ (here $\rho$ is the half-sum of positive roots).

\begin{proposition}\label{archimedeanevalue}
	Let $f_\infty$ be the Archimedean component of the function $f$ in Assumption~\ref{testfunction}.
	Let $(\pi,V)$ be a generic irreducible unitary representation representation of $G(\R)$ with trivial central character.
	Then the operator $\overline{\pi}(f_\infty)$ from Proposition~\ref{Kfixed} acts by a scalar $\lambda_\pi(f_\infty)$ on the $K_{\infty}$ fixed
	subspace $V^{K_{\infty}}$ and annihilates the orthogonal complement of this subspace.
	Moreover, provided this subspace $V^{K_\infty}$ is non zero, then $\pi$ is a principal series representation,
	and $\lambda_\pi(f_\infty)=\tilde{f_\infty}(-\nu)$, where $\tilde{f_\infty}$ is the spherical transform of $f_\infty$ and $\nu$ is
	the spectral parameter of $\pi$.
\end{proposition}
\begin{proof}
	If $V^{K_\infty}$ is zero then by Proposition~\ref{Kfixed} the statement is vacuous.
	Assume now $\pi$ has a non-zero fixed vector.
	By Proposition~\ref{MustBePrincipalSeries}, $\pi$ is then a principal series.
	Then $V^{K_\infty}$ is one-dimensional, so if $v$ is any $K_\infty$-fixed vector in $V$ then
	we have 
	\begin{equation}\label{sptr}
		\pi(f_\infty)v=\lambda_\pi(f_\infty)v
	\end{equation} 
	for some complex number $\lambda_\pi(f)$.
	Since $\pi$ is induced by a character of the Borel subgroup, 
	to compute the eigenvalue $\lambda_\pi(f)$, we may realize $\pi$ as acting by right translation on
	a space of functions $\phi$ satisfying for all $g \in G(\R)$, $n \in U(\R)$ and $a \in T^{+}(\R)$
\begin{equation}\label{stdspace}
	\phi(nag)=e^{(\rho+\nu)(\log(a))}\phi(g),
\end{equation}
where $\nu \in \Lie{a}^*(\C)$ is the {\bf spectral parameter} of $\pi$.
We may view a $Z(\R)$-invariant function supported on $G(\R)^+$ as a function on $\Sp_4(\R)$, so 
the operator $\overline{\pi}(f)$ of Proposition~\ref{Kfixed} is given by
\begin{equation}\label{infinitR}
	\overline{\pi}(f_\infty)v=\int_{\GmodZ(\R)} f_\infty(g)\pi(g)vdg=\int_{\Sp_4(\R)} f_\infty(g)\pi(g)vdg.
\end{equation}
If $\phi$ is a non-zero $K_\infty$-fixed function satisfying~(\ref{stdspace}) then because of the Iwasawa decomposition we must have 
$\phi(1) \neq 0$. 
Using the integration formula~\cite{Helgason}*{Ch.~I~Corollary~5.3} and right-$K_\infty$ invariance we may compute
\begin{align*}
	\pi(f_\infty) \phi(1) &=\int_{\Sp_4(\R)} f_\infty(g)\pi(g) \phi(1) dg \\
	&=\int_{K_{\infty}}\int_{U A^+}f_\infty(an)\phi(an)dadndk
	=\int_{U A^+}f_\infty(an)e^{(\rho_{\text{B}}+\nu)(\log(a))}dadn\phi(1),
\end{align*}
where $A^+$ is the subgroup of $A(\R)$ with positive diagonal entries.
Therefore, using the Iwasawa decomposition and left-$K_\infty$ invariance of $f_\infty$, 
the eigenvalue $\lambda_\pi(f)$ is given by
\begin{align*}
	\lambda_\pi(f)&=\int_{\Sp_4(\R)} f_\infty(g)e^{(\rho+\nu)(A(g))}dg\\
	&=\int_{K_\infty}\int_{\Sp_4(\R)}	f_\infty(g)e^{(\rho+\nu)(A(kg))}dgdk
	=\int_{\Sp_4(\R)}f_\infty(g)\phi_{\nu}(g)dg=\tilde{f_\infty}(-\nu).
\end{align*}
\end{proof}

The spherical transform $\tilde{f_\infty}$ will thus  play the role of the test function on the spectral side of our formula.
On the other hand, the geometric side will involve some different integral transform of our test function $f_\infty$.
It is therefore natural to investigate the analytic properties of $\tilde{f_\infty}$, and to seek to recover $f_\infty$ from $\tilde{f_\infty}$.
This can be achieved by the Paley-Wiener theorem and Harish-Chandra inversion theorem.
\subsubsection{The Paley-Wiener theorem and Harish-Chandra inversion theorem.}
The material in this section is taken from~\cite{Helgason}.
Let us introduce a bit of notation.
We denote by $\langle, \rangle$ the Killing form on the Lie algebra of $\Sp_4(\R)$,
and we define for each $\nu \in \Lie{a}^*$ a vector $A_\nu \in \Lie{a}$ by 
$\nu(H)=\langle A_\nu, H\rangle$ for all $H \in \Lie{a}$.
We then define $\scal{\lambda}{\nu}=\scal{A_\lambda}{A_\nu}$.
We define $\Lie{a}_+$ as the subset of elements $H \in \Lie{a}$ satisfying $\alpha(H) > 0$ for all $\alpha \in \Phi_{\text{B}}$,
and $\Lie{a}^*_+=\{\nu \in \Lie{a} : A_\nu \in \Lie{a}_+ \}$.
Explicitly the Killing form is given by $\langle X, Y \rangle =6Tr(XY)$ and 
$\Lie{a}_+=\left\{\diag{x}{y}{-x}{-y}: 0<x<y \right\}$.

Harish-Chandra's {\bf $c$-function} captures the asymptotic behaviour of the spherical function and it gives the Plancherel measure.
More precisely, by Theorem~6.14 of~\cite{Helgason}*{Chap. IV}, if $H \in \Lie{a}^+$ and $\nu \in \Lie{a}^*_+$
then we have 
$$\lim_{t \to +\infty}e^{(-\nu+\rho)(tH)}\phi_{-i\nu}(\exp(tH))=c(-i\nu).$$
Moreover, $c(\nu)$ is given, for $\nu \in \Lie{a}^*_+$, by the absolutely convergent integral
\begin{equation}\label{cfunction}
	c(\nu)
	=\int_{U(\R)}e^{(\nu+\rho)(A(Ju))}du,
\end{equation}
where the measure $du$ is normalized so that $c(\rho)=1$,
and has meromorphic continuation to $\Lie{a}^*(\C)$ given in our situation by the expression
$$c(-i\nu)=
c_0\prod_{\alpha \in \Phi}
\frac{2^{-\scal{i\nu}{\alpha_0}}\Gamma(\scal{i\nu}{\alpha_0})}
{	\Gamma\left(\frac{\frac32+\scal{i\nu}{\alpha_0}}2\right)
	\Gamma\left(\frac{\frac12+\scal{i\nu}{\alpha_0}}2\right)},$$
where $\Phi$ is the set of roots,
$\alpha_0=\frac{\alpha}{\scal{\alpha}{\alpha}}$ and the constant $c_0$
is such that $c(\rho)=1$.
Using the duplication formula $\Gamma(z)\Gamma(z+\frac12)=\pi^{\frac12}2^{1-2z}\Gamma(2z)$, we can rewrite this as
$$c(-i\nu)=
\frac{c_0}{4\pi^2}\prod_{\alpha \in \Phi}
\frac{\Gamma(\scal{i\nu}{\alpha_0})}{\Gamma(\frac12+\scal{i\nu}{\alpha_0})},$$

We then have the following theorems
\begin{theorem}[Paley-Wiener theorem]\label{PW}
	Let $\PW^R(\Lie{a}^*_\C)$ the set of $\Omega$-invariant entire functions $h$ on $\Lie{\a}^*_\C$ such that for all
	$N \ge 0$ we have 
	$$h(\nu) \ll_N (1+|\nu|)^{-N}e^{R|\Re(\nu)|}.$$
	Let $$ \PW(\Lie{a}^*_\C)=\bigcup_{R>0}\PW^R(\Lie{a}^*_\C).$$
	Then the spherical transform $f \mapsto \tilde{f}$ is a bijection
	from $C^\infty_c(K_\infty\back \Sp_4(\R) / K_\infty)$ to  $\PW(\Lie{a}^*_\C)$.
\end{theorem}
\begin{theorem}[Inversion theorem]\label{sphericalinversion}
	There is a constant $c$ such that for every function $f \in C^\infty_c(K\back \Sp_4(\R) / K)$ we have
	for all $g\in \Sp_4(\R)$
	\begin{equation}\label{inversion}
		cf(g)=\int_{\Lie{a}^*} \tilde{f}(-i\nu)\phi_{-i\nu}(g)\frac{d\nu}{c(i\nu)c(-i\nu)}.
	\end{equation}
\end{theorem}
\begin{remark}
	The constant $c$ may be worked out by Exercise~C.4 of~\cite{Helgason}*{Chap.~IV}.
\end{remark}
\begin{remark}
	Using formulae $\Gamma(iz)\Gamma(-iz)=\frac{\pi}{z\sinh{\pi z}}$ and $\Gamma(\frac12-iz)\Gamma(\frac12+iz)=\frac{\pi}{\cosh{\pi z}}$, 
	the Plancherel measure is given by
	\begin{equation}\label{Plancherel}
		\frac{d\nu}{c(i\nu)c(-i\nu)}=\frac{16\pi^4}{c_0^2}\prod_{\alpha \in \Phi}\scal{\nu}{\alpha_0}\tanh(\pi\scal{\nu}{\alpha_0})d\nu.
	\end{equation}
\end{remark}

\subsubsection{The Whittaker function and the Jacquet integral}
As mentioned above, the Whittaker function is a non-zero $K_\infty$-fixed vector in the Whittaker model, and it is unique up to scaling.
It is given by (meromorphic continuation of) the Jacquet integral.
Namely, if $\psi$ is a generic character of $U(\R)$, we have the {\bf Jacquet integral}
\begin{equation}\label{jacquet}
	W(\nu, g, \psi)=\int_{U(\R)} e^{(\rho+\nu)(A(Jug))}\overline{\psi(u)}du.
\end{equation}

The Jacquet integral converges absolutely for $\Re(\nu) \in \Lie{a}^*_+$, as may be seen by 
using the absolute convergence of~(\ref{cfunction}) and computing
\begin{equation}\label{majoration}
	|W(\nu, g, \psi)| \le 
	\int_{U(\R)}|e^{(\nu + \rho)(A(Jug))}| du =
	e^{(\rho-\Re(\nu))(A(g))}c(\Re(\nu))
\end{equation}
Moreover, it has meromorphic continuation to all $\nu \in \Lie{a}^*_\C$.
Ishii~\cite{Ishii} computed explicit integral representations for the normalized Jacquet integral
$$\mathcal{W}(\nu,g,\psi)=\frac1{4\pi^2}\prod_{\alpha \in \Phi} \Gamma\left(\frac12+\scal{\nu}{\alpha_0}\right)W(\nu,g,\psi),$$
namely (note the different choice of minimal parabolic subgroup) if $a=\diag{a_1}{a_2}{a_1^{-1}}{a_2^{-1}} \in A^+(\R)$ then for any 
$\nu \in \Lie{a^*}_\C$
\begin{equation}\label{NiwaIntegral}
	\begin{aligned}
	\mathcal{W}(\nu,a,\psi)=2a_1a_2^2 &\int_0^\infty \int_0^\infty K_{\frac{\nu_2-\nu_1}{2}}(2\pi v_1)K_{\frac{\nu_1+\nu_2}2}(2 \pi v_2)\\
	& \times \exp\left(-\pi\left(\frac{a_2^2}{v_1v_2}+\frac{v_1v_2}{a_1^2}+a_1^2\left(\frac{v_1}{v_2}+\frac{v_2}{v_1}\right)\right)\right)\frac{dv_1dv_2}{v_1v_2}.
	\end{aligned}
\end{equation}
This implies in particular that the normalized Jacquet integral satisfies the functional equations
\begin{equation}\label{whittakerftneq}
	\mathcal{W}(\sigma \cdot \nu,g,\psi)=\mathcal{W}(\nu,g,\psi)
\end{equation}
for all $\sigma \in \Omega$.
If $t \in A^+(\R)$ and if we denote by $\psi_t$ the character $\psi_t(u)=\psi(t^{-1}ut)$, then it is easy to see 
(first by a change of variable in the domain where the Jacquet integral is absolutely convergent, then by meromorphic continuation)
that 
\begin{equation}\label{whittakertorus}
	W(\nu,g,\psi_t)=e^{(\rho-\nu)(\log(t))}W(\nu, t^{-1}g,\psi).
\end{equation}

\begin{remark}\label{WonA}
	By Lemma~\ref{A(Ju)} and changing variables $u(x,a,b,c) \mapsto u(-x,a,-b,-c)$, if $t \in A^+(\R)$then we have 
	$W(\nu, t,\psi)=W(\nu, t, \overline{\psi}).$
\end{remark}

\subsubsection{Wallach's Whittaker transform}
The following theorem is a consequence of~\cite{Wallach}*{Ch.~15}.
Let $L^2(U \back \Sp_4(\R) /K, \psi)$ be the space of functions $f$ on $\Sp_4(\R)$ satisfying for all 
$u \in U(\R)$, for all $k \in K_\infty$
and for all $g \in \Sp_4(\R)$
$$f(ugk)=\psi(u)f(g) \text{ and } \int_{U \back \Sp_4(\R)} |f(g)|^2dg < \infty.$$
\begin{theorem}[Wallach's Whittaker inversion]\label{WallWhit}
	Define for $\alpha \in C^{\infty}_c(\Lie{a^*})$ 
	\begin{equation*}
		\mathscr{T} (\alpha)(a)=\int_{\Lie{a}^*}\alpha(\nu)W(-i\nu,a,\psi)\frac{d\nu}{c(i\nu)c(-i\nu)}.
	\end{equation*}
	Then the image of the linear map $\mathscr{T}$ is a dense subset $\mathscr W \subset L^2(U \back \Sp_4(\R) /K, \psi)$
	containing $C^\infty_c(U \back \Sp_4(\R)/K, \psi)$, and $\mathscr{T}$ extends to a unitary operator onto 
	$L^2(U \back \Sp_4(\R) /K, \psi)$.
	Moreover, the inverse of $\mathscr{T}$ is given for all $f \in \mathscr W$ by the {\bf Whittaker transform}
	\begin{equation*}
		W(f)(\nu)=c \int_{A^+} f(a)W(i\nu,a,\psi)e^{-2\rho \log a}da,
	\end{equation*}
	where the constant $c$ is the same as in Theorem~\ref{sphericalinversion}.
\end{theorem}

\subsubsection{An integral transform}
Let $g \in G(\R)$, $t \in A^+(\R)$ and $\psi$ a generic character of $U(\R)$.
When dealing with the geometric side of the relative trace formula,
we shall be interested in the integral $$I(f_\infty)=\int_{U(\R)}f_\infty(tug)\overline{\psi}(u)du.$$

Using expression~(\ref{sphericalfunction}) and applying Theorem~5.20 of~\cite{Helgason}*{Ch.I} that relates integration
on $K_\infty$ to integration on $U(\R)$, one may establish the following identity for all $\nu \in \Lie{a}^*_\C$
\begin{equation}\label{sphericalU}
	\phi_{\nu}(g)=\int_{U(\R)}
	e^{(\rho+\nu)(A(Jug)}e^{(\rho-\nu)(A(Ju))}du.
\end{equation}
From this identity, ignoring convergence issues and treating integrals as if they were 
absolutely convergent, one may heuristically expect the following 
\begin{equation}\label{heuristics}
	\int_{U(\R)}\phi_\nu(tug)\overline{\psi(u)}du
	=W(\nu,g,\psi)W(-\nu,t^{-1},\overline{\psi}).
\end{equation}
However, the domain of absolute convergence of the two Jacquet integral in the right hand side are complementary from each other, 
and the integral in the left hand side is likely not absolutely convergent, making such a result,
where the left hand side is (optimistically) a semi-convergent integral and the right-hand side is defined by meromorphic continuation,
likely difficult to prove.

Carrying on with this heuristic and using Theorem~\ref{sphericalinversion}, let us write
\begin{align*}
	cI(f_\infty)&=\int_{U(\R)}\int_{\Lie{a}^*} \tilde{f_\infty}(-i\nu)\phi_{-i\nu}(tug)\frac{d\nu}{c(i\nu)c(-i\nu)}\overline{\psi}(u)du \\
	&=\int_{\Lie{a}^*}\tilde{f_\infty}(-i\nu)\int_{U(\R)} \phi_{-i\nu}(tug)\overline{\psi}(u)du\frac{d\nu}{c(i\nu)c(-i\nu)}\\
	&=\int_{\Lie{a}^*}\tilde{f_\infty}(-i\nu)W(-i\nu,g,\psi)W(i\nu,t^{-1},\overline{\psi})\frac{d\nu}{c(i\nu)c(-i\nu)}.
\end{align*}
Unlike~(\ref{heuristics}), this equality seems more reasonable.
Indeed, the left hand side is absolutely convergent because $f_\infty$ is compactly supported,
and in the right hand side $\tilde{f_\infty}$ has rapid decay.
%and the two Whittaker functions can be estimated
%thanks to integral representations obtained by Niwa~\cite{Niwa} and~Ishii\cite{Ishii}, together
%with estimates of the Bessel $K$-functions with explicit dependence in the spectral aspect.
We now give a rigorous proof of the following.
\begin{theorem}\label{GeomTransform}
	Let $f_\infty$ be a smooth, bi-$K_\infty$, compactly supported function on $\Sp_4(\R)$.
	Let $g \in G(\R)$, $t \in A^+(\R)$ and $\psi$ a generic character of $U(\R)$.
	Then we have 
	$$c\int_{U(\R)}f_\infty(tug)\overline{\psi}(u)du=\int_{\Lie{a}^*}\tilde{f_\infty}(-i\nu)W(-i\nu,g,\psi)W(i\nu,t^{-1},
	\overline{\psi})\frac{d\nu}{c(i\nu)c(-i\nu)},$$
	where $W(\nu,\cdot,\psi)$ is the $\psi$-Whittaker function of the principal series with spectral parameter~$\nu$.
\end{theorem}
\begin{proof}
	Both sides transform on the left by $U(\R)$ according to $\psi$, and are $K_\infty$-invariant.
	Thus by the Iwasawa decomposition, it suffices to prove it for $g=a \in A^+(\R)$.
	Also, by~(\ref{whittakertorus}), we may restrict ourself to $t=1$.
	With notations of Theorem~\ref{WallWhit}, we have
	$$
	\int_{\Lie{a}^*}\tilde{f_\infty}(-i\nu)W(-i\nu,a,\psi)W(i\nu,1,\overline{\psi})\frac{d\nu}{c(i\nu)c(-i\nu)}=
	\mathscr{T}(\alpha)(a),$$
	where $$\alpha(\nu)= \tilde{f_\infty}(-i\nu) W(i\nu,1,\overline{\psi}).$$
	Moreover $g \mapsto \int_{U(\R)}f_\infty(ug)\overline{\psi}(u)du$ belongs to $C^\infty_c(U \back \Sp_4(\R) /K, \psi)$
	since $f_\infty$ is smooth and compactly supported. Hence
	by Wallach's Whittaker inversion it suffices to show that
	for all $\nu \in \Lie{a}^*$ we have 
	\begin{equation}\label{toprove}
		\alpha(\nu) = \int_{{A}^+(\R)}e^{-2\rho \log a} \int_{U(\R)}f_\infty(ua)\overline{\psi}(u)duW(i\nu,a,\psi)da.
	\end{equation}
	Since both sides are meromorphic in $\nu$, it suffices to show this for 
	$\Re(i\nu) \in \Lie{a}^*_+$. In this region, the Jacquet integral 
	$W(i\nu,a,\psi)=\int_{U(\R)} e^{(\rho+i\nu)(A(Ju_1a))}\overline{\psi(u_1)}du_1$ converges absolutely. 
	By Remark~\ref{WonA}, the integral in~(\ref{toprove}) may then be written as
	\begin{align*}
		\int_{{A}^+(\R)}\int_{U(\R)}f_\infty(au)\overline{\psi}(aua^{-1})duW(i\nu,a,\overline{\psi})da
		=\int_{{A}^+(\R)}\int_{U(\R)}f_\infty(au)W(i\nu,au,\overline{\psi})&duda\\
		=\int_{{A}^+(\R)}\int_{U(\R)}f_\infty(au)\int_{U(\R)} e^{(\rho+i\nu)(A(Ju_1au))}\psi(u_1)du_1&duda
	\end{align*}
	Write $Ju_1=n\exp(A(Ju_1))k_0(Ju_1)$ with $n \in U(\R)$ and $k_0(Ju_1) \in K_\infty$.
	Then $A(Ju_1au)=A(Ju_1)+A(k_0au)$. 
	So the integral we have to evaluate becomes
	\begin{align*}
		\int_{{A}^+(\R)}\int_{U(\R)}& e^{(\rho+i\nu)(A(Ju_1)}\psi(u_1)\int_{U(\R)}f_\infty(au)e^{(\rho+i\nu)(A(k_0(Ju_1)au))}dudu_1da\\
		&=	\int_{{A}^+(\R)}\int_{U(\R)} e^{(\rho+i\nu)(A(Ju_1)}\psi(u_1)\int_{\Sp_4(\R)}f_\infty(g)e^{(\rho+i\nu)(A(k_0(Ju_1)g))}dgdu_1da\\
		&=	\int_{{A}^+(\R)}\int_{U(\R)} e^{(\rho+i\nu)(A(Ju_1)}\psi(u_1)\int_{\Sp_4(\R)}f_\infty(g)e^{(\rho+i\nu)(A(g))}dgdu_1da\\
		&=W(i\nu,1,\overline{\psi})\tilde{f}(-i\nu).
	\end{align*}
\end{proof}
\subsubsection{Estimates for the Whittaker function}
We close this section with some estimates for the Whittaker function to be used later on.
We begin with recalling the following estimate for Bessel $K$ functions.
	\begin{lemma}\label{EstimateBessel}
	Let $\sigma>0$. For $\Re(\nu)\in [-\sigma,\sigma]$ we have	for all $\epsilon>0$
	$$
	K_\nu(u) \ll \left\lbrace\begin{array}{ccc}
		(1+|\Im(\nu)|)^{\sigma+\epsilon}	u^{-\sigma-\epsilon} & \mbox{if} & 0<u\le 1+\frac{\pi}2|\Im(\nu)|,\\
		u^{-\frac12} e^{-u} & \mbox{if} & u >  1+\frac{\pi}2|\Im(\nu)|.
	\end{array}
	\right.
	$$

\end{lemma}
In the following lemma, we have only used trivial bounds and haven't seek for optimality.
\begin{lemma}\label{TrivialWhittaker}
Let $\sigma>0$. Let $\nu \in \Lie{a}^*_\C$ with $-\sigma < \frac{\Re(\nu_1-\nu_2)}2, \frac{\Re(\nu_1+\nu_2)}{2} < \sigma$ and $a \in A^{+}(\R)$.
For simplicity, set $r_1=\frac{|\Im(\nu_1-\nu_2)|}2$ and $r_2=\frac{|\Im(\nu_1+\nu_2)|}{2} $
Then for all $\epsilon>0$ we have 
\begin{align*}
	\mathcal{W}(\nu,a,\psi) &\ll (1+r_1)^{\sigma+1+\epsilon} (1+r_2)^{\sigma+1+\epsilon}a_1a_2^{-2\sigma-\epsilon}\\
	&+(1+r_1)^{-\frac32} (1+r_2)^{-\frac32}a_1a_2^{2}\\
	&+(1+r_1)^{\sigma+\epsilon}(1+r_2)^{-(\sigma+\frac52+\epsilon)}a_1^{-2\sigma-1-\epsilon}a_2^2\\
	&+(1+r_1)^{-(\sigma+\frac52+\epsilon)}(1+r_2)^{\sigma+\epsilon}a_1^{-2\sigma-1-\epsilon}a_2^2.
\end{align*}
\end{lemma}
\begin{proof}
	Follows trivially from the explicit integral representation~(\ref{NiwaIntegral}) and Lemma~\ref{EstimateBessel}.
\end{proof}
\begin{proposition}\label{WhittakerSpectral}
	Let $a \in A^{+}(\R)$.
	Then, for $\Re(\nu)$ small enough we have for all $\epsilon >0$
	 $$W(\nu,a,\psi) \ll_{\Re(\nu),a} \prod_{\alpha \in \Phi} |\scal{\Im(\nu)}{\alpha_0}|^{2|\scal{\Re(\nu)}{\alpha_0}|+\epsilon}.$$ 
\end{proposition}
\begin{proof}
	Observe that, if $\Re(\nu) \in \Lie{a}^*_+$, then the claim follows by the trivial bound~(\ref{majoration}).
	Next, if $\Re(\nu)$ belongs to any open Weyl chamber, there is $\sigma \in \Omega$ such that $Re(\sigma \cdot \nu) \in \Lie{a}^*_+$.
	The functional equation~({\ref{whittakerftneq}}) gives 
	$$W(\nu,a,\psi)=\prod_{\alpha \in \Phi} \frac{\Gamma\left(\frac12+\scal{\sigma \cdot \nu}{\alpha_0}\right)}{\Gamma\left(\frac12+\scal{ \nu}{\alpha_0}\right)}W(\sigma \cdot\nu,a,\psi).$$
	Since the Weyl group acts by permutation on the set of (positive and negative) roots, the product can be written as 
	$$\prod_{\alpha \in \Phi_\sigma} \frac{\Gamma\left(\frac12-\scal{\nu}{\alpha_0}\right)}{\Gamma\left(\frac12+\scal{ \nu}{\alpha_0}\right)}
	\ll \prod_{\alpha \in \Phi_\sigma} |\scal{\Im(\nu)}{\alpha_0}|^{-2\scal{\Re(\nu)}{\alpha_0}},$$
	where $\Phi_\sigma$ is the set of positive roots whose image by $\sigma$ is a negative root and we have used
	that $|\Gamma(x+iy)| \sim \sqrt{2\pi}e^{-\frac{\pi}{2}|y|}|y|^{x-\frac12}$ as $|y| \to \infty$ and that the numerator has no poles because 
	$\Re(\nu)$ is small enough.
	But if $\sigma \cdot \alpha$ is a negative root then we have $\scal{\Re(\nu)}{\alpha_0} <0$ and so we are done in this case again.
	Finally, if $\Re(\nu)$ belongs to a wall of a Weyl chamber, by Lemma~\ref{TrivialWhittaker} we may apply the Phragm{\'e}n-Lindel{\"o}f principle 
	to deduce the result.		
\end{proof}
\section{Eisenstein series and the spectral decomposition}
The goal of Eisenstein series is to describe the continuous spectrum. 
The latter is an orthogonal direct sum over standard parabolic subgroups $P$, each summand of which 
is a direct integral parametrized by $i\Lie{a}^*_P$.
Eisenstein series will give intertwining operators from some representation induced from 
$M_P$ to the corresponding part of the continuous spectrum.
One thus wants to define $E(\cdot,\phi,\nu)$ for $\phi$ in the space $\H_P$ of the aforementioned 
induced representation, and for $\nu \in i\Lie{a}^*_P$.
Because of convergence issues, one originally defines $E(\cdot,\phi,\nu)$ for $\phi$ 
lying a certain dense space of automorphic forms $\H_P^0 \subset \H_P$ and for $\nu \in \Lie{a}^*_P(\C)$ with large enough real part. 
The definition is then extended to all $\phi$ in the completion of $\H_P^0$ and to all
$\nu \in  \Lie{a}^*_P(\C)$.	 Our exposition follows Arthur, and in particular~\cite{ArthurIntro}.
\subsection{Definition of Eisenstein series}~\label{DefOfES}
	Fix a standard parabolic subgroup $P=N_PM_P$ throughout this section, and let $A_P$ be the centre of $M_P$,
and $A_P^+(\R)$ be the connected component of $1$ in $A_P(\R)$.
Let $R_{M_P, \text{disc}}$ be the restriction of the right regular representation of $M_P(\A)$
on the subspace of $L^2(M_P(\Q)A_P^+(\R)\back M_P(\A))$ that decompose discretely.
For $\nu \in \Lie{a}_P^*(\C),$ consider the tensor product
$R_{M_P, \text{disc}, \nu}(x)=R_{M_P, \text{disc}}(x)e^{\nu (H_{M_P}(x))}$.
The continuous spectrum is described via the Eisenstein series in terms of the induced representation
$$\I_P(\nu)=\text{Ind}_{P(\A)}^{G(\A)}(I_{N_P(\A)} \otimes R_{M_P, \text{disc}, \nu} ).$$
The space of this induced representation is independant of $\nu$ and is given in the following definition.
\begin{definition}\label{defofHP}
	Let $P=N_PM_P$ be the Levi decomposition of $P$, 
	$A_P$ be the centre of $M_P$ and $A_P^+(\R)$ be the connected component of $1$ in $A_P(\R)$. 
	We define $\H_P$ to be the Hilbert space obtained by completing
	the space $\H_P^0$ of functions
	\begin{equation}\label{Hp}
		\phi : N_P(\A) M_P(\Q) A_P^+(\R) \backslash G(\A) \to \C 
	\end{equation} such that
	\begin{enumerate}
		\item \label{discretness} 
		for any $x \in G(\A)$, the function $M_P(\A) \to \C, m \mapsto \phi(mx)$ 
		is $\mathscr Z_{M_P}$-finite, where $\mathscr Z_{M_P}$ is the centre of the universal
		enveloping algebra of $\mathfrak M_P(\C)$,
		\item \label{Kfinitness} $\phi$ is right $K$-finite,
		\item \label{L2} $\| \phi\|^2 = \int_K \int_{A_P(\R)^+ M_P(\Q) \backslash M_P(\A)} |\phi(mk)|^2dm dk
		< \infty.$
	\end{enumerate}
\end{definition}
Then the representation $\I_P(\nu)$ acts on $\H_P$ via 
$$(\I_P(\nu,y) \phi)(x)=\phi(xy)\exp((\nu+\rho_P)(H_P(xy)))\exp(-(\nu+\rho_P)(H_P(x))),$$
and is unitary for $\nu \in i\Lie{a}_P(\C)$.

We now define the Eisenstein series attached to $P$. 
We may $H_P$ extends to $P(\Q) \back G(\A)$ by setting $H_P(nmk)=H_P(m)$ ($n \in N_P, m \in M_P, k \in K$), therefore 
the expression in the following proposition is well defined.

\begin{proposition}\label{defofES}
	For $\nu \in \Lie{a}_P(\C)$ with large enough real part, if $x \in G(\A)$ and $\phi \in \H_P^0$, the 
	{\bf Eisenstein series}
	$$E(x, \phi, \nu)= \sum_{\delta \in P(\Q) \backslash G(\Q)} \phi(\delta x) \exp((\nu+\rho_P)(H_P(\delta x))$$
	converges absolutely.
\end{proposition}

Langlands provided analytic continuation of Eisenstein series, as well as the spectral decomposition of $L^2(Z(\R)G(\Q) \back G(\A))$.
The latter gives a decomposition of the right regular representation $R$ as direct sum over association classes of parabolic subgroups.
The class of $G$, viewed as a parabolic subgroup itself, gives the {\bf discrete spectrum}.
It consists on one hand of cuspidal functions on $Z(\R)G(\Q) \back G(\A)$ and on the other hand 
of residues of Eisenstein series attached to proper parabolic subgroups.
The contribution of the other classes is given by direct integrals of corresponding induced representations and gives the 
{\bf continuous spectrum}.
We now describe explicitly the Eisenstein series that are relevant for us.

\subsection{Action of the centre  and of the compact \texorpdfstring{$\Gamma$}{Gamma}}

Since our test function $f$ is bi-$\Gamma$-invariant and has central character $\omega$, 
Eisenstein series occurring in the spectral expansion of its kernel $K_f$ are only from the subspaces of $\H_P$ satisfying similar properties 
(see Lemma~\ref{UnprovedLemma} below for a formal justification).
Using the Peter-Weyl Theorem, we can further reduce:
\begin{lemma}\label{CentralCharactersofP}
	Let $P$ be a standard parabolic subgroup and $A_P$ it centre.
	Let  $\H_P^\Gamma(\omega)$ be the closed subspace of $\H_P$ 
	consisting in functions $\phi$ such that for all $z \in Z(\A)$ and $k \in \Gamma$, we have $\phi(zgk)=\omega(z)\phi(g)$.
	Then 
	\begin{equation}\label{DecompOfCentralChar}
		\H_P^\Gamma(\omega)=\bigoplus_{\chi} \H_P^\Gamma(\chi)
	\end{equation}
	where the $\chi$-orthogonal direct sum ranges over characters of $A^+_P(\R)A_P(\Q) \backslash A_P(\A)$ that coincide with 
	$\omega$ on $Z(\A)$, and $\H_P^\Gamma(\chi)$ is the subspace of $\H_P^\Gamma(\omega)$
	consisting in functions $\phi$ such that for all $z \in A_P(\A)$, $\phi(zg)=\chi(z)\phi(g)$.
\end{lemma}

\subsection{Explicit description of Eisenstein series}
	If $R_{M_P,\text{disc}}=\bigoplus_{\pi} \pi = \bigoplus_\pi \left(\bigotimes_v \pi_v \right)$ is the decomposition of 
	$R_{M_P,\text{disc}}$ into irreducible representations $\pi = \bigotimes_v \pi_v$ of $M_P(\A) / A_P(\R)^+$, then we have 
	$$\I_P(\nu)=\bigoplus_{\pi} \I_P(\pi_\nu) = \bigoplus_\pi \left(\bigotimes_v \I_P(\pi_{v,\nu})\right).$$
Moreover the representation space of each $\I_P(\pi_\nu)$ does not depend on $\nu$.
Hence, to describe the spaces $\H_P^\Gamma(\chi)$ it suffices to describe
\begin{itemize}
	\item the irreducible representations $\pi$ with central character $\chi$ occurring $R_{M_P,\text{disc}}$, 
	\item the $\Gamma$-fixed subspace of each representation $\I_P(\pi_\nu)$.
\end{itemize}
By the Iwasawa decomposition, elements of this space may be viewed as families of functions indexed by $\Gamma \back K$ satisfying some 
compatibility condition that we proceed to make explicit now. 
We also prove that the Archimedean part of $\I_P(\pi_\nu)$ is a principal series representation, and we provide its spectral parameter. 

\subsubsection{Borel Eisenstein series} 
\begin{lemma}\label{BorelSpectralParameter}
	The irreducible representations occurring in $R_{T,\text{disc}}$ are precisely characters $\chi$ of $T^+(\R)T(\Q) \back T(\A)$. 
	Let $\chi$ be such character and $\nu \in i\Lie{a}^*(\C)$. 
	The Archimedean part of $\I_B(\chi_\nu)$ is an irreducible principal series representation with spectral parameter $\nu$.
\end{lemma}
\begin{proof}
	The first part is because $T^+(\R)T(\Q) \back T(\A)$ is abelian.
	For the second part, since $\chi_\infty=1$ we have  $\I_B(\chi_\nu)_\infty= \I_B(e^\nu)$, which is irreducible
	because $\nu \in i\Lie{a}^*(\C)$ (see~\cite{Muic}*{Lemma~5.1}).
\end{proof}
Characters $\chi$ of $T^+(\R)T(\Q) \back T(\A)$ that coincide with $\omega$ on $Z(\A)$ are in one-to-one 
correspondence with triplets $(\omega_1, \omega_2, \omega_3)$ of characters of $\R_{>0}\Q^\* \backslash \A^\*$  satisfying
$\omega_1 \omega_2 \omega_3^2 = \omega$, via 
$$\chi\left(\left[ \begin{smallmatrix}x&&&\\&y&&\\&&tx^{-1}&\\&&&ty^{-1}\end{smallmatrix}\right]\right)=\omega_1(x)\omega_2(y) \omega_3(t).$$
Define a character of $B$ by 
$$\omega\left(\left[ \begin{smallmatrix}x&*&*&*\\ &y&*&*\\&&tx^{-1}&*\\&&&ty^{-1}\end{smallmatrix}\right]\right)
=\omega_1(x)\omega_2(y)\omega_3(t)$$
(note that this notation is sound, as it coincides with our original $\omega$ on scalar matrices.)
\begin{proposition}\label{Borelsummand}
	Let $\chi=(\omega_1,\omega_2,\omega_3)$ with $\omega_1 \omega_2 \omega_3^2 = \omega$.
	Consider $(\phi_k)_{k \in K / \Gamma}$ such that
	\begin{enumerate}
		\item for all $k$, $\phi_k \in \C$,
		\item \label{Borelcompatibility} if $\gamma \in K \cap B(\A)$ then for all $k$, 
		$\phi_k=\chi(\gamma^{-1})\phi_{\gamma k}$.
	\end{enumerate} 
	Then the function on $G(\A)$ given for  $u \in U(\A), t \in T(\A), k \in K$ by
	\begin{equation}\label{Borelnmk}
		\phi(utk)= \chi(t) \phi_k,
	\end{equation}
	is well-defined and belongs to ${\H_B}^\Gamma(\chi)$. 
	Moreover, every function in ${\H_{B}}^\Gamma(\chi)$ has this shape.
\end{proposition}
\begin{proof}
	We first prove that $\phi$ is well-defined.
	Suppose $u_1t_1k_1=u_2t_2k_2$.	
	In particular $k_1k_2^{-1}=(u_1t_1)^{-1}(u_2t_2) \in B(\A) \cap K$.
	Therefore
	$$
	\chi(t_1)\phi_{k_1}=	\chi(t_1)\chi({k_1k_2^{-1}})\phi_{k_2}\\
	=\chi(t_2)\phi_{k_2}.$$
	Next we show that $\phi$ belongs indeed to  ${\H_{B}}^\Gamma(\omega_1, \omega_2, \omega_3)$. 
	The fact that $\phi$ is invariant on the left by $U(\A)T(\Q)T(\R)^+$, the right invariance by $\Gamma$ and the fact that $\phi$ transforms
	under $T(\A)$ according to $\chi$ are obvious from the definition. Finally,
	\begin{align*}
		\int_K \int_{T(\R)^+ T(\Q) \backslash T(\A)} |\phi(mk)|^2dmdk
		&=\int_K \int_{(\R_{>0} \Q^\* \backslash \A^\*)^3} |\phi_k|^2dmdk\\
		&=Vol(\R_{>0} \Q^\* \backslash \A^\*)^3Vol(\Gamma)\sum_{k \in K / \Gamma}	|\phi_k|^2
		<\infty
	\end{align*}
	since $\R_{>0}\Q^\* \backslash \A^\*$ is compact and $K / \Gamma$ is finite.
		As a last point, we show that we thus exhaust all of ${\H_{B}}^\Gamma(\omega_1, \omega_2,\omega_3)$.
	Let $\phi \in {\H_{B}}^\Gamma(\omega_1, \omega_2,\omega_3)$. Define 
	$$\phi_k=\phi(k).$$
	Then it is clear that equation~(\ref{Borelnmk}) holds.
	As for condition~\ref{Klingencompatibility}, note that if $\gamma=t_\gamma u_\gamma \in K \cap B(\A)$ 
	with $t_\gamma \in T(\A)$ and $u_\gamma \in U(\A)$ then
	\begin{align*}
		\phi_{\gamma k}&=\phi(\gamma k)\\
		&=\phi (t_\gamma u_\gamma k) = \chi(t_\gamma) \phi(k)\\
		&=\chi(\gamma)\phi_k.
	\end{align*}
\end{proof}
\begin{remark}
	Consider the action of $K \cap B(\A)$ on $K / \Gamma$ by multiplication on the left.
	Then the compatibility condition 2. can only be met if $\omega$ is trivial on
	the stabilizer of each element of $K / \Gamma$.
	In this case, the dimension of ${\H_{B}}^\Gamma(\chi)$ is the number of distinct orbits.
\end{remark}

\subsubsection{Klingen Eisenstein series}~\label{KES}
Characters $\chi$ of $\Ak^+(\R)\Ak(\Q) \back \Ak(\A)$ that coincide with $\omega$ on $Z(\A)$ are in one-to-one 
correspondence with pairs $(\omega_1, \omega_2)$ of characters of $\R_{>0}\Q^\* \backslash \A^\*$  satisfying
$\omega_1 \omega_2 = \omega$, via 
$$\chi\left(\diag{u}{t}{u}{t^{-1}u^2}\right)=\omega_1(u)\omega_2(t).$$
For convenience, if $A=\mat{a}{b}{c}{d} \in \GL_2$, define 
$\iota_A=\left[ \begin{smallmatrix}a&&b&\\ &1&&\\ c&&d&\\&&&\det(A)\end{smallmatrix}\right]\in \Mk,$
and if 
$$
p=\left[ \begin{smallmatrix}
	a	&	&b	&	*			\\ 
	*	&t	&*	&	*			\\ 
	c	&	&d	&	*			\\
	&	&	& t^{-1}\det(A)
\end{smallmatrix}\right],
$$
define $\sigma_{\text K}(p)=A$ and $t(p)=t$.
We may extend $\sigma_{\text K}$ to all of $\Pk(\A)$ by setting $\sigma_{\text K}(nm)=\sigma_{\text K}(m)$ (and similarly for $t$), 
and we may view $\omega_2$ as the character of $\Pk(\A)$ defined by $\omega_2(p)=\omega_2(t(p))$.
\begin{lemma}\label{KlingenInducing}
	Let $\chi=(\omega_1,\omega_2)$ with $\omega_1\omega_2=\omega$.
	The irreducible representations with central character $\chi$ occurring in $R_{\Mk,\text{disc}}$ are twists $\omega_2 \otimes \pi$, where $\pi$ 
	occurs in the discrete spectrum of $L^2(\R_{>0}\GL_2(\Q)\back \GL_2(\A))$ and has central character $\omega_1$.
\end{lemma}
\begin{proof}
	Let $\pi$ be an irreducible representations with central character $\chi$ occurring $R_{\Mk,\text{disc}}$.
	By definition, we may realize $\pi$ in the subspace of $L^2(\Mk(\Q)\Ak^+(\R)\back \Mk(\A))$ consisting of functions with central character 
	$\chi$.
	This space identifies with $L^2(\R_{>0}\GL_2(\Q)\back \GL_2(\A),\omega_1)$ via 
	$\phi \mapsto \left( \left[ \begin{smallmatrix}
		a	&	&b	&				\\ 
			&t	&	&				\\ 
		c	&	&d	&				\\
		&	&	& t^{-1}\det(A)
	\end{smallmatrix}\right] \mapsto \omega_2(t)\phi(\mat{a}{b}{c}{d}) \right).$
\end{proof}

\begin{proposition}\label{Klingensummand}
	Let $\chi=(\omega_1,\omega_2)$ with $\omega_1\omega_2=\omega$.
	Let $(\pi,V_\pi)$ occur in the discrete spectrum of $L^2(\R_{>0}\GL_2(\Q)\back \GL_2(\A))$ with central character $\omega_1$.
	Consider $(\phi_k)_{k \in K / \Gamma}$ such that
	\begin{enumerate}
		\item for all $k$, $\phi_k \in V_\pi$,
		\item \label{Klingencompatibility} if $\gamma \in K \cap \Pk(\A)$ then for all $k$, 
		$\phi_k(\cdot \sigma_{\text K}(\gamma))=\omega_2 \circ t(\gamma^{-1})\phi_{\gamma k}$.
	\end{enumerate} 
	Then the function on $G(\A)$ given for  $n \in \Nk(\A), m \in \Mk(\A), k \in K$ by
	\begin{equation}\label{Klingennmk}
		\phi(nmk)= \omega_2 \circ t(m) \phi_k(\sigma_{\text K}(m)),
	\end{equation}
	is well-defined and belongs to ${\I_{\Pk}((\omega_2 \otimes \pi)_\nu)}^\Gamma$. 
	Moreover, every function in ${\I_{\Pk}((\omega_2 \otimes \pi)_\nu)}^\Gamma$ has this shape.
\end{proposition}
\begin{remark}\label{KlingenMaassForm}
	Condition~(\ref{Klingencompatibility}) implies that each $\phi_k$ is right-$SO_2(\R)$-invariant (and hence must be an adelic Maa{\ss} form or a character).
	Indeed, let $v \le \infty$ and let $k_v$ be a compact subgroup of $\GL_2(\Q_v)$ such that
	$$\left\{\iota_A : A \in k_v\right\} \subset K_v.$$
	Assume moreover that $K_v=\Gamma_v$. Then $K / \Gamma$ is left invariant by $\Gamma_v$, hence for all
	$A \in k_v$ we have $\phi_k( \cdot A)=\phi_{\iota_A k}=\phi_k$.
	In particular, for $v=\infty$, we may take $k_v=O_2(\R)$, hence the claim.
\end{remark}
\begin{proof}
	We first prove that $\phi$ is well-defined.
	Suppose $n_1m_1k_1=n_2m_2k_2$.	
	In particular $k_2k_1^{-1}=(n_2m_2)^{-1}(n_1m_1) \in \Pk(\A) \cap K$.
	Therefore $\sigma_{\text K}(m_1)=\sigma_{\text K}(n_1m_1)=\sigma_{\text K}(n_2m_2k_2k_1^{-1})=\sigma_{\text K}(m_2)\sigma_{\text K}(k_2k_1^{-1})$.
	Then 
	\begin{align*}
		\omega_2 \circ t(m_1) \phi_{k_1}(\sigma_{\text K}(m_1)) &=  \omega_2 \circ t(m_1)\phi_{k_1}(\sigma_{\text K}(m_2)\sigma_{\text K}(k_2k_1^{-1})) \\
		&= \omega_2 \circ t(m_1) \omega_2 \circ t(k_1k_2^{-1}) \phi_{k_2}(\sigma_{\text K}(m_2)) = \omega_2 \circ t(m_2)\phi_{k_2}(\sigma_{\text K}(m_2)).
	\end{align*}
	Next we show that $\phi$ belongs indeed to ${\I_{\Pk}((\omega_2 \otimes \pi)_\nu)}^\Gamma$. 
	The fact that $\phi$ is invariant on the left by $\Nk(\A)\Mk(\Q)\Ak(\R)^\circ$ and the right invariance by $\Gamma$ are obvious from the definition. 
	The fact that $\phi$ is square integrable follows from
		\begin{align*}
		\int_K \int_{\Ak(\R)^+ \Mk(\Q) \backslash \Mk(\A)} &|\phi(mk)|^2dmdk
		=\int_K \int_{\Ak(\R)^+ \Mk(\Q) \backslash \Mk(\A)} |\phi_k(\sigma_{\text K}(m))|^2dmdk\\
		&=\sum_{k \in K / \Gamma} Vol(\Gamma) \int_{\R_{>0} \GL_2(\Q) \backslash \GL_2(\A)} \int_{\R_{>0}\Q^\* \backslash \A^\*}
		|\phi_k(x)|^2dtdx
		<\infty
	\end{align*}
	since $\phi_k$ is square integrable, $\R_{>0}\Q^\* \backslash \A^\*$ is compact and $K / \Gamma$ is finite.
	Finally, we need to show that for all $g=nmk$, the function $\phi_g : \Mk(\A) \to \C, m_1 \mapsto \phi(m_1g)$
	transform under $\Mk(\A)$ on the right according to $\omega_2 \otimes \pi$.
	Indeed, for $m_1 \in \Mk(\A)$ we have 	
	$$\phi_g(m_1)=\phi(m_1nmk)=\phi(\underbrace{m_1nm_1^{-1}}_{\in \Nk(\A)}m_1mk)=\omega_2 \circ t (m) \omega_2 \circ t(m_1) \phi_k(m_1)$$
	hence the claim since $\phi_k \in V_\pi$.
	
	As a last point, we show that ${\I_{\Pk}((\omega_2 \otimes \pi)_\nu)}^\Gamma$ consists exactly in such functions.
	Let $\phi \in {\I_{\Pk}((\omega_2 \otimes \pi)_\nu)}^\Gamma$. Define 
	$$\phi_k(A)=\phi(\iota_A k).$$
	Then it is clear that equation~(\ref{Klingennmk}) holds.
	As for condition~(\ref{Klingencompatibility}), note that if $\gamma=n_\gamma m_\gamma \in K \cap \Pk(\A)$ then
	\begin{align*}
		\phi_k(A\pi(\gamma))&=\phi(\iota_A\iota_{\pi(\gamma)}k)\\
		&=\phi \left(\iota_A \left[ \begin{smallmatrix}1&&&\\&t(\gamma)^{-1}&&\\&&1&\\&&&t(\gamma)\end{smallmatrix}\right]
		m_\gamma k \right)\\
		&=\omega_2 \circ t (\gamma^{-1}) \phi(\iota_A n_\gamma^{-1} \gamma k)\\
		&=\omega_2 \circ t (\gamma^{-1}) \phi(\underbrace{\iota_A n_\gamma^{-1} \iota_A^{-1}}_{\in \Nk} \iota_A \gamma k)\\
		&=\omega_2 \circ t (\gamma^{-1}) \phi_{\gamma k}(A).
	\end{align*}
	Finally, by definition of ${\I_{\Pk}((\omega_2 \otimes \pi)_\nu)}$ the function $m \mapsto \phi(mk)$ transforms 
	under $\Mk(\A)$ on the right according to $\omega_2 \otimes \pi$, from which follows $\phi_k$ transforms according to $\pi$. 
\end{proof}
Finally we prove the following
\begin{proposition}	\label{KlingenSpectralParameter}
Let $(\pi,V_\pi)$ occur in the discrete spectrum of $L^2(\R_{>0}\GL_2(\Q)\back \GL_2(\A))$ with central character $\omega_1$.
$\I_{\Pk}((\omega_2 \otimes \pi)_\nu)$ has a $K_\infty$-fixed vector if and only if $\pi$ has a $O_2(\R)$-fixed vector.
In this case, $\I_{\Pk}((\omega_2 \otimes \pi)_\nu)_\infty$ is generic if and only if $\pi_\infty$ is a principal series.
Finally if  $\pi_\infty$ is a spherical principal series with spectral parameter $s$ and $\nu \in i\Lie{a}_{\Mk}^*$
then $\I_{\Pk}((\omega_2 \otimes \pi)_\nu)_\infty$ 
is a principal series representation with spectral parameter $\nu+\nu_{\text{K}}(s)$, where $\nu_{\text{K}}(s)$ is the element of $\Lie{a}^*(\C)$ 
corresponding to the character $\diag{y^{\frac12}}{u}{y^{-\frac12}}{u^{-1}} \mapsto |y|^s$.
\end{proposition}
\begin{proof}
	The first claim follows immediately from Proposition~\ref{Klingensummand}.
	By the spectral decomposition for $\GL_2$, if $\pi$ has a $O_2(\R)$-fixed vector then $\pi_\infty$ is either 
	a character or a principal series.
	But representations induced from a character of the Klingen subgroup are not generic. This shows the second claim.
	Finally assume $\pi_\infty$ is a spherical principal series on $\GL_2$ with spectral parameter $s$.
	Then we might see $\pi_\infty$ as the representation of $\PGL_2(\R)$ induced from the character 
	$\chi_s:\mat{y^{\frac12}}{x}{}{\pm y^{{-\frac12}}} \mapsto \left|y\right|^s$,
	where $s$ is either an imaginary number or a real number with $0<|s|<\frac12$.
	Define the following subgroups:
	$N_1= \left[ \begin{smallmatrix}
		1	&	&*	&				\\ 
			&1	&	&				\\ 
			&	&1	&				\\
			&	&	& 1
	\end{smallmatrix}\right] $,
$A_1=\left\{\diag{y^{\frac12}}{1}{\pm y^{-\frac12}}{1}:y \neq 0\right\}$, $M_1=\{\iota_A : A \in \PGL_2(\R)\}$.
Note that $N_1\Nk=U$, $A_1\Ak(\R)=T(\R)$ and $M_1\Ak=\Mk$
We might view $\chi_s$ as a character of $A_1N_1$.
	Since $\omega_2$ is trivial on $\Ak(\R)$, inducing in stage, we get
	\begin{align*}
	\I_{\Pk}((\omega_2 \otimes \pi)_\nu)_\infty	&=\Ind_{\Pk(\R)}^{G(\R)}\left(I_{\Nk(\R)} \otimes e^\nu \otimes \pi_{\infty}\right)\\
	&=\Ind_{\Pk(\R)}^{G(\R)}\left(I_{\Nk(\R)} \otimes e^\nu \otimes \Ind_{A_1N_1}^{M_1}(\chi_s)\right)\\
	&=\Ind_{\Pk(\R)}^{G(\R)}\Ind_{B(\R)}^{\Pk(\R)}\left(I_{\Nk(\R)} \otimes I_{N_1} \otimes e^{\nu+\nu_{\text{K}}(s)}\right)\\
	&=\Ind_{B(\R)}^{G(\R)}(I_U \otimes  e^{\nu+\nu_{\text{K}}(s)})
	\end{align*}
Since $\nu \in i \Lie{a}^*$, by Lemma~5.1 of~\cite{Muic} this representation is irreducible.
\end{proof}

\subsubsection{Siegel Eisenstein series}~\label{PES}

Characters $\chi$ of $\As^+(\R)\As(\Q) \back \As(\A)$ that coincide with $\omega$ on $Z(\A)$ are in one-to-one 
correspondence with pairs $(\omega_1, \omega_2)$ of characters of $\R_{>0}\Q^\* \backslash \A^\*$  satisfying
$\omega_1 \omega_2^2 = \omega$, via 
$$\chi\left(\diag{u}{u}{tu^{-1}}{tu^{-1}}\right)=\omega_1(u)\omega_2(t).$$
For convenience, if $A \in \GL_2$, define 
$\iota_A=\mat{A}{}{}{ \trans{A}^{-1}} \in \Ms,$
and if $p=\mat{A}{*}{}{t\trans{A}^{-1}} \in \Ps$, define 
$\sigma_{\text S}(p)=A$, and $\sigma_S(nm)=\sigma_S(m)$ 
\begin{lemma}\label{SiegelInducing}
	Let $\chi=(\omega_1,\omega_2)$ with $\omega_1\omega_2^2=\omega$.
	The irreducible representations with central character $\chi$ occurring in $R_{\Ms,\text{disc}}$ are twists $\omega_2 \otimes \pi$, where $\pi$ 
	occurs in the discrete spectrum of $L^2(\R_{>0}\GL_2(\Q)\back \GL_2(\A))$ and has central character $\omega_1$.
\end{lemma}
\begin{proof}
Similar as Lemma~\ref{KlingenInducing} with trivial modifications where required.
\end{proof}

\begin{proposition}\label{Siegelsummand}
	Let $\chi=(\omega_1,\omega_2)$ with $\omega_1\omega_2^2=\omega$.
		Let $(\pi,V_\pi)$ occur in the discrete spectrum of $L^2(\R_{>0}\GL_2(\Q)\back \GL_2(\A))$ with central character $\omega_1$.
	Consider $(\phi_k)_{k \in K / \Gamma}$ such that
	\begin{enumerate}
		\item for all $k$, $\phi_k \in V_\pi$,
		\item \label{Siegelcompatibility} if $\gamma \in K \cap \Ps(\A)$ then for all $k$, 
		$\phi_k(\cdot \sigma_{\text S}(\gamma))=\omega_2 \circ \mu(\gamma^{-1})\phi_{\gamma k}$.
	\end{enumerate} 
	Then the function on $G(\A)$ given for  $n \in \Ns(\A), m \in \Ms(\A), k \in K$ by
	\begin{equation}\label{Siegelnmk}
		\phi(nmk)= \omega_2 \circ \mu(m) \phi_k(\sigma_{\text S}(m)),
	\end{equation}
	is well-defined and belongs to ${\I_{\Ps}((\omega_2 \otimes \pi)_\nu)}^\Gamma$. 
	Moreover, every function in ${\I_{\Ps}((\omega_2 \otimes \pi)_\nu)}^\Gamma$ has this shape.
\end{proposition}
\begin{remark}
	Similarly as Remark~\ref{KlingenMaassForm}, condition~(\ref{Siegelcompatibility}) implies that each $\phi_k$ is right-$\O_2(\R)$-invariant 
	(and hence must be an adelic Maa{\ss} form or a character).
\end{remark}
\begin{proof}
	Same proof as Proposition~\ref{Klingensummand}, with trivial modifications where required.
\end{proof}
\begin{proposition}	\label{SiegelSpectralParameter}
	Let $(\pi,V_\pi)$ occur in the discrete spectrum of $L^2(\R_{>0}\GL_2(\Q)\back \GL_2(\A))$ with central character $\omega_1$.
	$\I_{\Ps}((\omega_2 \otimes \pi)_\nu)$ has a $K_\infty$-fixed vector if and only if $\pi$ has a $O_2(\R)$-fixed vector.
	In this case, $\I_{\Ps}((\omega_2 \otimes \pi)_\nu)_\infty$ is generic if and only if $\pi_\infty$ is a principal series.
	Finally if  $\pi_\infty$ is a spherical principal series with spectral parameter $s$ and $\nu \in i\Lie{a}_{\Ms}^*$
	then $\I_{\Pk}((\omega_2 \otimes \pi)_\nu)_\infty$ 
	is a principal series representation with spectral parameter $\nu+\nu_{\text{S}}(s)$, where $\nu_{\text{S}}(s)$ is the element of $\Lie{a}^*(\C)$ 
	corresponding to the character $\diag{y^{\frac12}u}{y^{-\frac12}u}{y^{-\frac12}u^{-1}}{y^{\frac12}u^{-1}} \mapsto |y|^s$.
\end{proposition}
\begin{proof}
	Same proof as Proposition~\ref{KlingenSpectralParameter}, with trivial modifications where required.
\end{proof}

\subsection{Spectral expansion of the kernel}
We now give the spectral expansion of the kernel. 
\begin{definition}\label{ONBase}
	For each standard parabolic $P$ we choose an orthonormal basis $\B_{P}$ of $\H_P(\omega)$ such that 
	\begin{enumerate}
		\item \label{DecompIrred} if $R_{M_P,\text{disc}}=\bigoplus_\pi \pi$ is the decomposition of the restriction of the right regular representation of $M_P(\A)$
		on the subspace of $L^2(M_P(\Q)A_P^0(\R) \back G(\A))$ that decompose discretely, then 
		$\B_P=\bigcup_\pi \B_{\pi}$, where each $\B_\pi$ is a basis of the space of the corresponding induced representation $\I_P(\pi_\nu)$,
		(note that this space does not depend on~$\nu$). 
		\item \label{factor} for each representation $\pi=\bigotimes_v \pi_v$ as above, for each place $v$ there is an orthonormal basis $\B_{\pi,v}$ of
		the local representation $\pi_v$ such that  $\B_{\pi}$ consists in factorizable vectors 
		$\phi=\bigotimes_{v \le \infty} \phi_v$ where each $\phi_v$ belongs to the corresponding $\B_{\pi,v}$.
		\item \label{Ktypes} for each representation $\pi_v$, we have 
		$\B_{\pi,v} = \bigcup_{\tau} \B_{\pi,v, \tau}$, where the union is over the irreducible representations $\tau$ of $\Gamma_v$,
		and $\B_{\pi, v,\tau}$ is a basis of the space of $\pi_v$ consisting of vectors $\phi$ satisfying $\pi_v(\gamma)\phi=\tau(\gamma) \phi$
		for all $\gamma \in \Gamma_v$.
	\end{enumerate}
\end{definition}
Note that conditions~(\ref{factor}) and~(\ref{Ktypes}) imply in particular that elements of $\B_P$ are in $\H_P^0$.
\begin{definition}
For each standard parabolic $P$ and for each irreducible representation $\pi$ occuring in $R_{M_P,\text{disc}}$,
define $\B_{\pi, 1}$ to be the subset of $\B_{\pi}$ consisting in vectors $\phi$ whose each local component
$\phi_v$ belongs to $\B_{\pi,v, 1}$, and set $\B_P^\Gamma=\bigcup_\pi \B_{\pi, 1}$.
If $\chi$ is a character of $A_P(\A)$, define
$$\Gen_P(\chi)=\bigcup_{\pi} \B_{\pi,1},$$ 
where the union runs over representations $\pi$ with central character $\chi$ and such that the induced representations
$\I_P(\pi_\nu)$ are generic.
\end{definition}
If $u \in \H_P(\omega)$, define  $$\I_P(\nu ,f)u=\int_{\overline{G(\A)}}f(y)\I_P(\nu ,y)udy.$$
\begin{proposition}\label{globalevalue}
	Let $\nu \in i\Lie{a}_p^*$.
	Let $u \in \B_P$.
	Then either $\I_P(\nu ,f)u=0$ or $u \in \B_P^\Gamma $.
	In the latter case, say $u \in \B_\pi$. Then if $\pi$ is generic we have 
	$$\I_P(\nu ,f)u=\lambda_f(u,\nu)u,$$
	where $\lambda_f(u,\nu)=\lambda_{f_{\infty}}(u,\nu)\lambda_{f_{\text{fin}}}(u,\nu)$,
	and 
	$$\lambda_{f_{\infty}}(u,\nu)=
	\begin{cases}
		\tilde{f_\infty}(\nu) \text{ if } P=B,\\
		\tilde{f_\infty}(\nu+\nu_{\text{K}}(s_u)) \text{ if } P=\Pk \text{ and } \pi_\infty \text{ has spectral parameter } s_u,\\
		\tilde{f_\infty}(\nu+\nu_{\text{S}}(s_u)) \text{ if } P=\Ps \text{ and } \pi_\infty \text{ has spectral parameter } s_u,\\
		\tilde{f_\infty}(\nu_u) \text{ if } P=G \text{ and } \pi_\infty \text{ has spectral parameter } \nu_u,\\
	\end{cases}$$
and, following notations of Proposition~\ref{localevalue}, $\lambda_{f_{\text{fin}}}(u,\nu)$ is the eigenvalue of the Hecke operator
 $$\bigotimes_{\Gamma_p=G(\Z_p)}\overline{\pi_{p,\nu}}(\tilde{f_p}).$$
\end{proposition}
\begin{remark}
	If $P=G$ then $\Lie{a}_P=\{0\}$ and  $\I_P(\nu ,f)=R(f)$.
\end{remark}
\begin{proof}
	This is a combination of Propositions~\ref{localevalue},~\ref{archimedeanevalue}, Lemma~\ref{BorelSpectralParameter} and
	Propositions~\ref{KlingenSpectralParameter} and~\ref{SiegelSpectralParameter}.
\end{proof}
The following statement~\cite{ArthurSpectralExpansion}*{pages 928-935} may be viewed as a rigorous version of the informal discussion
in Section~\ref{basickernel}. 
\begin{lemma}\label{UnprovedLemma}
	Let $f$ as in Assumption~\ref{testfunction}. Then we have a pointwise equality
	$$K_f(x,y)=\sum_{P} n_P^{-1}  \int_{i\Lie{a}_{P}^*}\sum_{u \in \B_P}E(x, \I_P(\nu ,f)u,\nu )\overline{E(y,u,\nu)}d\nu.$$
		Here, $n_G=1,$ $n_B=8$, $n_{\Pk}=2$ and $n_{\Ps}=2$.
\end{lemma}
However, for the later purpose of interchanging integration order, we want to show that the above expressions for the kernel converge absolutely. 
To this end, we need the following stronger statement.
\begin{proposition}\label{ACV}
	Let $f$ as  in Assumption~\ref{testfunction}. Then the following expression defines a continuous function in the variables $x$ and $y$
	$$K_{\text{abs}}(x,y)=\sum_{P} n_P^{-1}  \int_{i\Lie{a}_{P}^*}\sum_{u \in \B_P}|E(x, \I_P(\nu ,f)u,\nu )\overline{E(y,u,\nu)}|d\nu.$$
\end{proposition}	
We do not give a proof of this proposition here, as a similar statement was proven in the setting of $\GL_2$ in \S~6 of \cite{KL}, the
proof thereof can be directly adapted.
By combining it with Lemmas~\ref{Borelsummand},~\ref{Klingensummand},~\ref{Siegelsummand} and Proposition~\ref{globalevalue},
we obtain the following corollary.
\begin{corollary}\label{KernelSpectralForm}
	Let $f$ as in Assumption~\ref{testfunction}.
	Then we have a pointwise equality 
	$$K_{f}(x,y)=K_{\text{disc}}(x,y)+K_{\text{B}}(x,y)+K_{\text{K}}(x,y)+K_{\text{S}}(x,y)+K_{\text{ng}}(x,y),$$
	where 
	$$K_{\text{disc}}(x,y)=\sum_{u \in \Gen_G(\omega)}\tilde{f_\infty}(\nu_u)\lambda_{f_{\text{fin}}}(u)u(x)\overline{u(y)},$$
	$$K_{\text{B}}(x,y)=\frac18 \sum_{\omega_1\omega_2\omega_3^2=\omega}\sum_{u \in \Gen_B(\omega_1,\omega_2,\omega_3)} 
	\int_{i\Lie{a}^*}\tilde{f_\infty}(\nu)\lambda_{f_{\text{fin}}}(u,\nu)E(x, u,\nu )\overline{E(y,u,\nu)}d\nu.$$
	$$K_{\text{K}}(x,y)=\frac12 \sum_{\omega_1\omega_2=\omega}\sum_{u \in \Gen_{\Pk}(\omega_1,\omega_2)}
	\int_{i\Lie{a}_\text{K}^*}\tilde{f_\infty}(\nu+\nu_{\text{K}}(s_u))\lambda_{f_{\text{fin}}}(u,\nu)E(x, u,\nu )\overline{E(y,u,\nu)}d\nu,$$
	$$K_{\text{S}}(x,y)=\frac12 \sum_{\omega_1\omega_2^2=\omega}\sum_{u \in \Gen_{\Ps}(\omega_1,\omega_2)} 
	\int_{i\Lie{a}_\text{S}^*}\tilde{f_\infty}(\nu+\nu_{\text{S}}(s_u))\lambda_{f_{\text{fin}}}(u,\nu)E(x, u,\nu )\overline{E(y,u,\nu)}d\nu,$$
	and all the automorphic forms involved in $K_{\text{ng}}$ are not generic.
\end{corollary} 
Actually, no automorphic form from the residual spectrum is generic, as shown by the following lemma. 
Thus $K_{\text{disc}}$ consists only in elements from the cuspidal spectrum.

\begin{lemma}
	Let $(\pi,V_\pi)$ be any irreducible representation occurring in the residual spectrum of 
	$L^2(Z(\R)G(\Q) \backslash G(\A), \omega)$. Then $\pi$ is non generic.
\end{lemma}

\begin{proof}
	We will rely on results of Kim that describe the residual spectrum of $\Sp_4$. 
	Thus we first need to show that the $\res \pi$ given by Definition~\ref{DefOfRes} belongs to the residual spectrum of $\Sp_4(\A)$. 
	First, $\res \pi$ occurs in the discrete spectrum of $L^2(\Sp_4(\Q) \backslash \Sp_4(\A))$, because there are only finitely many possibilities 
	for the Archimedean component of any irreducible representation occurring in $\res \pi$.
	Moreover  $\res \pi$ and is not cuspidal by Lemma~\ref{cuspidalrestriction}.
	Hence $\res \pi$ belongs to the residual spectrum of $\Sp_4(\A)$, as claimed. 
	In view of Lemma~\ref{genericrestriction}, it suffices to prove that the residual spectrum of $\Sp_4(\A)$ is not generic.
	By Theorem~3.3 and Remark~3.2 of~\cite{Kim}, the representations occurring from poles of Siegel Eisenstein
	series are non generic. Similarly, by Theorem~4.1 and Remark~4.2 of~\cite{Kim}, the representations occurring
	from poles of Klingen Eisenstein series are non generic. 
	Finally, by~\cite{Kim}~\S~5.3,  irreducible representations $\pi$ occurring from the poles of Borel Eisenstein series are described as follows.
	On the one hand, we have the space of constant functions, which is clearly not generic.
	On the other hand, for every non-trivial quadratic gr\"ossencharacter $\mu$ of $\Q$ we have a representation $B(\mu)$ 
	whose local components are irreducible subquotients of the induced representation $\Ind_B^{\Sp_4}(|\cdot|_v\mu_v \times \mu_v)$.
	Therefore, in the terminology of~\cite{RS2}*{\S~2.2}, for all prime~$p$, $\pi_p$ belongs to Group~V	if $\mu_p \neq 1$, and to Group~VI if $\mu_p=1$.
	Now by Table~A.2 of~\cite{RS2}, we see that the only generic representations in Group~V and~VI are those from~Va and VI~a.
	But Table~A.12 shows that neither of these have a $K_p$-fixed vector. Since almost all $\pi_p$ contain a $K_p$-fixed vector, at least one local component of
	$\pi$ must be non-generic, and thus $\pi$ is not globally generic.
\end{proof}

\subsection{The spectral side of the trace formula}
	Let $\psi_1=\psi_{\m_1}, \psi_2=\psi_{\m_2}$ be generic characters of $U(\A) / U(\Q)$.
Fix $t_1, t_2 \in A^0(\R)$ and consider the basic integral
\begin{equation}\label{BasicIntegral}
	I=\int_{(U(\Q) \backslash U(\A))^2} K_f(xt_1,yt_2) \overline{\psi_{\m_1}(x)}\psi_{\m_2}(y)dx dy.
\end{equation}
Our goal is to compute it in two different ways -- using the spectral decomposition of the kernel 
$K_f$ on the one hand, and its expression as a series together with the Bruhat decomposition on the other hand.
The latter will constitute the geometric side and will be addressed in Section~\ref{GeometricSide}. We now focus on the former.
Using the spectral expansion of the kernel $K_f$ given by Lemma~(\ref{UnprovedLemma}), we can evaluate the basic integral~(\ref{BasicIntegral}) as
$$I=\int_{(U(\Q) \backslash U(\A))^2} \sum_{P} n_P^{-1}  \int_{i\Lie{a}_{P}^*}\sum_{u \in \B_P}E(xt_1, \I_P(\nu ,f)u,\nu )\overline{E(yt_2,u,\nu)}d\nu \overline{\psi_{\m_1}(x)}\psi_{\m_2}\}(y)dx dy.$$
By Proposition~\ref{ACV}, this expression is absolutely integrable since $(U(\Q) \backslash U(\A))^2$ is compact.
Thus we may interchange integration order, thus obtaining the Whittaker coefficients of the automorphic forms involved here.
By Corollary~\ref{KernelSpectralForm}, we get a discrete contribution and a residual contribution, and a continuous contribution 
-- which itself splits into the contribution of the various parabolic classes.
Thus the spectral side of the Kuznetsov formula is given as follows.
\begin{proposition}
	We have $I=\frac{1}{(\m_{1,1}\m_{2,1})^4|\m_{1,2}\m_{2,1}|^3}\left(\Sigma_{\text{disc}}+\Sigma_B+\Sigma_K+\Sigma_S\right)$, where
		$$\Sigma_{\text{disc}}=\sum_{u \in \Gen_G(\omega)}\tilde{f_\infty}(\nu_u)\lambda_{f_{\text{fin}}}(u)\W_{\psi}(u)(t_1t_{\m_1}^{-1})\overline{\W_{\psi}(u)}(t_2t_{\m_2}^{-1}),$$
	\begin{align*}
		\Sigma_{\text{B}}=\frac18 \sum_{\omega_1\omega_2\omega_3^2=\omega}\sum_{u \in \Gen_B(\omega_1,\omega_2,\omega_3)} &
		\int_{i\Lie{a}^*}\tilde{f_\infty}(\nu) \lambda_{f_{\text{fin}}}(u,\nu)\\
		&\times\W_{\psi}(E(\cdot, u,\nu ))(t_1t_{\m_1}^{-1})\overline{\W_{\psi}(E(\cdot,u,\nu))}(t_2t_{\m_2}^{-1})d\nu.
	\end{align*}
\begin{align*}
		\Sigma_{\text{K}}=\frac12 \sum_{\omega_1\omega_2=\omega}\sum_{u \in \Gen_{\Pk}(\omega_1,\omega_2)}&
	\int_{i\Lie{a}_\text{K}^*}\tilde{f_\infty}(\nu+\nu_{\text{K}}(s_u))\lambda_{f_{\text{fin}}}(u,\nu)\\
	&\times \W_{\psi}(E(\cdot, u,\nu ))(t_1t_{\m_1}^{-1})\overline{\W_{\psi}(E(\cdot,u,\nu))}(t_2t_{\m_2}^{-1})d\nu,
\end{align*}
\begin{align*}
		\Sigma_{\text{S}}=\frac12 \sum_{\omega_1\omega_2^2=\omega}\sum_{u \in \Gen_{\Ps}(\omega_1,\omega_2)} &
	\int_{i\Lie{a}_\text{S}^*}\tilde{f_\infty}(\nu+\nu_{\text{S}}(s_u))\lambda_{f_{\text{fin}}}(u,\nu)\\
	&\times \W_{\psi}(E(\cdot, u,\nu ))(t_1t_{\m_1}^{-1})\overline{\W_{\psi}(E(\cdot,u,\nu))}(t_2t_{\m_2}^{-1})d\nu.
\end{align*}
\end{proposition}	

\section{The geometric side of the trace formula}\label{GeometricSide}
Breaking the sum~(\ref{TheKernel}) over $U(\Q) \times U(\Q)$ orbits
leads to a sum over representatives of the double cosets of $U \backslash G / U$ of
orbital integrals. Specifically, set $H=U \times U$, acting on $G$ by
$$(x,y) \cdot \delta = x^{-1} \delta y,$$
and denote by $H_\delta$ the stabilizer of $\delta$.
Since $f$ has compact support, the infinite sum $\sum_{\delta \in \GmodZ(\Q)} |f(t_1^{-1}x^{-1} \delta yt_2)|$ is
in fact locally finite and hence 
defines a continuous function in $x$ and $y$ on the compact set 
$H(\Q) \backslash H(\A)$. Thus we may interchange summation and integration order, getting
\begin{align*}
	I &= 
	\int_{H(\Q) \backslash H(\A)} \sum_{\delta \in \GmodZ(\Q)} f(t_1^{-1}x^{-1} \delta yt_2)\overline{\psi_{\m_1}(x)}\psi_{\m_2}(y)dx dy\\
	&=\sum_{\delta \in \GmodZ(\Q)} \int_{H(\Q) \backslash H(\A)}f(t_1^{-1}x^{-1} \delta yt_2)\overline{\psi_{\m_1}(x)}\psi_{\m_2}(y)dx dy\\
	&= \sum_{\delta \in U(\Q)\backslash \GmodZ(\Q)/U(\Q)} I_{\delta}(f),
\end{align*}
where 
\begin{equation}\label{orbital}
	I_{\delta}(f)=\int_{H_\delta(\Q) \backslash H(\A)} f(t_1^{-1}x^{-1} \delta yt_2)\overline{\psi_{\m_1}(x)}\psi_{\m_2}(y)d(x,y),
\end{equation} 
and $d(x,y)$ is the quotient measure on $H_\delta(\Q) \backslash H(\A)$.
Using the Bruhat decomposition $G=B \Omega B= \coprod_{\sigma \in \Omega}U \sigma T U$, we have 
\begin{equation} \label{BruhatmodZ}
	U \backslash \GmodZ / U = \coprod_{\sigma \in \Omega} \sigma \TmodZ,
\end{equation}
where $\TmodZ= T / Z$.
We can then compute separately the contribution from each element from the Weyl group.
Writing $H(\A)=H_\delta(\A) \times (H_\delta(\A) \backslash H(\A))$, we can factor out the integral
of $\overline{\psi_{\m_1}}\otimes\psi_{\m_2}$ over the compact group $H_\delta(\Q) \backslash H_\delta(\A)$
in~(\ref{orbital}). Therefore, $I_\delta(f)$ vanishes unless the character
$ \overline{\psi_{\m_1}} \otimes\psi_{\m_2}$ is trivial on $H_\delta(\A)$. Following
Knightly and Li, we shall call the orbits $H \cdot \delta$ such that
$\overline{\psi_{\m_1}} \otimes \psi_{\m_2}$ is trivial on $H_\delta(\A)$ {\bf relevant}.

\subsection{Relevant orbits}
In order to characterize the relevant orbits, let us introduce a bit of notation.
A set of representatives of $T(\Q) / Z(\Q)$ is given by the elements 
\begin{equation}\label{diagrepres}
	\delta_1 \doteq \diag{d_1}{1}{d_2}{d_1d_2}, d_1,d_2 \in \Q^\*.
\end{equation}
For each $\sigma \in \Omega$, the corresponding set of representatives of $\sigma \TmodZ(\Q)$ in~(\ref{BruhatmodZ})
is given by elements of the form 
\begin{equation}\label{ds}
	\delta_\sigma=\sigma \delta_1,
\end{equation}
and $H_{\delta_\sigma}(\A)$ consists in pairs 
$(u,\delta_\sigma^{-1}u\delta_\sigma)=(u,\delta_1^{-1}\sigma^{-1}u\delta_1\sigma)$ such that both component lie in $U(\A)$.
Since conjugation by $\delta_1$ preserves $U(\A)$,
the condition that the second component lies in $U(\A)$ is equivalent to $u \in U(\A)$ and $\sigma^{-1} u \sigma \in U(\A)$.
We accordingly make the following definition.
\begin{definition}
	For $\sigma \in \Omega,$ define
	\begin{equation}\label{Usigma}
		U_\sigma(\A) =\{x \in U(\A) :\sigma^{-1} x \sigma \in U(\A)\},
	\end{equation}
	and
	\begin{equation}\label{Dsigma}
		D_\sigma(\A) = U_\sigma(\A) \times \sigma^{-1} U_\sigma(\A) \sigma 
	\end{equation}
Then we have 
\begin{equation}\label{paramHd}
	H_{\delta_\sigma}(\A)=\{(u,\delta_\sigma^{-1}u\delta_\sigma): u \in U_\sigma(\A)\} \subset D_\sigma(\A).
\end{equation}
\end{definition}
\begin{lemma}\label{relevantorbits}
The relevant orbits are the ones corresponding to the following elements:
\begin{itemize}
	\item $\sigma=1$ with $\delta_1= {t_{\m_2}}^{-1}t_{\m_1}=
	\diag{\frac{\m_{1,1}}{\m_{2,1}}}{1}{\frac{\m_{1,1}\m_{1,2}}{\m_{2,1}\m_{2,2}}}{\frac{\m_{1,1}^2\m_{1,2}}{\m_{2,1}^2\m_{22}}}$,
	\item $\sigma=s_1s_2s_1$ with $\delta_1$ satisfying $d_1\m_{1,2}=d_2\m_{2,2}$,
	\item $\sigma=s_2s_1s_2$ with $\delta_1$ satisfying $\m_{1,1}=-d_1\m_{2,1}$,
	\item $\sigma=s_1s_2s_1s_2=J$ with no condition on $\delta_1$. 
\end{itemize}
\end{lemma}
\begin{proof}
		For each representative $\delta_\sigma$ as in~(\ref{ds}), let us fix $u_1 \in U_\sigma(\A)$, and compute 
	$\delta_\sigma^{-1}u_1\delta_\sigma$ in order to determine under which condition $\psi_{\m_1} \otimes \overline{\psi_{\m_2}}$
	is trivial on $H_{\delta_\sigma}(\A)$.
	For $\sigma=1$, we have $U_\sigma=U$, hence we may take $u_1=\left[
	\begin{smallmatrix}
		1	& 		&	c	& 	a-cx	\\
		x 	& 	1	& a	& b			\\
		& 		&	 1	& 	-x		\\
		&		&		&	1			\\
	\end{smallmatrix}
	\right]$. Then we have
	$\delta^{-1}u_1\delta=\left[
	\begin{smallmatrix}
		1	& 		&	c\frac{d_2}{d_1}	& 	(a-cx)d_2	\\
		xd_1 	& 	1	& ad_2	& bd_1d_2			\\
		& 		&	 1	& 	-xd_1		\\
		&		&		&	1			\\
	\end{smallmatrix}
	\right]$ Thus, by~(\ref{gencharconj}), the condition that $\psi_{\m_1} \otimes \overline{\psi_{\m_2}}$ be trivial on 
	$H_{\delta_1}(\A)$ is equivalent to $\delta_1={t_{\m_2}}^{-1}t_{\m_1}.$
	
	For $\sigma=s_1$, we have $U_\sigma(\A)=\left\{\left[	\begin{smallmatrix}
		1	& 		&	c	& 	a	\\
		 	& 	1	& a	& b			\\
		& 		&	 1	& 			\\
		&		&		&	1			\\
	\end{smallmatrix}\right]: a,b,c \in \A\right\}$, and if $u_1=\left[	\begin{smallmatrix}
	1	& 		&	c	& 	a	\\
		& 	1	&	 a	& b			\\
		& 		&	 1	& 			\\
		&		&		&	1			\\
\end{smallmatrix}\right]$, then 	$\delta^{-1}u_1\delta=\left[
\begin{smallmatrix}
1	& 		&	b\frac{d_2}{d_1}	& 	ad_2	\\
 	& 	1	& ad_2	& cd_1d_2			\\
	& 		&	 1	& 			\\
	&		&		&	1			\\
\end{smallmatrix}
\right],$
hence the condition that $\psi_{\m_1} \otimes \overline{\psi_{\m_2}}$ be trivial on 
$H_{\delta_{s_1}}(\A)$ is equivalent to 
$\theta\left(\m_{1,2}c-\m_{2,2}\frac{d_2}{d_1}b\right)=1$
for all $b,c \in \A$, which is equivalent to $\m_{1,2}=\m_{2,2}=0$ and thus contradicts the fact that $\psi_{\m_1}$ and 
$\psi_{\m_2}$ are generic.

Similar calculations show that $\sigma=s_2,s_1s_2$ and $s_2s_1$ yield no relevant orbit.

For $\sigma=s_1s_2s_1$  we have $U_\sigma(\A)=\left\{\left[	\begin{smallmatrix}
	1	& 		&	c	& 		\\
		& 	1	& 		& 			\\
		& 		&	 1	& 			\\
		&		&		&	1			\\
\end{smallmatrix}\right]: c \in \A\right\}$,
and if $u_1=\left[	\begin{smallmatrix}
	1	& 		&	c	& 		\\
	& 	1	& 		& 			\\
	& 		&	 1	& 			\\
	&		&		&	1			\\
\end{smallmatrix}\right]$
then we have $\delta^{-1}u_1\delta=\left[
\begin{smallmatrix}
	1 	& 	& c\frac{d_2}{d_1}	&		\\
		& 1	&					& 		\\
		& 	& 		1			& 		\\
		&	& 					&	1	\\
\end{smallmatrix}
\right]$,
hence the condition that $\psi_{\m_1} \otimes \overline{\psi_{\m_2}}$ be trivial on 
$H_{\delta_{s_{121}}}(\A)$ is equivalent to 
$\theta\left(\left(\m_{1,2}-\m_{2,2}\frac{d_2}{d_1}\right)c\right)=1$
for all $c \in \A$. This is equivalent to $d_1\m_{1,2}=d_2\m_{2,2}$.

The calculation for $\sigma=s_2s_1s_2$ is similar.
Finally,  for $\sigma=s_1s_2s_1s_2=J$ the long Weyl element, 	$H_{\delta_{s_{1212}}}(\A)$ is trivial.
\end{proof}
A case by case calculation also shows the following.
\begin{lemma}\label{Usigmarelevant}
Let $\sigma \in \Omega$. Then there exists $\delta \in \TmodZ(\Q)$ such that the orbit of $\delta_\sigma=\sigma\delta$ is relevant 
if and only if $U_\sigma=\{u \in U : \sigma^{-1}u\sigma=u\}.$
\end{lemma}
In the sequel, we shall call such elements of the Weyl group {\bf relevant} as well.
In particular, by definition of the relevant orbits, and by~(\ref{paramHd}), we have the following.
\begin{corollary}\label{psionUsigma}
Suppose that the orbit of $\delta_\sigma=\sigma\delta$ is relevant. 
Then for all $u \in U_\sigma$ we have $\psi_{\m_2}(\delta^{-1}u\delta)=\psi_{\m_1}(u).$
\end{corollary}

\subsection{General shape of the relevant orbital integrals}
\begin{lemma}\label{quotientmap}
	For each $\delta_1 \in \TmodZ(\Q)$ and $\sigma \in \Omega$, the map 
	\begin{align*}
		\varphi : D_\sigma(\A) & \to D_\sigma(\A) \\
		(u_1,u_2) & \mapsto (u_1, \delta_\sigma^{-1}u_1^{-1}\delta_\sigma u_2).
	\end{align*}
	induces a bijective map
	\begin{align*}
		H_{\delta_\sigma} (\Q) \back D_\sigma(\A) \to & \left(U_\sigma(\Q) \times \{1\}\right) \back D_\sigma(\A) \\ &\cong \left(U_\sigma(\Q)\back U_\sigma(\A)\right)\times \left( \sigma^{-1} U_\sigma(\A) \sigma \right)
	\end{align*}
preserving the quotient measures.
\end{lemma}

\begin{proof}
	To prove $\varphi$ is well defined it is sufficient to prove that for any
	$(u_1,u_2)\in  U_\sigma(\A) \times \sigma^{-1} U_\sigma(\A)$ we have 
	$\varphi_2(u_1,u_2)=\delta_\sigma^{-1}u_1^{-1}\delta_\sigma u_2 \in \sigma^{-1} U_\sigma(\A)$. 
	This is equivalent to the condition $\sigma \varphi_2(u_1,u_2) \sigma^{-1} \in U_\sigma(\A)$, which in turn is equivalent to
	$$\begin{cases}
		\sigma \varphi_2(u_1,u_2) \sigma^{-1} \in U(\A) \\
		\varphi_2(u_1,u_2) \in U(\A).
	\end{cases}$$
	But $\varphi_2(u_1,u_2)=\delta_1^{-1} \sigma^{-1} u_1^{-1} \sigma \delta_1 u_2$, and since $u_1 \in U_\sigma(\A)$,
	we have $\sigma^{-1} u_1^{-1} \sigma \in U(\A)$ and it follows $\varphi_2(u_1,u_2) \in U(\A)$ as desired.
	On the other hand, 
	\begin{align*}
		\sigma \varphi_2(u_1,u_2) \sigma^{-1} &= \sigma \delta_1^{-1} \sigma^{-1} u_1^{-1} \sigma \delta_1 u_2 \sigma^{-1}\\
		&=(\sigma \delta_1 \sigma^{-1})^{-1} u_1^{-1} (\sigma \delta_1 \sigma^{-1}) \sigma u_2 \sigma^{-1}
	\end{align*}
	By definition of the Weyl group, $\sigma \delta_1 \sigma^{-1} \in T(\A)$ so 
	$(\sigma \delta_1 \sigma^{-1})^{-1} u_1^{-1} (\sigma \delta_1 \sigma^{-1}) \in U(\A)$. Furthermore, 
	$ \sigma u_2 \sigma^{-1} \in U_\sigma(\A) \subset U(\A)$ and it also follows that $\sigma \varphi_2(u_1,u_2) \sigma^{-1} \in U(\A).$
	
	Next, for $h=(h_1,h_2) \in H_{\delta_\sigma}(\Q)$, we clearly have $\varphi(h)=(h_1,1)$, and 
	\begin{align*}
		\varphi(h(u_1,u_2))&=(h_1u_1, \delta_\sigma^{-1} u_1^{-1} \underbrace{h_1^{-1} \delta_\sigma h_2}_{=\delta_\sigma} u_2)\\
		&=\varphi(h)\varphi(u_1,u_2).
	\end{align*}
	Finally if we define $\psi(u_1,u_2)=(u_1^{-1}, u_2)$, then
	$\psi \circ \varphi$ is an involution, and in particular $\varphi$ is bijective, which establishes the lemma.
\end{proof}

\begin{corollary}\label{corqoutientmap}
	Let $\delta_1 \in \TmodZ(\Q)$ and $\sigma$ be a relevant element of the Weyl group.
	We have a measure preserving map 
	\begin{align*}
		\varphi :H_{\delta_\sigma}(\Q) \back H(\A) & \to 
		\left(U_\sigma(\Q) \back U_\sigma(\A)\right) \times \left(U_\sigma(\A) \back U(\A) \right)
		\times  U(\A) \\
		(x,y) & \mapsto 
		\left(U_\sigma(\Q) u_1, U_\sigma(\A) u_2, u_3\right)
	\end{align*}
with $u_1u_2=x$ and $u_3=\ds^{-1} u_1^{-1} \ds y$.
\end{corollary}
\begin{remark}
	The assumption that $\sigma$ is relevant is not really needed here, but it simplifies slightly the proof.
\end{remark}
\begin{proof}
Let $\overline{U_\sigma}=U \cap \sigma^{-1} \trans{U} \sigma$. Then the quotient space $U_\sigma \back U$ may be identified with $\overline{U_\sigma}$,
and the map $U_\sigma \times \overline{U_\sigma}, (u_\sigma,u_1) \mapsto u_\sigma u_1$ preserves the Haar measures.
Define $\overline{D_\sigma}=\overline{U_\sigma} \times \overline{U_\sigma}$.
Using that $\sigma$ is relevant and hence, by Lemma~\ref{Usigmarelevant}, that $D_\sigma(\A)=U_\sigma(\A) \times U_\sigma(\A)$, we obtain a measure preserving map 
\begin{align*}
	 H_{\delta_\sigma}(\Q) \back H(\A) & \to \left(H_{\delta_\sigma}(\Q) \back D_\sigma(\A) \right) \times \overline{D_\sigma}\\
	H_{\delta_\sigma}(\Q)(x,y) & \mapsto (H_{\delta_\sigma}(\Q) (x_\sigma,y_\sigma),(x_1,y_1)).
\end{align*}
	Composing the first coordinate with the map obtained in Lemma~\ref{quotientmap}, we get a measure preserving map
\begin{align*}
	 H_{\delta_\sigma}(\Q) \back H(\A) & \to \left(\left(U_\sigma(\Q) \times \{1\}\right) \back D_\sigma(\A) \right) \times  \overline{D_\sigma}\\
	H_{\delta_\sigma}(\Q)(x,y) & \mapsto \left(\left(U_\sigma(\Q) \times \{1\}\right) (x_\sigma,\ds^{-1}x_\sigma^{-1} \ds y_\sigma),(x_1, y_1)\right).
\end{align*}
Finally, composing with $U_\sigma(\A) \times \overline{U_\sigma} \to U(\A), (y_\sigma,y_1) \mapsto y_\sigma y_1$ we obtain
\begin{align*}
	H_{\delta_\sigma}(\Q) \back H(\A) & \to \left(U_\sigma(\Q) \back U_\sigma(\A)\right) \times \overline{U_\sigma}(\A) \times   U(\A)\\
	H_{\delta_\sigma}(\Q)(x,y) & \mapsto \left(U_\sigma(\Q)x_\sigma,x_1,\ds^{-1}x_\sigma^{-1} \ds y\right).
\end{align*}
\end{proof}

\begin{proposition}\label{relevantintegrals}
	Let $H \cdot \delta_\sigma$ be a relevant orbit.
	Then the integral~(\ref{orbital}) can be expressed as
	\begin{align*}
		I_{\delta_\sigma}(f)&=\int_{U_\sigma(\A) \back U(\A)}\int_{U(\A)}f(t_1^{-1} u \delta_\sigma u_1 t_2){\psi_{\m_1}(u)}\overline{\psi_{\m_2}(u_1)}dudu_1.
	\end{align*}
Moreover, it factors as  $	I_{\delta_\sigma}(f)=I_{\delta_\sigma}(f_{\infty})I_{\delta_\sigma}(f_{\text{fin}})$, where we have set 
$f_{fin}=\prod_p f_p$.
\end{proposition}
\begin{remark}
	Note that the integral is well-defined by Corollary~\ref{psionUsigma}.
\end{remark}
\begin{remark}\label{d1d2}
	By Assumption~\ref{testfunction}, the support of $f_\infty$ is included in $G^+(\R)=\{g \in G(\R), \mu(g)>0\}$. 
	Therefore, if $\delta_1=\diag{d_1}{1}{d_2}{d_1d_2}$, we have $I_{\delta_\sigma}(f_{\infty}) \neq 0$ only if $d_1d_2>0$.
\end{remark}
\begin{proof}
	By Corollary~\ref{corqoutientmap} we can make the change of variable $(u_1,u_2,u_3)=\varphi(x,y)$ in~(\ref{orbital}).
	So we get
\begin{equation}\label{Idelta1}
	\begin{split}
	I_{\delta}(f)=\int_{U_\sigma(\Q) \back U_\sigma(\A)}\int_{U_\sigma(\A) \back U(\A)}\int_{U(\A)}
	f(t_1^{-1} u_2^{-1}\delta_\sigma u_3 t_2)) \quad \quad \quad \quad \quad \quad \\
\quad \quad	\times \psi_{\m_1}(u_1u_2)\overline{\psi_{\m_2}(\delta_\sigma^{-1}u_1\delta_\sigma u_3)}du_3du_2du_1.
	\end{split}
\end{equation}
We have 
\begin{align*}
	\psi_{\m_1}(u_1u_2)\overline{\psi_{\m_2}(\delta_\sigma^{-1}u_1\delta_\sigma u_3)}&=
	\psi_{\m_1}(u_2)\psi_{\m_1}(u_1)\overline{\psi_{\m_2}(\delta_\sigma^{-1}u_1\delta_\sigma)\psi_{\m_2} (u_3)}\\
	&=\psi_{\m_1}(u_2)\overline{\psi_{\m_2}(u_3)}
\end{align*}since $(u_1, \delta_\sigma^{-1}u_1\delta_\sigma) \in H_\delta(\A)$ and we assume $H \cdot \delta_\sigma$ is relevant orbit.
Reporting this equality in~(\ref{Idelta1}), we get 
$$I_{\delta_\sigma}(f)=\int_{U_\sigma(\A) \back U(\A)}\int_{U(\A)}f(t_1^{-1} u_2^{-1} \delta_\sigma u_3 t_2)\psi_{\m_1}(u_2)\overline{\psi_{\m_2}(u_3)}du_3du_2.$$
Write $u_3=u_\sigma u_1$ with $u_\sigma \in U_\sigma$ and $u_1 \in U_\sigma \back U$.
Then by Lemma~\ref{Usigmarelevant} we have 
$$u_2^{-1} \delta_\sigma u_3=u_2^{-1} \sigma\delta u_\sigma u_1
=u_2^{-1}  \sigma\delta u_\sigma \delta^{-1}\delta u_1=u_2^{-1} \delta u_\sigma \delta^{-1} \sigma \delta u_1,$$
and by Corollary~\ref{psionUsigma} we have 
\begin{align*}
	\psi_{\m_1}(u_2)\overline{\psi_{\m_2}(u_3)}&=\psi_{\m_1}(u_2)\overline{\psi_{\m_2}(u_\sigma u_1)}\\
	&=\psi_{\m_1}(u_2)\overline{\psi_{\m_1}(\delta u_\sigma \delta^{-1})} \overline{\psi_{\m_2}(u_1)}
	=\psi_{\m_1}(\delta u_\sigma^{-1} \delta^{-1}u_2) \overline{\psi_{\m_2}(u_1)}
\end{align*}
Setting $u=\delta u_\sigma^{-1}\delta^{-1} u_2$ we get the result.
\end{proof}

\subsection{The Archimedean orbital integrals}\label{ArchInt}
By Theorem~\ref{GeomTransform} and using equation~(\ref{whittakertorus}) we have the following
\begin{lemma}\label{InTermsOfW}
	Let $H \cdot \delta_\sigma$ be a relevant orbit. Then the corresponding Archimedean orbital integral is given by 
	\begin{align*}
		I_{\delta_\sigma}(f_{\infty})=\frac1c\frac{\Delta_\sigma(t_{\m_2})}{|\m_{1,1}^4\m_{1,2}^3|}
		\int_{U_\sigma(\R) \back U(\R)}\int_{\Lie{a}^*}&
		\tilde{f_\infty}(-i\nu)W(i\nu, t_{\m_1}^{-1}t_1, \psi_{\one})\\
	&	\times W(-i\nu,t_{\m_1}^{-1}\delta_\sigma t_{\m_2}u_1 t_{\m_2}^{-1} t_2,\overline{\psi_{\one}})
		\frac{d\nu}{c(i\nu)c(-i\nu)}
		{\overline{\psi_{\one}(u_1)}}du_1,
	\end{align*}
	where the constant $c$ is the one appearing in the spherical inversion theorem
	and $\Delta_\sigma$ is the modulus character of the group $U_\sigma(\R) \back U(\R)$.
\end{lemma}
Note that the above integral is well-defined.
More generally, let $\psi$ be a generic character and $\sigma$ a relevant element of the Weyl group.
Then by Lemma~\ref{Usigmarelevant}, for all $t \in G(\R)$ the integral 
$$\int_{U_\sigma(\R) \back U(\R)}\int_{\Lie{a}^*}g(-i\nu)W(-i\nu,y \sigma u_1 t,\psi){d\nu}
\overline{\psi(u_1)}du_1 $$
is well defined as long as the commutator $yuy^{-1}u^{-1}$ belongs to $U_0(\R)=\left\{
\left[\begin{smallmatrix}
	1	& 		&		& 	a	\\
	 	& 	1	& a	& b			\\
	& 		&	 1	& 		\\
	&		&		&	1			\\
\end{smallmatrix}\right],a,b \in \R \right\}$ for all $u \in U_\sigma(\R)$.
The following conjecture, due to Buttcane~\cite{Buttcane}, should enable to take the $\Lie{a}^*$ integral out in Lemma~\ref{InTermsOfW}.
\begin{conjecture}[Interchange of integral]\label{interchange}
	Let $g$ be holomorphic with rapid decay on an open tube domain of $\Lie{a}^*_\C$ containing $\Lie{a}^*$,
	and let $t \in G(\R)$.
	Let $\psi$ be a generic character, $\sigma$ be a relevant element of the Weyl group.
	Then for almost all $y \in \Sp_4(\R)$ satisfying $yuy^{-1}u^{-1}\in U_0(\R)$ for all $u \in U_\sigma(\R)$ we have 
	$$\int_{U_\sigma(\R) \back U(\R)}\int_{\Lie{a}^*}g(-i\nu)W(-i\nu,y \sigma u_1 t,\psi){d\nu}
	\overline{\psi(u_1)}du_1=\int_{\Lie{a}^*}g(-i\nu) \tilde{K_\sigma}(-i\nu, y,t)d\nu$$
	where 
	$$\tilde{K_\sigma}(-i\nu, y,t)=\lim_{R \to 0}\int_{U_\sigma(\R) \back U(\R)}h\left(\frac{\|u_1\|}R\right)
	W(-i\nu,y \sigma u_1 t,\psi)	\overline{\psi(u_1)}du_1,$$
	for some fixed, smooth, compactly supported $h$ with $h(0)=1$.
	Moreover $\tilde{K_\sigma}$ is entire in $\nu$ and smooth and polynomially bounded in $t$ and $y$ for $\Re(-i\nu)$ in some
	fixed compact set.
\end{conjecture}	
Note that Conjecture~\ref{interchange} is not needed for $\sigma=1$ since in this case 
we have $U_\sigma=U$ and hence $\tilde{K_\sigma}(-i\nu, y,t)=W(-i\nu,y \sigma t,\psi)$. 
Consider now the case $\sigma=J$ the long Weyl element. In this case $U_\sigma$ is trivial.
Let $u \in U(\R)$ and $k \in K_\infty$.
Then changing variables and using the fact that $W(-i\nu,y \cdot,\psi)$ is right-$K_\infty$ invariant we have 	
\begin{equation}\label{TildeKIsWhittaker}
	\tilde{K_\sigma}(-i\nu, y,utk)=\psi(u)\tilde{K_\sigma}(-i\nu, y,t).
\end{equation}
Moreover $\tilde{K_\sigma}(-i\nu, y, \cdot)$ is an eigenfunction of the centre of the universal enveloping algebra in each variable,
with eigenvalue matching those of $W(-i\nu, \cdot,\psi)$. 
It follows from the uniqueness of the Whittaker model that $$\tilde{K_\sigma}(-i\nu, y,t)=K_\sigma(-i\nu,y,\psi)W(-i\nu,t,\psi)$$
for some function $K_\sigma(-i\nu,y, \psi)$ that we call the long Weyl element {\bf Bessel function}.
$K_\sigma(-i\nu, \cdot, \psi)$ is itself an eigenfunction of the centre of the universal enveloping algebra
with eigenvalue matching those of $W(-i\nu, \cdot,\psi)$, and satisfies for all $u \in U(\R)$ the transformation rule
\begin{equation}\label{BesselTransformation}
	K_\sigma(-i\nu, uy, \psi)=\psi(u)K_\sigma(-i\nu, y, \psi)=K_\sigma(-i\nu, y \sigma u \sigma^{-1}, \psi).
\end{equation}
For the remaining two relevant elements of the Weyl group
$\tilde{K_\sigma}(-i\nu, y, \cdot)$ still satisfies relation~(\ref{TildeKIsWhittaker}) for all $u \in U(\R)$.
Thus we can still factor $\tilde{K_\sigma}(-i\nu, y,t)=K_\sigma(-i\nu,y,\psi)W(-i\nu,t,\psi)$ for appropriate $y$,
and define $K_\sigma(-i\nu,y,\psi)$ to be the $\sigma$-{\bf Bessel function}.
However because of the restriction on $y$, the conditions satisfied by $K_\sigma(-i\nu,y,\psi)$ are more complicated.

Buttcane has announced a proof for Conjecture~\ref{interchange} in a more general context.
This would thus yield a uniform expression for the Archimedean integrals attached to the various elements of the Weyl group.

\begin{proposition}~\label{ArchWithConj}
	Assume Conjecture~\ref{interchange}
	Let $H \cdot \delta_\sigma$ be a relevant orbit. Then the corresponding Archimedean orbital integral is given by 
	\begin{align*}
		I_{\delta_\sigma}(f_{\infty})=\frac1c\frac{\Delta_\sigma(t_{\m_2})}{|\m_{1,1}^4\m_{1,2}^3|}
	\int_{\Lie{a^*}}\tilde{f_\infty}(-i\nu)
	K_\sigma(-i\nu,t_{\m_1}^{-1}\sigma \delta t_{\m_2} \sigma^{-1},{\psi_{\one}})\\
	\times W(-i\nu, t_{\m_2}^{-1} t_2,{\psi_{\one}})
	W(i\nu, t_{\m_1}^{-1}t_1, \psi_{\one})\frac{d\nu}{c(i\nu)c(-i\nu)}
	\end{align*}
	where the constant $c$ is the one appearing in the spherical inversion theorem
	and $\Delta_\sigma$ is the modular character of the group $U_\sigma(\R) \back U(\R)$.	 
\end{proposition}

\begin{proof}
	We apply the statement of Conjecture~\ref{interchange} to the integral in Lemma~\ref{InTermsOfW} for the function
	defined by 
	$$g(-i\nu)=\frac{1}{c(i\nu)c(-i\nu)}\tilde{f_\infty}(-i\nu)W(i\nu, t_{\m_1}^{-1}t_1, \psi_{\m_1}),$$ which has rapid decay
	by the rapid decay of $\tilde{f_\infty}$ (Theorem~\ref{PW}), the explicit expression~(\ref{Plancherel}) for the spectral measure, and the estimate for the Whittaker function in the spectral aspect given by Proposition~\ref{WhittakerSpectral}.
\end{proof}

\subsection{Symplectic Kloosterman sums}
In this section, we compute the non-Archimedean part of the orbital integrals when the finite part of the test function satisfies the following.
\begin{assumption}\label{HeckeOprtr}
	Recall from Assumption~\ref{testfunction} that we assume $f=f_\infty\prod_p f_p$ has central character $\omega.$
	We now further assume that there are two coprime positive integers $N$ and $\n$  such that
	$\omega$ is trivial on $(1+N\hat{\Z}) \cap \hat{\Z}^\*$, and 
	the function $f_{fin}$ is supported on $Z(\A_{fin})M(\n,N)$ and satisfies
	$$f_{fin}(zm)=\frac1{Vol(\overline{\Gamma_1(N)})}\overline{\omega}(z)$$
 	for $z \in Z(\A_{fin})$ and $m\in M(\n,N)$,
	where 	$$M(\n,N)=\left\{g \in G(\A_{fin}) \cap \Mat_4(\hat{\Z}) : g \equiv \left[
	\begin{smallmatrix}
		* 	& 	& *	& *	\\
		*	& 1	& *	& *	\\
			&	& *	& *	\\
			&	& 	& *
	\end{smallmatrix}\right] \mod N, \mu(g) \in \n  \hat{\Z}^\* \right\},$$
and $$\Gamma_1(N)=\left\{g \in G(\hat{\Z}) : g \equiv \left[
\begin{smallmatrix}
	* 	& 	& *	& *	\\
	*	& 1	& *	& *	\\
	&	& *	& *	\\
	&	& 	& *
\end{smallmatrix}\right] \mod N\right\}.$$
\end{assumption}
\begin{remark}
	With this choice, $f=\bigotimes_p f_p$, and each $f_p$ is left and right $\Gamma_p(N)$-invariant,
	where $$\Gamma_p(N)=\{g \in G(\Z_p): g \equiv \left[
	\begin{smallmatrix}
		* 	& 	& *	& *	\\
		*	& 1	& *	& *	\\
			&	& *	& *	\\
			&	& 	& *
	\end{smallmatrix}\right] \mod N \}.$$
	In particular, if $x, c \in \Z_p$ then $\Gamma_p$ contains the matrix $\left[
	\begin{smallmatrix}
		1 	& 	& c	& -cx	\\
		x	& 1	& 	& 	\\
		&	& 1	& -x	\\
		&	& 	& 1
	\end{smallmatrix}\right].$
	Thus if $\phi$ is right-$\Gamma_p$-invariant for all prime $p$, changing variables 
	$u \mapsto u\left[
	\begin{smallmatrix}
		1 	& 	& c	& -cx	\\
		x	& 1	& 	& 	\\
		&	& 1	& -x	\\
		&	& 	& 1
	\end{smallmatrix}\right]$ in the integral expression of the $\psi_\m$-Whittaker coefficient of $\phi$,
	we get $\W_{\psi_m}(\phi)(g)=\overline{\theta(m_1x+m_2c)}\W_{\psi_m}(\phi)(g)$ for all $g$.
	Therefore $\W_{\psi_m}(\phi)=0$ unless $m_1$ and $m_2$ are integers.
	Henceforth, we shall assume $\m_1$ and $\m_2$ are two pairs of integers.
\end{remark}
\begin{remark}
	Note that $\Gamma=K_\infty \prod_p\Gamma_p(N)$ is contained both in the Borel, Klingen, Siegel, and paramodular congruence subgroup of level $N$,
	thus any automorphic form that is fixed by one of these groups is also fixed by $\Gamma$, and hence will appear in our formula.
	One could fix a different choice of congruence subgroup, and accordingly define different types of Kloosterman sums.
\end{remark}
Under Assumption~\ref{HeckeOprtr}, we can give more restriction about the element $\delta_1$ in Lemma~\ref{relevantorbits}.

\begin{lemma}\label{Supportdelta}
	Let $\sigma \in \Omega$, $\delta_1=\diag{d_1}{1}{d_2}{d_1d_2}$
	such that the orbit of $\delta_\sigma=\sigma \delta_1$ is relevant.
	Assume $I_{\delta_\sigma}(f_{\text{fin}}) \neq 0$.
	Then there is an integer $s$ such that $d_1d_2=\pm\frac{\n}{s^2}$.
\end{lemma}
\begin{proof}
	For all $u \in U(\A)$ and $u_1 \in U_\sigma(\A) \back U(\A)$
	we have $\mu(u \delta_\sigma u_1)=d_1d_2$. 
	So by Assumption~\ref{HeckeOprtr}, $u \delta_\sigma u_1$ belongs to the support of $f$
	only if $d_1d_2 \in \A_{fin}^2\hat{\Z}^\* \n$.
	Since $d_1d_2$ is a rational number, there must be a rational number $s$ such that $d_1d_2=\pm \frac{\n}{s^2}$.
	But the second diagonal entry of $s \sigma^{-1} u \delta_\sigma u_1$ is $s$ therefore $s$ must belong to $\hat{\Z}$,
	hence $s$ is an integer. 
\end{proof}
Henceforth, we shall assume $\delta$ is as in Lemma~\ref{Supportdelta}. By Remark~\ref{d1d2}, we could also assume that $d_1d_2>0$ 
(which would then fix the sign in the equality $d_1d_2=\pm\frac{\n}{s^2}$ above). However, we do not need doing so for now, and we shall not, 
in view of possible applications with a different choice of test function at the Archimedean place.
\begin{remark}
	Consider the case $N=\n=1$. Then $M(\n,N)=\GSp_4(\hat{\Z})=\prod_p \GSp_4(\Z_p).$
	For simplicity, set $\eta=\delta_\sigma$, and if $p$ is a prime and $x \in G(\A)$, write $x_p$ for the $p$-th
	component of $x$.
	Also write $\psi_{p,1}$ and $\psi_{p,2}$ for the local $p$-th components of the characters $\psi_{\m_1}$ and $\overline{\psi_{\m_2}}$, respectively.
	In particular, these characters are trivial on $U(\Z_p)$.
	Then we have 
	\begin{align*}
		I_{\delta_\sigma}(f_{fin})&= \frac1{Vol(\overline{\Gamma_1(N)})}
		\int_{U_\sigma(\A_{fin}) \back U(\A_{fin})}\int_{U(\A_{fin})}\1_{\GSp_4(\hat{\Z})}(s u \eta v){\psi_{\m_1}(u)}\overline{\psi_{\m_2}(v)}dudv\\
		&=\prod_p \frac1{Vol(\overline{\Gamma_p(N)})} \int_{U_\sigma(\Q_p) \back U(\Q_p)}\int_{U(\Q_p)}\1_{\GSp_4(\Z_p)}(s_p u_p \eta_p v_p){\psi_{p,1}(u_p)}{\psi_{p,2}(v_p)}du_pdv_p.\\
	\end{align*}
	For all but finitely many prime $p$, the entries of $s_p\eta_p$ are in $\Z_p^\*$. For those primes, by the explicit Bruhat decomposition 
	(see Lemma~\ref{uoppu} below),
	the condition $s_p u_p \eta_p v_p \in \GSp_4(\Z_p)$ is equivalent to $u_p \in U(\Z_p)$ and $v_p \in U_\sigma(\Z_p) \back U(\Z_p)$, and hence
	$$\int_{U_\sigma(\Q_p) \back U(\Q_p)}\int_{U(\Q_p)}\1_{\GSp_4(\Z_p)}(s_p u_p \eta_p v_p){\psi_{p,1}(u_p)}{\psi_{p,2}(v_p)}du_pdv_p
	=1$$ 
	For the remaining primes $p$, noticing that $U_\sigma(\Q_p) \back U(\Q_p)$ may be identified to the subgroup
	$\overline{U_\sigma}(\Q_p)=U(\Q_p) \cap \sigma^{-1} \trans{U(\Q_p)} \sigma$, 
	the local integral equals the Kloosteman sum $\Kloos(\eta,\psi_{p,1},\psi_{p,2})$ as defined in~\cite{SHMKloosterman} 
	when $\eta \in \Sp_4(\Q_p)$ (note that we denote here by $U_\sigma$ what is denoted there by $\overline{U_\sigma}$, and conversely).
\end{remark}
We now treat separately the contribution from each relevant element of the Weyl group from a global point of view.
To alleviate notations, we shall not include $N$ and $\omega$ in the argument of the Kloosterman sums we proceed to define. 

\subsubsection{The identity contribution}
\begin{definition}\label{sumid}
	Let $a,b,d,N$ be integers such that $d \mid N$.
	Then the following sum is well-defined
	$$S(a,b,d,N)=\sum_{\substack{x,y \in \Z/ N \Z\\d \mid xy}}e\left(\frac{ax+by}N\right).$$
\end{definition}
\begin{lemma}
	Let $a,b,d,N$ be integers such that $d \mid N$.
	Write $a=\prod_i p_i^{a_i}$, where $p_i$ are distinct primes and $a_i$ are integers, and similarly for $b, d, N$.
	Then we have 
	$$S(a,b,d,N)=\prod_i S(p_i^{a_i}, p_i^{b_i}, p_i^{d_i}, p_i^{N_i}).$$
	Moreover if $n$ is a positive integer, $i,j,k$ are non-negative integers with $k \le n$ and $p$ is a prime, then we have
	\begin{align*}
		S(p^i,p^j,p^k,p^n)&=p^{2n-k}(1-p^{-1})\max(0,k+1-\max(0,n-i)-\max(0,n-j))\\
	&+p^{2n-k-1}\left(\1_{\substack{i \ge n\\j \ge n}}-
	\1_{\substack{i < n\\j < n\\i+j\ge 2n-k-1}}\right).
	\end{align*}
	In particular, it follows that $S(p^i,p^j,p^k,p^n)$ is non-zero only if 
	\begin{equation}\label{sumnonzero}
	(n-i)+(n-j) \le k+1.	
	\end{equation}
\end{lemma}
\begin{proof}
	The factorization is immediate from the Chinese remainders theorem.
	Now let us evaluate $S=S(p^i,p^j,p^k,p^n)$. 
	We have (here, abusing slightly notations, we set $v_p(0)=n$)
	$$S=\sum_{h=0}^k \sum_{\substack{x \in \Z / p^n \Z \\ v_p(x)=h}} e\left(\frac{p^ix}{p^n}\right)
	 \sum_{\substack{y \in \Z / p^n \Z \\ v_p(y) \ge k-h}} e\left(\frac{p^jy}{p^n}\right)
	 + \sum_{\substack{x \in \Z / p^n \Z \\ v_p(x) \ge k+1}} e\left(\frac{p^ix}{p^n}\right)
	 \sum_{y \in \Z / p^n \Z} e\left(\frac{p^jy}{p^n}\right).$$
	 Now if $\ell$ is any non-negative integer, we have
	 $$ \sum_{\substack{y \in \Z / p^n \Z \\ v_p(y) \ge \ell}} e\left(\frac{p^jy}{p^n}\right)=
	 \begin{cases}
	 	p^{n-\ell} \text{ if } j+\ell \ge n \text{ and } \ell \le n\\
	 	0 \text{ otherwise.}
	 \end{cases}$$
Hence 
$$S=\sum_{h=0}^{k-\max(0,n-j)}p^{n-k+h} \sum_{\substack{x \in \Z / p^n \Z \\ v_p(x)=h}} e\left(\frac{p^ix}{p^n}\right)
+p^{2n-k-1}\1_{\substack{j \ge n\\i+k+1 \ge n\\k<n}}.$$
Now $$\sum_{\substack{x \in \Z / p^n \Z \\ v_p(x)=h}} e\left(\frac{p^ix}{p^n}\right)=
\sum_{\substack{x \in \Z / p^n \Z \\ v_p(x) \ge h}} e\left(\frac{p^ix}{p^n}\right)
- \sum_{\substack{x \in \Z / p^n \Z \\ v_p(x) \ge h+1}} e\left(\frac{p^ix}{p^n}\right),$$
hence the $h$-sum becomes
$$\left(\sum_{h=\max(0,n-i)}^{k-\max(0,n-j)} p^{2n-k}(1-p^{-1})\right)-p^{2n-k-1}\1_{0 \le n-i-1 \le k-\max(0,n-j)}
+p^{2n-k-1}\1_{k-\max(0,n-j)=n},$$
so 
\begin{align*}
	S=p^{2n-k}(1-p^{-1})\max(0,k+1-\max(0,n-i)-\max(0,n-j))\\
	+p^{2n-k-1}(\1_{\substack{j \ge n\\i+k+1 \ge n\\k<n}}-\1_{0 \le n-i-1 \le k-\max(0,n-j)}+\1_{k-\max(0,n-j)=n}).
\end{align*}
Finally, it can be checked by inspection of cases that  
$$\1_{\substack{j \ge n\\i+k+1 \ge n\\k<n}}-\1_{0 \le n-i-1 \le k-\max(0,n-j)}+\1_{k-\max(0,n-j)=n}=
\1_{\substack{i \ge n\\j \ge n}}-
\1_{\substack{i < n\\j < n\\i+j\ge 2n-k-1}}.$$
\end{proof}
\begin{proposition}\label{idcontribution}
	Let $\sigma=1$, $\delta_1=\diag{d_1}{1}{d_2}{d_1d_2}$ with $d_1d_2=\pm\frac{\n}{s^2}$ for some integer $s$.
	Then $I_{\delta_\sigma}(f_{fin})=0$ unless all of the following holds:
	\begin{enumerate}
		\item $s$ divides $\n$,
		\item $d_1=\frac{\m_{1,1}}{\m_{2,1}}$ and $sd_1$ is an integer dividing $\n$,
		\item $d_2=\frac{\m_{1,1}\m_{1,2}}{\m_{2,1}\m_{2,2}}$,
	\end{enumerate}
	If all these conditions are met, let $d=\gcd(s,sd_1s,d_2,\frac{\n}s)$, and $D=\gcd(sd_1,\frac{n}s)$. Then 
	\begin{equation}\label{IdContribution}
		I_{\delta_\sigma}(f_{fin})=\frac{\overline{\omega_N(s)}}{Vol(\overline{\Gamma_1(N)})}\frac{\n d}{|s^3d_1|}
	S\left(\m_{1,1}\frac{\n}D,\m_{1,2}sd_1,d,\n\right),
	\end{equation}
	where $\omega_N(s)=\prod_{p \mid N} \omega_p(s)$.
\end{proposition} 
\begin{remark}
	The integer $s$ is only determined up to sign. However, expression~(\ref{IdContribution}) does not depend on the sign of $s$, 
	since $S(a,b,d,\n)=S(a,-b,d,\n)$ and $\omega_N(-1)=\omega(-1)=1$ as $\omega_p(-1)=1$ for all $p \nmid N$.
\end{remark}
\begin{remark}
	The two pair of integers $\m_1$ and $\m_2$ essentially play symmetric roles in our formula.
	More precisely, for our choice of test function $f$, the operator $\omega_N(\n)^{\frac12}R(f)$ is self-adjoint.
	Thus exchanging $\m_1$ and $\m_2$ amounts to take the complex conjugate of the spectral side and multiply it by $\omega_N(\n)$.
	Hence the geometric side, and in particular the identity contribution, should enjoy the same symmetries.
	Proposition~\ref{idcontribution} says that the identity element has a non-zero contribution only if there is an integer $t$ dividing $\n$
	with $\frac{\n}t=\pm\frac{\m_{1,2}}{\m_{2,2}}t$ and such that $s=\frac{\m_{2,1}}{\m_{1,1}}t$ is also an integer dividing $\n$.
	This condition is indeed symmetric, as interchanging $\m_1$ and $\m_2$ amounts to replace $t$ with $\frac{\n}t$ and $s$ with $\frac{\n}s$.
	In addition, we have 
	$S\left(\m_{1,1}\frac{\n}{\gcd(t,\frac{\n}s)},\m_{1,2}t,d,\n\right)=S\left(\m_{2,1}\frac{\n}{\gcd(s,\frac{\n}t)},\m_{2,2}\frac\n{t},d,\n\right)$.
	Finally, using that $|s^3d_1|=\left|\n^3\frac{\m_{2,1}^4\m_{2,2}^3}{\m_{1,1}^4\m_{1,2}^3}\right|^{\frac12}$,
	multiplying $\frac{\n }{|s^3d_1|}$ by the factor $\frac1{|\m_{1,1}^4\m_{1,2}^3|}$ that comes from the Archimedean part in Proposition~\ref{ArchWithConj}
	gives $\n^{-\frac12}(\m_{1,1}\m_{2,1})^{-2}|\m_{1,2}\m_{2,2}|^{-\frac32}$.
\end{remark}
\begin{remark}
	In the case $\n=1$ we must have $s=\pm1$ and hence $\m_{1,1}=\pm\m_{2,1}$. Together with the condition 
	$d_1d_2=\pm\frac{\m_{1,1}^2\m_{1,2}}{\m_{2,1}^2\m_{2,2}}=\frac{\n}{s^2}$ this also gives $\m_{1,2}=\pm\m_{2,2}.$
\end{remark}
\begin{remark}
	Using condition~(\ref{sumnonzero}) we find that the contribution from the identity element is non-zero only if for 
	all prime $p \mid \n$ we have $$v_p(s) \le v_p(\m_{2,1})+v_p(\m_{2,1})+\min(0,v_p(\m_{2,1})-v_p(\m_{1,1}))+1,$$
	which in turn implies that for all prime $p$ we have 
	$$v_p(\n) \le 2 \min(v_p(\m_{1,1}),v_p(\m_{2,1}))+v_p(\m_{1,2})+v_p(\m_{2,2})+1.$$
\end{remark}
\begin{proof}
The finite part of the orbital integral corresponding to the identity element reduces to
\begin{align*}
	I_{\delta_\sigma}(f_{fin})= \int_{U(\A_{fin})}f( u \delta){\psi_{\m_1}(u)}du
	=\int_{U(\A_{fin})}f(s u \delta){\psi_{\m_1}(u)}du.
\end{align*}
Assume it is non-zero.
Note that by Lemma~\ref{Supportdelta} we have $\mu(su\delta)=\n$.
Then by Assumption~\ref{HeckeOprtr}, $su\delta \in Supp(f)$ if and only if $su\delta \in \Mat_4(\hat{\Z})$.
In particular, each entry of $s\delta$ must be an integer.
Furthermore by Lemma~\ref{relevantorbits} we must have 
$\delta=\diag{d_1}{1}{d_2}{d_1d_2}$ with  $d_1=\frac{\m_{1,1}}{\m_{2,1}}$, $d_2=\frac{\m_{1,1}\m_{1,2}}{\m_{2,1}\m_{2,2}}$.
So we learn that $sd_1= s\frac{\m_{1,1}}{\m_{2,1}} \in \Z$, $s \mid \n$, and $sd_1 \mid \n$.
Now let us examine the non-diagonal entries of $su\delta$.
Write $u=\left[
\begin{smallmatrix}
	1 	& 	& c	& a-cx	\\
	x	& 1	& a	& b		\\
	&	& 1	& -x	\\
	&	& 	& 1
\end{smallmatrix}\right]$.
Then the following conditions must hold:
\begin{enumerate}
	\item $sd_1x \in \hat{\Z}$ and $\frac{\n}{s}x \in \hat{\Z}$,
	\item $c' \doteq \frac{\n}{sd_1}c \in \hat{\Z}$,
	\item $a' \doteq \frac{\n}{sd_1}a \in \hat{\Z}$,
	\item $\frac{\n}{s}(a-cx) \in \hat{\Z}$,
	\item $b' \doteq \frac{\n}{s}b \in \hat{\Z}$.
\end{enumerate}
Condition~(1) is equivalent to $x \in \frac1{D}\hat{\Z}$, where $D=\gcd(sd_1, \frac{\n}s)$ (note that $sd_1 \mid sD$).
Set $x'=Dx$.
Then condition~(4) gives
$d_1a'-\frac{d_1}{D}c'x' \in \hat{\Z}$.
Combined with conditions~(1),~(2) and~(3), this is equivalent to 
$ c'x' \equiv Da' \mod \frac{D}{d_1}$.
Now, $\psi_{\m_1}(u)=\theta_{fin}(\m_{1,1}x+\m_{1,2}c)$ and $f(su\delta)=\frac{\overline{\omega_N(s)}}{Vol(\overline{\Gamma_1(N)})}$. 
Therefore integration with respect to $b$ gives
 $Vol\left(\frac{s}{\n}\hat{\Z}\right)\frac{\overline{\omega_N(s)}}{Vol(\overline{\Gamma_1(N)})}=
 \frac{\n}s\frac{\overline{\omega_N(s)}}{Vol(\overline{\Gamma_1(N)})}.$
Next, changing variables $x=\frac1Dx'$ and $c=\frac{sd_1}\n c'$, for fixed $a$ the $x,c$-integral is 
$$I(a)=\frac{\overline{\omega_N(s)}}{Vol(\overline{\Gamma_1(N)})}\frac{\n^2D}{s^2d_1}\int\int_{ c'x' \equiv Da' \mod \frac{D}{d_1}}
\theta_{fin}\left(\m_{1,1}\frac{x'}{D}+\m_{1,2}\frac{sd_1}\n c'\right)dx'dc'.$$
Since $D \mid sd_1$ and $\m_{1,2}\frac{s^2d_1^2}\n=\m_{2,2}$, and since $\theta_{fin}$ is trivial on $\hat{\Z}$ 
the integrand is constant on cosets $x'+sd_1\hat{\Z}$ and $c' + sd_1\hat{\Z}$.
As $sd_1 \mid sD$, it is also constant on cosets $x'+sD\hat{\Z}$ and $c' + sD\hat{\Z}$.
Therefore we get 
\begin{align*}
	I(a)&=\frac{\overline{\omega_N(s)}}{Vol(\overline{\Gamma_1(N)})}\frac{\n^2}{|Dd_1s^4|}\sum_{\substack{x,y \in \Z / sD \Z \\ xy \in Da'+\frac{D}{d_1} \hat{\Z}}} 
	e\left(\frac{\m_{1,1}x}D+\frac{\m_{1,2}sd_1y}{\n}\right)
\end{align*}
Finally the $a$ integrand depends only on $a' \mod \frac{D}{d_1} \hat{\Z}$, thus, setting $d=\gcd\left(D,\frac{D}{d_1}\right)=\gcd(s,sd_1sd_2,\frac{\n}s)$ we get
\begin{align*}
	I_{\delta_\sigma}(f_{fin})&=\frac{\overline{\omega_N(s)}}{Vol(\overline{\Gamma_1(N)})}
	\frac{\n^3 }{|s^5D^2d_1|}\sum_{a \in \Z / \frac{D}{d_1} \Z}\sum_{\substack{x,y \in \Z / sD \Z \\ xy \in Da+\frac{D}{d_1} \hat{\Z}}} 
	e\left(\frac{\m_{1,1}x}D+\frac{\m_{1,2}sd_1y}{\n}\right)\\
	&=\frac{\overline{\omega_N(s)}}{Vol(\overline{\Gamma_1(N)})}\frac{\n^3 }{|s^5D^2d_1|}d
	\sum_{\substack{x,y \in \Z / sD \Z \\ xy \in d \Z}} 
	e\left(\frac{\m_{1,1}x}D+\frac{\m_{1,2}sd_1y}{\n}\right)\\
	&=\frac{\overline{\omega_N(s)}}{Vol(\overline{\Gamma_1(N)})}\frac{\n }{|s^3d_1|}d
		\sum_{\substack{x,y \in \Z / \n \Z \\ xy \in d \Z}}  
	e\left(\frac{\m_{1,1}x}D+\frac{\m_{1,2}sd_1y}{\n}\right)\\
\end{align*}
\end{proof}

\subsubsection{The contribution from the longest Weyl element}
The following lemma makes it explicit how to compute the Bruhat decomposition for elements in the cell of the 
long Weyl element. 
One could do the same for each element of the Weyl group, but, as it is straightforward calculations, we only
include this case for the sake of clarity in latter arguments.
\begin{lemma}\label{uoppu}
	Let $\F$ be a field, and let $g \in \GSp_4(\F)$.
	Assume
	$$g=\left[
		\begin{smallmatrix}
		1 	& 		&   c_1			& 	a_1	\\
		x_1	& 	1	&	a_1+c_1x_1	& 	b_1			\\
			& 		&	 1			& 	-x_1		\\
			&		&				&	1			\\
	\end{smallmatrix}
	\right]
	J
	\diag{t_1}{t_2}{t_3t_1^{-1}}{t_3t_2^{-1}}
	\left[
	\begin{smallmatrix}
		1 	& 		&   c_2	& 	a_2-c_2x_2	\\
		x_2	& 	1	&	a_2	& b_2			\\
		& 		&	 1	& 	-x_2		\\
		&		&		&	1			\\
	\end{smallmatrix}
	\right]
	= \mat{A}{B}{C}{D} = 
	\bigmat{\block{a_{11}}{a_{12}}{a_{21}}{a_{22}}}
	{B}
	{\block{c_{11}}{c_{12}}{c_{21}}{c_{22}}}
	{\block{d_{11}}{d_{12}}{d_{21}}{d_{22}}} .$$
	Set
	\begin{alignat*}{1}
		\Delta_1=\mat{a_{11}}{a_{12}}{c_{21}}{c_{22}},& \quad \Delta_2=\mat{c_{12}}{d_{11}}{c_{22}}{d_{21}}.
	\end{alignat*}
	Then	
	\begin{alignat*}{5}
		t_3=\mu(g), & \quad t_2=- c_{22}, & \quad t_1t_2=\det(C),
	\end{alignat*}
\begin{alignat*}{5}
	x_1=-\frac{c_{12}}{c_{22}},& \quad x_2=\frac{c_{21}}{c_{22}}, & \quad c_1=\frac{\det(\Delta_1)}{\det(C)},& \quad c_2=-\frac{\det(\Delta_2)}{\det(C)},
\end{alignat*}
	\begin{alignat*}{5}
		a_{1} =\frac{a_{12}}{c_{22}} & \quad a_2 = \frac{d_{21}}{c_{22}}&\quad b_1 =\frac{a_{22}}{c_{22}} & \quad b_2 =\frac{d_{22}}{c_{22}}.
	\end{alignat*}
	Moreover, if $g= \mat{A}{B}{C}{D} \in \GSp_4(\F)$ with $C=\mat{c_{11}}{c_{12}}{c_{21}}{c_{22}}$ satisfies 
	$\det(C) \neq 0$ and $c_{22} \neq 0$ then~$g \in UJ T U$.
\end{lemma}
\begin{proof}
	The first claims follow by computing explicitly
	\begin{alignat*}{3}
		C&=\mat{1}{-x_1}{}{1}\mat{-t_1}{}{}{-t_2}\mat{1}{}{x_2}{1}=\mat{-t_1+t_2x_1x_2}{t_2x_1}{-t_2x_2}{-t_2},\\
		\Delta_1&=\mat{c_1}{a_1}{}{1}\mat{-t_1}{}{}{-t_2}\mat{1}{}{x_2}{1}, & \quad
		\Delta_2&=\mat{1}{-x_1}{}{1}\mat{-t_1}{}{}{-t_2}\mat{}{c_2}{1}{a_2},\\
		D&=\mat{1}{-x_1}{}{1}\mat{-t_1}{}{}{-t_2}\mat{c_2}{a_2-c_2x_2}{a_2}{b_2},&\quad
		A&=\mat{c_1}{a_1}{a_1+c_1x_1}{b_1}\mat{-t_1}{}{}{-t_2}\mat{1}{}{x_2}{1}.
	\end{alignat*}
	To prove the last claim, it suffices to show that provided $\det{C} \neq 0$ and $c_{22} \neq 0$,
	there exist at most one $g \in \GSp_4(\F)$ with the specified values for $\mu(g)$, $C$, $a_{12}, a_{22},$  $d_{21}, d_{22}$,  
	$\det(\Delta_1)$ and $\det(\Delta_2)$.
	Since $c_{22} \neq 0$, the values of $a_{12}$, $c_{21}$ and $\det(\Delta_1)=a_{11}c_{22}-c_{21}a_{12}$ determine the value of 
	$a_{11}$. The equation $\trans{A}C=\trans{C}A$ then gives $a_{12}c_{11}+a_{22}c_{21}=a_{11}c_{12}+a_{21}c_{22}$,
	which determines the value of $a_{21}$ hence of $A$.
	The same reasoning using  $\det(\Delta_1)$ and $C\trans{D}=D\trans{C}$ instead similarly fixes $D$.
	Finally the equation $\trans{A}D-\trans{C}B=\mu(g)$ fixes $B$ since we are assuming $C$ is invertible.
\end{proof}
\begin{definition}\label{KloosLWE}
	Let $s,d,m$ be three non-zero integers and define 
	$$C_J(s,d,m)= \{g=\mat{A}{B}{C}{D} :	\det(C)=d, c_{22}=-s, \mu(g)=m\}$$
and
	$$\Gamma_J(N,s,d,m)= \Gamma_1(N) \cap \Mat_4(\Z) \cap C_J(s,d,m).$$
	For $g=\mat{A}{B}{C}{D} \in \GSp_4$,
	let $\Delta_1=\mat{a_{11}}{a_{12}}{c_{21}}{c_{22}}$
	and $\Delta_2=\mat{c_{12}}{d_{11}}{c_{22}}{d_{21}}$.
	Then, for $\m_1,\m_2$ two pair of non-zero integers, we define the following generalized twisted Kloosterman sum
	\begin{align*}
		\Kloos_J(\m_1,\m_2,s,d,m)&=\\
		\sum_{g \in U(\Z) \back \Gamma_J(N,s,d,m) / U(\Z)}&\overline{\omega_N}(a_{22})
		e\left(\frac{\m_{11}c_{12}+\m_{21}c_{21}}{s}+\frac{\m_{12}\det(\Delta_1)+\m_{22}\det(\Delta_2)}{d}\right).
	\end{align*}
\end{definition}
\begin{remark}
	Using Lemma~\ref{uoppu}, we can see that $\Kloos_J(\m_1,\m_2,s,d,m)$ is well defined. 
	Indeed, matrices in $\Gamma_J(\n,N,d,s)$ are of the form $$g=u(x_1,a_1,b_1,c_1)J\diag{\frac{d}s}{s}{m\frac{s}{d}}{\frac{m}s}u(x_2,a_2,b_2,c_2).$$
	Then $\frac{c_{12}}s=x_1$,$\frac{c_{21}}s=-x_2$, $\frac{\det(\Delta_1)}{d}=c_1$ and $\frac{\det(\Delta_2)}{d}=-c_2$.
	Now multiplying g on the left (resp. on the right) by an element of $U(\Z)$ does not change the classes of $x_1$ and $c_1$ (resp. $x_2$ and $c_2$)
	in $\R / \Z$. 
\end{remark}

\begin{proposition}\label{KloosJ}
	Let $\sigma=J$, $\delta_1=\diag{d_1}{1}{d_2}{d_1d_2}$
	with $d_1d_2=\pm\frac{\n}{s^2}$ for some integer $s$.
	Then we have $I_{\delta_\sigma}(f_{fin})=\frac{1}{Vol(\overline{\Gamma_1(N)})}\Kloos_J(\m_1,\m_2,s,d_1s^2,s^2d_1d_2)$.
\end{proposition}
\begin{remark}
	The set $\Gamma_J(N,s,d_1s^2,s^2d_1d_2)$ is non-empty only if $N$ divides $s$ and $N^2$ divides $d_1s^2$.
\end{remark}
\begin{proof}
The finite part of the orbital integral corresponding to the longest Weyl element reduces to
\begin{align*}
	I_{\delta_\sigma}(f_{fin})&= \int_{U(\A_{fin})}\int_{U(\A_{fin})}f( u_1 J \delta u_2){\psi_{\m_1}(u_1)}\overline{\psi_{\m_2}(u_2)}du_1du_2\\
	&=\int_{U(\A_{fin})}\int_{U(\A_{fin})}f(s u_1 J \delta u_2){\psi_{\m_1}(u_1)}\overline{\psi_{\m_2}(u_2)}du_1du_2.
\end{align*}
By Assumption~\ref{HeckeOprtr} $su_1 J \delta u_2 \in Supp(f)$ if and only if $su_1J\delta u_2=\mat{A}{B}{C}{D} \in Z(\A_{fin})M(\n,N)$.
In this case, we have $f(s u_1 J \delta u_2)=\frac{\overline{\omega_N}(a_{22})}{Vol(\overline{\Gamma_1(N)})}$, 
and Lemma~\ref{uoppu} shows that
$${\psi_{\m_1}(u_1)}\overline{\psi_{\m_2}(u_2)}=
e\left(-\frac{\m_{11}c_{12}+\m_{21}c_{21}}{c_{22}}+\frac{\m_{12}\det(\Delta_1)+\m_{22}\det(\Delta_2)}{\det(C)}\right).$$
Moreover, $f$ is left and right $U(\hat{\Z})$-invariant, and the characters $\psi_{\m_1}$ and $\psi_{\m_2}$
are trivial on~$\hat{\Z}$.
Therefore, if we consider the map $\varphi: U(\A_{fin}) \times U(\A_{fin}) \to G(\A), (u_1,u_2) \mapsto s u_1 J \delta u_2$, we have
\begin{align*}	
I_{\delta_\sigma}(f_{fin})=
\sum_{U(\hat{\Z}) \back \left(M(\n,N) \cap Im(\varphi)\right) / U(\hat{\Z})}
&\frac{\overline{\omega_N}(a_{22})}{Vol(\overline{\Gamma_1(N)})}\\
& \times
e\left(-\frac{\m_{11}c_{12}+\m_{21}c_{21}}{c_{22}}+\frac{\m_{12}\det(\Delta_1)+\m_{22}\det(\Delta_2)}{\det(C)}\right).
\end{align*}
Now by Lemma~\ref{uoppu}, $Im(\varphi)=C_J(s,d_1s^2,s^2d_1d_2)$.
Therefore, $$U(\hat{\Z}) \back  (M(\n,N) \cap Im(\varphi)) / U(\hat{\Z})$$ may be identified to $U(\Z) \back \Gamma_J(N,s,d_1s^2,s^2d_1d_2) / U(\Z)$.
\end{proof}

\subsubsection{Contribution from \texorpdfstring{$\sigma=s_1s_2s_1$}{s1s2s1}}
\begin{definition}
	Let $s,d,m$ be three non-zero integers and define  
	$$C_{121}(s,d,m)=  \{g=\mat{A}{B}{C}{D}:	\det(C)=0, c_{22}=-s, \det(\Delta_2)=d, \mu(g)=m\}$$
	and 
	$$\Gamma_{121}(N,s,d,m)= \Gamma_1(N) \cap \Mat_4(\Z) \cap C_{121}(s,d,m)$$
	For $g=\mat{A}{B}{C}{D} \in \GSp_4$, let $\Delta_3=\mat{a_{12}}{b_{11}}{c_{22}}{d_{21}}$.
	Then we define the following generalized twisted  Kloosterman sum
	\begin{align*}
		\Kloos_{121}(\m_1,\m_2,s,d,m)&=\\
		\sum_{g \in U(\Z) \back \Gamma_{121}(N,s,d,m) / \overline{U_\sigma}(\Z)}&\overline{\omega_N}(a_{22})
		e\left(\frac{\m_{11}c_{12}+\m_{21}c_{21}}{s}+\frac{\m_{12}\det(\Delta_3)}{d}\right).
	\end{align*}
\end{definition}

By a similar argument as in the case of the long Weyl element, $\Kloos_{121}(\m_1,\m_2,s,d,m)$ is well-defined, and together  with the condition on $\delta$, 
from Lemma~\ref{relevantorbits} we get the following.
\begin{proposition}
	Let $\sigma=s_1s_2s_1$, $\delta_1=\diag{d_1}{1}{d_2}{d_1d_2}$
	with $d_1d_2=\pm \frac{\n}{s^2}$ for some integer $s$ and $d_1\m_{1,2}=d_2\m_{2,2}$.
	Then we have $I_{\delta_\sigma}(f_{fin})=\frac{1}{Vol(\overline{\Gamma_1(N)})}\Kloos_{121}(\m_1,\m_2,s,d_2s^2,s^2d_1d_2)$.
\end{proposition}
\begin{remark}
	The set $\Gamma_{121}(\n,N,s,d_2s^2,s^2d_1d_2)$ is non-empty only if $N$ divides $s$ and $N^2$ divides $d_2s^2$.
\end{remark}
\subsubsection{Contribution from \texorpdfstring{$\sigma=s_2s_1s_2$}{s2s1s2}}
\begin{definition}
	Let $s,d$ be three non-zero integers and define 
	$$C_{212}(s,d,m)= \{g=\mat{A}{B}{C}{D} : \det(C)=-d, c_{22}=0, c_{21}=-s, \mu(g)=m \}$$
	and
	$$\Gamma_{212}(N,s,d,m)=\Mat_4(\Z) \cap \Gamma_1(N) \cap C_{212}.$$
	We define the following generalized twisted Kloosterman sum
	\begin{align*}
		\Kloos_{212}(\m_1,\m_2,s,d,m)&=\\
		\sum_{g \in U(\Z) \back \Gamma_{212}(N,s,d,m) / \overline{U_\sigma}(\Z)}&\overline{\omega_N}(a_{22})
		e\left(\frac{\m_{11}c_{11}+\m_{22}d_{21}}{s}-\frac{\m_{12}\det(\Delta_1)}{d}\right).
	\end{align*}
\end{definition}
By a similar argument as above, $\Kloos_{212}(\m_1,\m_2,d,s)$ is well defined, and we have the following.
\begin{proposition}
	Let $\sigma=s_1s_2s_1$, $\delta_1=\diag{d_1}{1}{d_2}{d_1d_2}$
	with $d_1d_2=\pm\frac{\n}{s^2}$ for some integer $s$ and $\m_{1,1}=-d_1\m_{2,1}$.
	Then we have $I_{\delta_\sigma}(f_{fin})=\frac{1}{Vol(\overline{\Gamma_1(N)})}\Kloos_{212}(\m_1,\m_2,sd_1,d_1s^2,sd_1d_2)$.
\end{proposition}
\begin{remark}
	The set $\Gamma_{212}(\n,N,d_1s^2,d_1s)$ is non-empty only if $N$ divides $d_1s$ and $N^2$ divides $d_1s^2$.
\end{remark}

\section{The final formula}

We now assemble the material from previous sections to obtain our relative trace formula.
Let $N \ge 1$ be an integer.
We define the {\bf adelic congruence subgroup} $\Gamma_1(N)$ to be matrices of the form $g_\infty g_{fin}$
where $g_\infty \in K_\infty$ and $g_{fin} \in \{g \in G(\hat{\Z}): g \equiv \left[
\begin{smallmatrix}
	* 	& 	& *	& *	\\
	*	& 1	& *	& *	\\
	&	& *	& *	\\
	&	& 	& *
\end{smallmatrix}\right] \mod N \}.$
Fix a character $\omega: \Q^\*\R^\* \back \A^\* \to \C$, that we may see as a character of the centre of $G(\A)$.
Assume that $\omega$ is trivial on $(1+N\hat{\Z}) \cap \hat{\Z}^\*$, and define $\omega_N(t)=\prod_{p \mid N}\omega(t_p)$.
For each standard parabolic subgroup  $P=N_PM_P$ (including $G$ itself), consider the space $\H_P$ defined in Section~\ref{DefOfES}.
For each character $\chi$ of the centre of $M_P$ whose restriction to the centre of $G$ coincides with $\omega$, 
let $\Gen_P(\chi)$ be an orthonormal basis consisting of factorizable vectors of the subspaces of functions $\phi$ in $\H_P$ 
that are generic, $\Gamma_1(N)$-fixed, and have central character $\chi$.
Specifically,
\begin{itemize}
	\item If $P=G$ then $\Gen_P(\omega)$ consists in cuspidal elements of $L^2(Z(\R)G(\Q)\back G(\A), \omega)^{\Gamma_1(N)}$,
	\item If $P=B$, each such character $\chi$ may be identified with a triplet  of characters $(\omega_1,\omega_2,\omega_3)$
	satisfying $\omega_1 \omega_2 \omega_3^2=\omega$. 
	Choose a set of representatives $S=\{k_1, \cdots, k_d\}$ of $(K \cap B(\A)) \back K / \Gamma_1(N)$. 
	Then there is a basis $(e_i)_{1 \le i \le d}$
	of $\C^S$ such that functions in $\Gen_P(\omega_1,\omega_2,\omega_3)$ are of the form 
	$$\phi_j^{\text B}(b k_i \gamma) = \chi(b)e_j(k_i)$$ for $b \in B(\A)$, $\gamma \in \Gamma_1(N)$.
	\item If $P=\Pk$, each such character $\chi$ may be identified with a pair  of characters $(\omega_1,\omega_2)$
	satisfying $\omega_1\omega_2=\omega$. 
	Choose a set of representatives $S=\{k_1, \cdots, k_d\}$ of $(K \cap \Pk(\A)) \back K / \Gamma_1(N)$. 
	Keeping notations of~\S~\ref{KES}, for each $i$, consider the compact subgroup of $\GL_2$ given by 
	$C_i=\sigma_{\text K}\left(Stab_{K \cap \Pk(\A)} (k_i)\right)$.
	Then, for each cuspidal automorphic representation $\pi$ of $\GL_2$ with central character $\omega_1$
	and whose Archimedean component is a principal series there is a basis $(u_j)_j=(u_{j,i})_{i,j}$ of $\prod_i \pi^{C_i}$
	such that functions in $\Gen_P(\omega_1,\omega_2)$ are of the form 
	$$\phi_{\pi,j}^{\text K}(p k_i \gamma) = \omega_2(p)u_{j,i}(\sigma_{\text K}(p))$$ for $p \in \Pk(\A)$, $\gamma \in \Gamma_1(N)$.
	In particular each $u_{i,j}$ is a $\GL_2$ adelic Maa{\ss} forms.
	\item If $P=\Ps$, each such character $\chi$ may be identified with a pair  of characters $(\omega_1,\omega_2)$
	satisfying $\omega_1\omega_2^2=\omega$. 
	Choose a set of representatives $S=\{k_1, \cdots, k_d\}$ of $(K \cap \Ps(\A)) \back K / \Gamma_1(N)$. 
	Keeping notations of~\S~\ref{PES}, for each $i$, consider the compact subgroup of $\GL_2$ given by 
	$C_i=\sigma_{\text K}\left(Stab_{K \cap \Ps(\A)} (k_i)\right)$.
	Then, for each cuspidal automorphic representation $\pi$ of $\GL_2$ with central character $\omega_1$
	and whose Archimedean component is a principal series there is a basis $(u_j)_j=(u_{j,i})_{i,j}$ of $\prod_i \pi^{C_i}$
	such that functions in $\Gen_P(\omega_1,\omega_2)$ are of the form 
	$$\phi_{\pi,j}^{\text S}(p k_i \gamma) = \omega_2 \circ \mu (p)u_{j,i}(\sigma_{\text K}(p))$$ for $p \in \Ps(\A)$, $\gamma \in \Gamma_1(N)$.
	In particular each $u_{i,j}$ is a $\GL_2$ adelic Maa{\ss} forms.
\end{itemize}

Now fix an integer $\n>0$ coprime to $N$. Consider 	
$$M(\n,N)=\left\{g \in G(\A_{fin}) \cap \Mat_4(\hat{\Z}) : g \equiv \left[
\begin{smallmatrix}
	* 	& 	& *	& *	\\
	*	& 1	& *	& *	\\
		&	& *	& *	\\
		&	& 	& *
\end{smallmatrix}\right] \mod N, \mu(g) \in \n  \hat{\Z}^\* \right\}.$$
Define the {\bf $\n$-th Hecke operator of level $\Gamma_1(N)$} by
$$T_{\n}\phi(g) =\int_{M(\n,N)}  \phi(gx) dx.$$
Then for every standard parabolic subgroup $P$, for every element $u \in \Gen_P$ and for every $\nu \in i\Lie{a}_P^*$, the Eisenstein series $E(\cdot,u,\nu)$
is an eigenfunction of $T_{\n}$. We shall denote the corresponding eigenvalue by $\lambda_{\n}(u,\nu)$.
Then we have the following.

\begin{theorem}\label{MainTheorem}
	Let $\m_1$, $\m_2$ be two pairs of non-zero integers, $t_1,t_2 \in A^0(\R)$.
	Let $h$ be a Paley-Wiener function on $\Lie{a}_\C$ and let $c$ be the constant appearing in Theorem~\ref{sphericalinversion}.
	Then we have $$c(\Sigma_{cusp}+\Sigma_{B}+\Sigma_K+\Sigma_S)=\frac1{Vol(\overline{\Gamma_1(N)})}(K_1+K_{121}+K_{212}+K_J).$$
	The expression $\Sigma_{cusp}+\Sigma_{B}+\Sigma_K+\Sigma_S$ is given by
	$$\Sigma_{\text{cusp}}=\sum_{u \in \Gen_G(\omega)}
	h(\nu_u)\lambda_{\n}(u)\W_{\psi}(u)(t_1t_{\m_1}^{-1})\overline{\W_{\psi}(u)}(t_2t_{\m_2}^{-1}),$$
	\begin{align*}
		\Sigma_{\text{B}}=\frac18 \sum_{\omega_1\omega_2\omega_3^2=\omega}\sum_{u \in \Gen_B(\omega_1,\omega_2,\omega_3)}& 
		\int_{i\Lie{a}^*}h(\nu)\lambda_{\n}(u,\nu)\\
		&\times \W_{\psi}(E(\cdot, u,\nu ))(t_1t_{\m_1}^{-1})\overline{\W_{\psi}(E(\cdot,u,\nu))}(t_2t_{\m_2}^{-1})d\nu,
	\end{align*}
	\begin{align*}
		\Sigma_{\text{K}}=\frac12 \sum_{\omega_1\omega_2=\omega}\sum_{u \in \Gen_{\Pk}(\omega_1,\omega_2)}&
		\int_{i\Lie{a}_\text{K}^*}h(\nu+\nu_{\text{K}}(s_u))\lambda_{\n}(u,\nu)\\
	&\times	\W_{\psi}(E(\cdot, u,\nu ))(t_1t_{\m_1}^{-1})\overline{\W_{\psi}(E(\cdot,u,\nu))}(t_2t_{\m_2}^{-1})d\nu,
	\end{align*}
\begin{align*}
	\Sigma_{\text{S}}=\frac12 \sum_{\omega_1\omega_2^2=\omega}\sum_{u \in \Gen_{\Ps}(\omega_1,\omega_2)} &
	\int_{i\Lie{a}_\text{S}^*}h(\nu+\nu_{\text{S}}(s_u))\lambda_{\n}(u,\nu)\\
	& \times \W_{\psi}(E(\cdot, u,\nu ))(t_1t_{\m_1}^{-1})\overline{\W_{\psi}(E(\cdot,u,\nu))}(t_2t_{\m_2}^{-1})d\nu,
\end{align*}
	where $\nu_u$ (resp. $s_u$) is the spectral parameter of the representation of $GSp_4(\R)$ (resp. $GL_2(\R)$) attached to $u$,
	$\nu_{\text{K}}$ and $\nu_\text{S}$ are given by Propositions~\ref{KlingenSpectralParameter} and~\ref{SiegelSpectralParameter}.
	On the right hand side, 
	\begin{itemize}
		\item $K_1$ is non-zero only if  there is an integer $t$ dividing $\n$
		with $\frac{\n}t=\frac{\m_{1,2}}{\m_{2,2}}t$ and such that $s=\frac{\m_{2,1}}{\m_{1,1}}t$ is also an integer dividing $\n$, in which case,
		setting $d=\gcd(s,\frac{\n}s,t,\frac{\n}t)$ and 
		\begin{align*}
			T(\n,\m_1,\m_2)=&d \times \overline{\omega_N(s)} \n^{-\frac12} (\m_{1,1}\m_{2,1})^{-2}|\m_{1,2}\m_{2,2}|^{-\frac32}
			\times S\left(\m_{1,1}\frac{\n}{\gcd(t,\frac{\n}s)},\m_{1,2}t,d,\n\right)
		\end{align*} we have 
		$$K_1=T(\n,\m_1,\m_2)
		\int_{\Lie{a}^*}h(-i\nu)W(i\nu, t_{\m_1}^{-1}t_1, \psi_{\one})W(-i\nu,t_{\m_2}^{-1}  t_2,\overline{\psi_{\one})}
		\frac{d\nu}{c(i\nu)c(-i\nu)}.$$
		\item The contribution of the long Weyl element is $$K_J=\sum_{\substack{N \mid s\\N^2 \mid k}}\Kloos_J(\m_1,\m_2,s,k,\n)I_J(h)\left(\frac{k}{s^2},\frac{\n}k\right),$$
		\item The contribution of $s_1s_2s_1$ is non-zero only if $\n\frac{\m_{1,2}}{\m_{2,2}}=b^2$ for some rational number $b$, in which case it is given by
		$$K_{121}=\m_{2,2}\sum_{N \mid kb}\Kloos_{121}(\m_1,\m_2,Nk,bNk,\n)I_{121}(h)\left(\frac{\n}{Nkb},\frac{b}{Nk}\right),$$
		\item The contribution of $s_2s_1s_2$ is given by
		$$K_{212}=\m_{2,1}\sum_{\substack{\m_{2,1}N \mid s\m_{1,1} \\ \m_{2,1}N^2 \mid s^2\m_{1,1}}}\Kloos_{212}\left(\m_1,\m_2,-s\frac{\m_{1,1}}{\m_{2,1}},-s^2\frac{\m_{1,1}}{\m_{2,1}},\n\right)I_{212}(h)\left(-\frac{\m_{1,1}}{\m_{2,1}},-\frac{\n}{s^2}\frac{\m_{2,1}}{\m_{1,1}}\right),$$
	\end{itemize}
where we have defined $I_\sigma(h)$
as the integral
\begin{align*}
	I_\sigma(h)(d_1,d_2)&=\int_{\Lie{a}^*}h(-i\nu)W(i\nu, t_{\m_1}^{-1}t_1, \psi_{\one})\\
	&\times W\left(-i\nu,t_{\m_1}^{-1}\sigma \diag{d_1}{1}{d_2}{d_1d_2}t_{\m_2}u_1 t_{\m_2}^{-1} t_2,\overline{\psi_{\one}}\right)
\frac{d\nu}{c(i\nu)c(-i\nu)}
{\overline{\psi_{\one}(u_1)}}du_1.
\end{align*}
	Moreover, if Conjecture~1 is true then we have
\begin{align*}
	I_\sigma(h)(d_1,d_2)=\int_{\Lie{a^*}}h(-i\nu)
	&K_\sigma\left(-i\nu,t_{\m_1}^{-1}\sigma \diag{d_1}{1}{d_2}{d_1d_2} t_{\m_2} \sigma^{-1},{\psi_{\one}}\right)\\
	&\times W(-i\nu, t_{\m_2}^{-1} t_2,{\psi_{\one}})
	W(i\nu, t_{\m_1}^{-1}t_1, \psi_{\one})\frac{d\nu}{c(i\nu)c(-i\nu)},
\end{align*}
	where the generalised Bessel functions $K_\sigma$ have been defined in \S~\ref{ArchInt}.
\end{theorem}

\end{document}